\documentclass[11pt,a4paper]{article}
\usepackage{}
\usepackage{amsfonts}
\usepackage{mathrsfs}
\usepackage{multirow}
\usepackage{makecell,booktabs}
\usepackage{subfigure}
 \usepackage{epsfig}
 \usepackage{epstopdf}
\usepackage[T1]{fontenc}
\usepackage{geometry}
\usepackage{amsbsy,amsmath,latexsym,amsfonts, epsfig, color, authblk, amssymb, graphics, bm, caption}
\usepackage{algorithm}
\usepackage{algorithmic}
\usepackage{cases}
\usepackage{rotating}
\usepackage[colorlinks, citecolor=blue]{hyperref}
\graphicspath{{Fig/}}

\newtheorem{theorem}{Theorem}[section]
\newtheorem{lemma}{Lemma}[section]

\newtheorem{definition}{Definition}[section]
\newtheorem{example}{Example}[section]
\newtheorem{proposition}{Proposition}[section]
\newtheorem{corollary}{Corollary}[section]
\newtheorem{assumption}{Assumption}

\newtheorem{remark}{Remark}[section]

\newtheorem{aexample}{Example}

\newenvironment{proof}{{\noindent \bf Proof:}}{\hfill$\Box$\medskip}

\newcommand{\qed}{\hfill$\Box$\medskip}

\definecolor{lred}{rgb}{1,0.8,0.8}
\definecolor{lblue}{rgb}{0.8,0.8,1}
\definecolor{dred}{rgb}{0.6,0,0}
\definecolor{dblue}{rgb}{0,0,0.5}
\definecolor{dgreen}{rgb}{0,0.5,0.5}

\title{An inexact LPA for DC composite optimization and application to matrix completions with outliers}

\author{Ting Tao\footnote{(\href{mailto:taoting@fosu.edu.cn}{taoting@fosu.edu.cn}) School of Mathematics, Foshan University, Foshan },\ \	
 Ruyu Liu\footnote{(\href{mailto:maruyuliu@mail.scut.edu.cn}{maruyuliu@mail.scut.edu.cn}) School of Mathematics, South China University of Technology},\ \ {\rm and}\ \	 	
 Shaohua Pan\footnote{(\href{mailto:shhpan@scut.edu.cn}{shhpan@scut.edu.cn}) School of Mathematics, South China University of Technology, Guangzhou}}

\begin{document}

 \maketitle

\begin{abstract}This paper concerns a class of DC composite optimization problems which, as an extension of convex composite optimization problems and DC programs with nonsmooth components, often arises in robust factorization models of low-rank matrix recovery. For this class of nonconvex and nonsmooth problems, we propose an inexact linearized proximal algorithm (iLPA) by computing at each step an inexact minimizer of a strongly convex majorization constructed with a partial linearization of their objective functions at the current iterate. We establish the full convergence of the generated iterate sequence under the Kurdyka-\L\"ojasiewicz (KL) property of a potential function, and employ the composite structure to provide a verifiable condition for the potential function to satisfy the KL property of exponent $1/2$ at the limit point, so for the iterate sequence to have a local R-linear convergence rate. This condition is weaker than the one provided in \cite[Theorem 3.2]{LiPong18} for identifying the KL property of exponent $p\in[0,1)$ for a general composite function. The proposed iLPA is applied to a robust factorization model for matrix completion with outliers and non-uniform sampling, and numerical comparisons with the Polyak subgradient method and a proximal alternating minimization (PAM) method validate its efficiency.
\end{abstract}


\maketitle

\medskip
\section{Introduction}\label{sec1}

The class of problems of minimizing the composition of well-structured outer functions (such as convex functions, piecewise linear-quadratic functions, DC functions) with smooth mappings plays a crucial role in numerical optimization. This class of problems provides a unified framework for studying the theory of many important classes of optimization problems such as amenable optimization, convex conic optimization and convex inclusions (see \cite{BS00,RW98}), and has extensive applications in many fields such as machine learning, statistics, signal and image processing, data science,  telecommunication, and so on  (see, e.g., \cite{LeTai05,LeTai18,Sra11}). The class of convex composite optimization problems, as reviewed later, has attracted the study of many researchers in the past few decades. In this paper, we are interested in much harder and challenging DC composite optimization problem of the form
\begin{equation}\label{prob}
 \min_{x\in\mathbb{X}}\,\Phi(x)\!:=\vartheta_1(F(x))-\vartheta_2(G(x))+h(x),
\end{equation}
where $F\!:\mathbb{X}\to\!\mathbb{Y},G\!:\mathbb{X}\to\mathbb{Z}$ and $\vartheta_1\!:\mathbb{Y}\to\mathbb{R},\vartheta_2\!:\mathbb{Z}\to\mathbb{R},h\!:\mathbb{X}\to\overline{\mathbb{R}}$ satisfy the following assumption:
\begin{assumption}\label{ass0}
\begin{enumerate}
 \item[(i)] $F$ and $G$ are differentiable on an open set $\mathcal{O}\supset{\rm dom}\,h$ (the domain of $h$), and their differential mappings $F'$ and $G'$ are strictly continuous on $\mathcal{O}$;

 \item[(ii)] $\vartheta_1$ and $\vartheta_2$ are convex functions, and the conjugate $\vartheta_1^*$ of $\vartheta_1$ is continuous relative to ${\rm dom}\,\vartheta_1^*$; 
 
 \item[(iii)] $h$ is a proper lower semicontinuous (lsc) convex function that is continuous relative to ${\rm dom}\,h\ne\emptyset$, and its conjugate $h^*$ is continuous relative to ${\rm dom}\,h^*$;   
  
 \item[(iv)] the function $\Phi$ is bounded from below, i.e., $\inf_{x\in\mathbb{X}}\Phi(x)>-\infty$.
\end{enumerate}
\end{assumption}

The nonconvex and nonsmooth problem \eqref{prob}, as shown by Examples \ref{exam1}-\ref{exam3} below, not only covers exact penalty problems of DC programs with nonconvex constraints, but also has important applications in machine learning and image processing. Though the first two terms of $\Phi$ can be compactly written as $\vartheta(H(\cdot))$ for a DC function $\vartheta$ and a smooth mapping $H$, we keep the current form for its clear structure.
 
 \begin{example}\label{exam1}
 Let $\Omega\subset\mathbb{X}$ and $K\subset\mathbb{Y}$ be the simple closed convex sets, and let $f_1,f_2:\mathbb{X}\to\mathbb{R}$ be the convex functions. Consider the following DC program with nonconvex constraints 
 \begin{equation*}
 \min_{x\in\Omega}\big\{f_1(x)-f_2(x)\ \ {\rm s.t.}\ \ g(x)\in K\big\},
 \end{equation*}
 where $g:\mathbb{X}\to\mathbb{Y}$ is a twice continuously differentiable mapping. A common exact penalty for it is 
 \[
   \min_{x\in\mathbb{X}}f_1(x)+\beta{\rm dist}(g(x),K)-f_2(x)+\chi_{\Omega}(x),
 \]
 where $\beta>0$ is the penalty parameter, ${\rm dist}(\cdot,K)$ denotes a distance function induced by a norm on $\mathbb{Y}$, and $\chi_{\Omega}$ represents the indicator function of the set $\Omega$. Obviously, the penalty problem has the form \eqref{prob} with  $\vartheta_1(y^1,y^2)=f_1(y^1)+\beta{\rm dist}(y^2,K)$ for $(y^1,y^2)\in\mathbb{X}\times\mathbb{Y}$, $ F(x)=(x;g(x))$ for $x\in\mathbb{X}$, $\vartheta_2(z)=f_2(z)$ for $z\in\mathbb{X}$, $G(x)=x$ and $h(x)=\chi_{\Omega}(x)$ for $x\in\mathbb{X}$. 
 \end{example}
 \begin{example}\label{exam2}
 Many image tasks can be modelled as the nonconvex and nonsmooth problem (see \cite{Jonas18})
 \begin{equation*}\label{image-model}
 \min_{u\in[a,b]}E(u):=T_{w}(Ag(u))+\|Du\|_1,
 \end{equation*}
 where $[a,b]$ is a box set in $\mathbb{R}^n$, $T_{w}(z)=\sum_{i=1}^m\min\{|z_i-w_i|^p,\gamma\}$ with given $p\ge 1,\gamma>0$ and $w\in\mathbb{R}^m$, $A\in\mathbb{R}^{m\times n}$ is a matrix, $g\!:\mathbb{R}^n\to\mathbb{R}^n$ is a mapping defined by $g(u):=(e^{u_1},\ldots,e^{u_n})^{\top}$, and $D\!:\mathbb{R}^n\to\mathbb{R}^q$ denotes a finite-difference gradient operator. Note that $T_{w}(z)=\sum_{i=1}^m|z_i-w_i|^p-\sum_{i=1}^m\max\{|z_i-w_i|^p-\gamma,0\}$. This model takes the form of \eqref{prob} with $F(u)=(Ag(u);D(u)),G(u)=g(u)$ and $h(u)=\chi_{[a,b]}(u)$ for $u\in\mathbb{R}^n$, and $\vartheta_1(z,y)=\sum_{i=1}^m|z_i-w_i|^p+\|y\|_1$ and $\vartheta_2(z)=\sum_{i=1}^m\max\{|z_i-w_i|^p-\gamma,0\}$ for $(z,y)\in\mathbb{R}^m\times\mathbb{R}^q$.
 \end{example}
 \begin{example}\label{exam3}
 In machine learning, the robust factorized model of low-rank matrix recovery is given by
\begin{align}\label{SCAD-loss}
 \min_{U\in\mathbb{R}^{n_1\times r},V\in\mathbb{R}^{n_2\times r}}\Phi(U,V):=\vartheta(\mathcal{A}(UV^{\top})\!-b)+\lambda(\|U\|_{2,1}+\|V\|_{2,1}),
\end{align}
 where $\|\cdot\|_{2,1}$ denotes the column $\ell_{2,1}$-norm of matrices,  $\lambda(\|U\|_{2,1}+\|V\|_{2,1})$ is a regularizer used to reduce the rank via column sparsity,  $\mathcal{A}\!:\mathbb{R}^{n_1\times n_2}\to\mathbb{R}^{m}$ is a sampling operator, $b\in\mathbb{R}^{m}$ is an observation vector, and $\vartheta\!:\mathbb{R}^m\to\mathbb{R}$ is a DC function to promote sparsity. Such $\vartheta$ includes the popular SCAD, MCP and capped $\ell_1$-norm functions, which are shown to be the equivalent DC surrogates of the zero-norm function \cite{ZhangPan22}, and the associated DC loss is more robust against outliers and heavy-tailed noise. By letting $\vartheta=\vartheta_1-\vartheta_2$, this problem takes the form of \eqref{prob} with $F(x)=G(x)=\mathcal{A}(UV^{\top})-b$ and $h(x)=\lambda(\|U\|_{2,1}+\|V\|_{2,1})$ for $x=(U,V)\in\mathbb{R}^{n_1\times r}\times\mathbb{R}^{n_2\times r}$. When choosing $\vartheta$ to be the SCAD function, we have $\vartheta_1(y)=\|y\|_1$ and $\vartheta_2(y)=\frac{1}{\rho}\sum_{i=1}^{m}\theta_{a}(\rho|y_i|)$ for $y\in\mathbb{R}^m$ with $\rho>0,a>1$ and $\theta_{a}$ defined by
 \begin{equation}\label{theta-a}
  \theta_{a}(s):=\left\{\begin{array}{cl}
	0 & {\rm if}\ s\leq \frac{2}{a+1},\\
	\frac{((a+1)s-2)^2}{4(a^2-1)} & {\rm if}\ \frac{2}{a+1}<s\leq \frac{2a}{a+1},\\
	s-1 & {\rm if}\ s>\frac{2a}{a+1}.
 \end{array}\right.
 \end{equation}
 \end{example}
\subsection{Related works}\label{sec1.1} 

Model \eqref{prob} with $\vartheta_2\!\equiv 0$ and $h\!\equiv 0$ becomes the convex composite optimization problem \cite{Burke95,Duchi19,Fletcher82}. For this class of problems, the Gauss-Newton method is a classical one for which the global quadratic convergence of the iterate sequence was achieved in \cite{Burke95} by assuming that $C:=\mathop{\arg\min}_{y\in\mathbb{Y}}\vartheta_1(y)$ is a set of weak sharp minima of $\vartheta_1$ and the cluster point is a regular point of inclusion $F(x)\in C$, and a similar convergence result was got under weaker conditions in \cite{Li07}. Another popular one, allowing $\vartheta_1$ to be extended real-valued and prox-regular, is the linearized proximal algorithm (LPA) proposed in \cite{Lewis16}. Each iteration of this method first performs a trial step by seeking a local optimal solution of a proximal linearized subproblem (that becomes strongly convex if $\vartheta_1$ is convex), and then derives a new iterate from the trial step by an efficient projection and/or other enhancements. Criticality of accumulation points under prox-regularity and identification under partial smoothness was studied in \cite{Lewis16}. Later, Hu et al. \cite{HuYang16} proposed a globalized LPA by using a backtracking line search, and obtained the global superlinear convergence of order ${2}/{p}$ with $p\in[1,2)$ for the iterate sequence by assuming that a cluster point $\overline{x}$ is a regular point of inclusion $F(x)\in C$ where $C$ is the set of local weak sharp minima of order $p$ for $\vartheta_1$ at $F(\overline{x})$; Pauwels \cite{Pauwels16} proved the full convergence of the iterate sequence for the LPA with a backtracking search for the proximal parameters under the definability of $F$ and $\vartheta_1$ in the same o-minimal structure of the real field and the twice continuous differentiability of $F$. For the iteration complexity analysis of the LPA, the reader is referred to \cite{Cartis11,Drusvyatskiy19,ZhengMa24}. In addition, the subgradient method was also studied for this class of composite problems \cite{Charisopoulos21,Charisopoulos2019,Davis18}.
For the standard NLP, i.e., model \eqref{prob} with $\vartheta_1=\chi_{\mathbb{R}_{-}^m}$ and $\vartheta_2\!\equiv 0$, Botle and Pauwels \cite{Bolte16} also proved that the iterate sequences of the moving balls method \cite{Auslender10}, the penalized SQP method and the extended SQP method converge to a KKT point if all components of $F$ are semialgebraic and the (generalized) MFCQ holds. 

 Almost all of the aforementioned methods require solving a convex or strongly convex program exactly at each iteration, which is impossible in practical computation. A practical inexact algorithm was proposed in \cite{HuYang16}, but its convergence analysis restricts a starting point in a neighborhood of a quasi-regular point of the inclusion $F(x)\!\in C$, which does not necessarily exist. This implies that, even for convex composite problems, it is necessary to develop a globally convergent and practical algorithm.

Problem \eqref{prob} with $F\equiv\mathcal{I}\equiv G$ reduces to a standard DC program with nonsmoooth components. For this class of problems, a well-known method is the DC algorithm (DCA) of \cite{LeTai18,Pham97}, which in each step linearizes the second DC component to yield a convex subproblem and uses its exact solution to define a new iterate; another popular one is the proximal linearized method (PLM) of \cite{Nguyen17,Pang17,Souza16,Sun03} that can be viewed as a regularized variant of DCA because the convex subproblems are augmented with a proximal term to prevent tailing-off effect that makes calculations unstable as the iteration progresses. For the DCA, Le Thi et al. \cite{LeTai18Jota} proved the convergence of the iterate sequence by assuming that the objective function is subanalytic and continuous relative to its domain, and either of DC components is L-smooth around every critical point; for the PLM, Nguyen et al. \cite{Nguyen17} achieved the same convergence result under the KL property of the objective function and the L-smoothness of the second DC component, or under the strong KL property of the objective function and the L-smoothness of the first DC component. To accelerate the DCA, Artacho et al. \cite{Artacho20} proposed a boosted DC algorithm (BDCA) with monotone line search by requiring the second DC component to be differentiable, and proved the convergence of the iterate sequence by assuming that the objective function has the strong KL property at critical points and the gradient of the second DC component is strictly continuous around the critical points. For the problem \eqref{prob} with $ F\equiv(f;\mathcal{I}),G\equiv\mathcal{I}$ and $\vartheta_1(t,y):=t+\theta(y)$, where $\theta:\mathbb{X}\to\mathbb{R}$ is convex and $f:\mathbb{X}\to\mathbb{R}$ is an L-smooth function, Liu et al. \cite{LiuPong19} established that the iterate sequence generated by the PLM with extrapolation is convergent under the KL property of a potential function, which removes the differentiability restriction on the second DC component in the convergence analysis of \cite{Wen18}. In addition, for the problem \eqref{prob} with $\vartheta_1(t,y)=t+\chi_{\mathbb{R}_{-}^m}(y)$, each $F_i\ (i=0,1,\ldots,m)$ being an L-smooth function and $G\equiv\mathcal{I}$, Yu et al. \cite{YuLu21} studied the monotone line search variant of the sequential convex programming (SCP) method in \cite{Lu12} by combining the idea of the moving balls method and that of the PLM, and proved that the iterate sequence converges to a stationary point of \eqref{prob} under the MFCQ and the KL property of a potential function if all $F_i\ (i=1,\ldots,m)$ are twice continuously differentiable and $\vartheta_2$ is Lipschitz continuously differentiable on an open set containing the stationary point set. Just recently, Le Thi et al. \cite{LeThi23} developed a DC composite algorithm for the problem \eqref{prob} with $\vartheta_1$ and $\vartheta_2$ allowed to be extended-valued, which extends the LPA proposed in \cite{Lewis16} to a general DC composite problem. They achieved the convergence of the objective value sequence, and proved that every accumulation point $\overline{x}$ of the iterate sequence is a stationary point of \eqref{prob} by requiring that $\vartheta_1$ is continuous relative to its domain and $\vartheta_2$ is locally Lipschitz around $G(\overline{x})$.  

Notice that most of the above DC algorithms also require solving a (strongly) convex program exactly in each step, which is impractical unless the first DC component is simple. Although an inexact PLM was proposed in \cite{Oliveira19,Souza16}, the full convergence of the iterate sequences was not obtained. For DC programs with the second DC component having a special structure, some enhanced proximal DC algorithms were proposed in \cite{Dong21,Lu19,Pang17} to seek better d-stationary points, but they are inapplicable to large-scale DC programs since at least one strongly convex program is needed to solve exactly in each step. An inexact enhanced proximal DC algorithm was also proposed in \cite{Lu19}, but the convergence of the iterate sequence was not established. Thus, even for DC programs with nonsmooth components, it is imperative to develop an inexact PLM with a full convergence certificate.

The problem \eqref{prob} cannot be reformulated as a DC program since the involved $F'$ and $G'$ are assumed to be strictly continuous on ${\rm dom}\,h$, rather than globally Lipschitz continuous. Then, the above-mentioned DCAs, PLMs and SCP cannot be directly applied to solve \eqref{prob}. In addition, the DC composite loss $\vartheta_1(F(x))-\vartheta_2(G(x))$, catering to the outliers appearing in the observation for low-rank matrix recovery, hinders the direct applications of the above inexact LPAs and moving balls methods.
\subsection{Main contributions}\label{sec1.2}

This work aims at developing a practical algorithm with a full convergence certificate for the challenging DC composite problem \eqref{prob}. Its main contributions are summarized as follows.
\begin{itemize}
\item We propose an inexact LPA with a full convergence certificate for the problem \eqref{prob}. It computes in each step an inexact minimizer of a strongly convex majorization constructed by the linearization of the inner $F$ and $G$ at the current iterate $x^k$ and the concave function $-\vartheta_2$ at $G(x^k)$. The generated iterate sequence is proved to converge to a stationary point in the sense of Definition  \ref{spoint-def} under Assumptions \ref{ass0}-\ref{ass2} and the KL property of a potential function $\Xi$ defined in \eqref{Xi-fun}. As discussed in Section \ref{sec2.1}, the stationary point in Definition \ref{spoint-def} is stronger than those obtained with the above DCAs, PLMs and SCP method, and it becomes the best d-stationary point when $\vartheta_2$ is smooth. When $\vartheta_2\equiv 0$ and $h\equiv 0$, our iLPA is an inexact version of the composite Gauss-Newton method in \cite{Pauwels16}, so the convergence results extend that of \cite{Pauwels16} via a different analysis technique. To the best of our knowledge, this is the first practical inexact LPA for DC composite problems to satisfy the full convergence certificate of iterate sequences. Though the DC composite problem in \cite{LeThi23} is more general than \eqref{prob}, their algorithm requires solving exactly a (strongly) convex program at each iteration and lacks the full convergence guarantee of the iterate sequence even for \eqref{prob}. 

\item We provide a verifiable condition for the potential function $\Xi$ to satisfy the KL property of exponent $p\in[1/2,1)$ at any critical point  $(\overline{x},\overline{x},\overline{z},\overline{\mathcal{Q}})$, by leveraging the KL property of exponent $p\in[1/2,1)$ for an almost separable nonsmooth function at $(F(\overline{x}),\overline{x},G(\overline{x}),\overline{z})$ and a condition on the subspace ${\rm Ker}([\nabla F(\overline{x})\ \ \mathcal{I}\ \ \nabla G(\overline{x})])$; see Proposition \ref{prop-KL}. This result contributes to identifying the KL property of exponent $p\in[0,1)$ for a general composite function via that of its outer nonsmooth function. As discussed in Remark \ref{remark42-KL}, the  condition in Proposition \ref{prop-KL} is weaker than the one proposed in \cite[Theorem 3.2]{LiPong18} for checking this property for a general composite function. When $\vartheta_2\equiv 0$, the discussion after Corollary \ref{corollary-KL} shows that the KL property of exponent $1/2$ for the outer almost separable function is equivalent to the local weak sharpness of order 2 for its minimum set, and the subspace condition has no direct implication relation with the quasi-regularity in \cite{HuYang16}. 

\item We develop an efficient solver (named dPPASN) for computing the subproblems of the iLPA by combining the dual proximal point algorithm (PPA) with the semismooth Newton method in Section \ref{sec5}. The proposed iLPA along with dPPASN is applied to the DC programs with nonsmooth components in Section \ref{sec6.2} and the robust factorization model \eqref{SCAD-loss} from matrix completion with outliers and non-uniform sampling in Section \ref{sec6.3}. Numerical comparison with nmBDCA, a non-monotone boosted DC algorithm \cite{Ferreria21}, on the DC program examples from \cite{Artacho20,Barkova21} indicates that our iLPA can seek more global optimal or known best solutions and the returned average objective values are better than those of nmBDCA; see Tables \ref{DCtab1}-\ref{DCtab2}. Numerical comparisons are conducted with the Polyak subgradient method \cite{Charisopoulos21,LiZhu20} and a proximal alternating minimization (PAM) method (see Appendix C for their iteration steps) for the robust factorization model \eqref{SCAD-loss}. The results show that our iLPA is more robust with respect to the regularization coefficient $\lambda$ for synthetic and real data, the returned solutions have better relative errors for synthetic data and a little worse NMAEs than those by the PAM for the real movie and netflix datasets, and it exhibits the comparable running time with the subGM and much less running time than the PAM for large-scale real data instances.
\end{itemize} 
\subsection{Notation}\label{sec1.3}
   
Throughout this paper, $\mathbb{X},\mathbb{Y}$ and $\mathbb{Z}$ denote the Euclidean vector spaces with the inner product $\langle\cdot,\cdot\rangle$ and its induced norm $\|\cdot\|$, $\overline{\mathbb{R}}\!:=(-\infty,\infty]$ denotes the extended real number set, and $\mathcal{I}$ represents an identity mapping. For a linear mapping $\mathcal{B}$, $\mathcal{B}^*$ denotes its adjoint. A linear mapping $\mathcal{Q}:\mathbb{X}\to\mathbb{X}$ is said to be positive semidefinite if it is self-adjoint and $\langle z,\mathcal{Q}z\rangle\ge 0$ for all $z\in\mathbb{X}$, and denote by $\mathbb{S}_{++}$ ($\mathbb{S}_{+}$) the set of all positive definite (semidefinite) linear mappings from $\mathbb{X}$ to $\mathbb{X}$. With a linear mapping $\mathcal{Q}\in\mathbb{S}_{+}$, let $\|z\|_{\mathcal{Q}}:=\sqrt{\langle z,\mathcal{Q}z\rangle}$ for $z\in\mathbb{X}$. For an integer $k\ge 1$, write $[k]:=\{1,\ldots,k\}$. For a mapping $H:\mathbb{X}\to\mathbb{Y}$, if it  is strictly continuous at $x$, ${\rm lip}\,H(x)$ denotes its Lipschitz modulus at $x$; if it is differentiable at $x$, $\nabla H(x)$ denotes the adjoint of $H'(x)$, the differential mapping of $H$ at $x$; if it is twice differentiable at $x$, $D^2H(x)$ denotes the twice differential mapping of $H$ at $x$, and $D^2H(x)(u,\cdot)$ for $u\in\mathbb{X}$ is a linear mapping from $\mathbb{X}$ to $\mathbb{Y}$. For a set $\Omega\subset\mathbb{Y}$, $\chi_{\Omega}$ denotes the indicator of $\Omega$. For given $x\in\mathbb{X}$ and $\varepsilon>0$, $\mathbb{B}(x,\varepsilon)$ denotes the closed ball centered at $x$ with radius $\varepsilon>0$.  For a proper $f\!:\mathbb{X}\to\overline{\mathbb{R}}$, $f^*(x^*):=\sup_{x\in\mathbb{X}}\{\langle x^*,x\rangle-f(x)\}$ denotes its conjugate; $\widehat{\partial}\!f(x)$ and $\partial\!f(x)$ denote the regular and basic (limiting) subdifferential of $f$ at $x$, respectively; a vector $x\in\mathbb{X}$ is called a critical point of $f$ if $0\in\partial\!f(x)$ and the set of its critical points is denoted as ${\rm crit}f$. When $f=\chi_{\Omega}$ for a closed set $\Omega\subset\mathbb{X}$, $\partial\!f(x)=\mathcal{N}_{\Omega}(x)$, the normal cone to $\Omega$ at $x$. In the rest of this paper, we often use the notation $\Theta:=\Theta_1-\Theta_2$ with $\Theta_1:=\vartheta_1\circ F$ and $\Theta_2:=\vartheta_2\circ G$. 

\section{Preliminaries}\label{sec2}

\subsection{Stationary points of the problem \eqref{prob}}\label{sec2.1}

Recall that $\vartheta_1$ and $\vartheta_2$ are finite convex functions. From \cite[Theorem 10.6]{RW98}, it follows that $\Theta_1$ is regular at any $x\in\mathbb{X}$ and $\partial\Theta_1(x)=\nabla\!F(x)\partial\vartheta_1(F(x))$, which along with \cite[Exercise 10.10]{RW98} implies that $\partial(\Theta_1+h)(x)=\partial\Theta_1(x)+\partial h(x)$ for all $x\in{\rm dom}\,h$. In addition, invoking \cite[Theorem 10.6]{RW98}, we have $\partial(-\Theta_2)(x)\subset\nabla\! G(x)\partial(-\vartheta_2)(G(x))$ for $x\in\mathbb{X}$. Thus, at any $x\in{\rm dom}\,\Phi$, it holds
\[
 \partial\Phi(x)\subset \nabla\!F(x)\partial\vartheta_1(F(x))+\nabla\! G(x)\partial(-\vartheta_2)(G(x))+\partial h(x).
\]
This motivates us to introduce the following definition of stationary points for the problem \eqref{prob}.

\begin{definition}\label{spoint-def}
 A vector $x\in\mathbb{X}$ is called a stationary point of the problem \eqref{prob} if
 $0\in\nabla\! F(x)\partial\vartheta_1(F(x))+\nabla\!G(x)\partial(-\vartheta_2)(G(x))+\partial h(x)$ , and we denote by $\mathcal{S}^*$ the set of stationary points of \eqref{prob}.
\end{definition}
\begin{remark}\label{remark-spoint}
 When $h\equiv 0$, the set of stationary points $\mathcal{S}^*$ is strictly contained in the following sets 
 \begin{equation}\label{spoint-relation}
 \big\{x\in\mathbb{X}\ |\ \partial(\vartheta_1\circ F)(x)\cap\partial (\vartheta_2\circ G)(x)\ne\emptyset\big\}=\big\{x\in\mathbb{X}\ |\ \widehat{\partial} (\vartheta_1\circ F)(x)\cap\widehat{\partial} (\vartheta_2\circ G)(x)\ne\emptyset\big\}.
 \end{equation}
 Indeed, from \cite[Theorem 10.6]{RW98} and the strict continuity of $\vartheta_1$ and $\vartheta_2$,
 the equality in \eqref{spoint-relation} holds and 
 \begin{equation*}
  \partial(\vartheta_1\circ F)(x)=\nabla\! F(x)\partial\vartheta_1(F(x))\ \ {\rm and}\ \  \partial(\vartheta_2\circ G)(x)=\nabla\!G(x)\partial\vartheta_2(G(x)).
 \end{equation*}
 Note that $\partial(\vartheta_1\circ F)(x)\cap\partial (\vartheta_2\circ G)(x)\ne\emptyset$ if and only if $0\in\partial(\vartheta_1\circ F)(x)-\partial(\vartheta_2\circ G)(x)$. Thus, the sets in \eqref{spoint-relation} coincide with $\overline{\mathcal{S}}^*:=\{x\in\mathbb{X}\ |\ 0\in\nabla\! F(x)\partial\vartheta_1(F(x))-\nabla\! G(x)\partial\vartheta_2(G(x))\}$. On the other hand,  by \cite[Corollary 9.21]{RW98}, the strict continuity and convexity of $\vartheta_2$ implies that   
 $\partial(-\vartheta_2)(z)\subset -\partial\vartheta_2(z)$ for any $z\in\mathbb{Z}$, and the inclusion may be strict. Together with Definition \ref{spoint-def}, we have $\mathcal{S}^*\subset\overline{S}^*$ and the inclusion is strict. Thus, the claimed inclusion holds. When $F=G$ and $h\equiv 0$, the second set in \eqref{spoint-relation} is precisely the one introduced in \cite{LeThi23}. Consequently, we conclude that the set of stationary points $\mathcal{S}^*$ in Definition \ref{spoint-def} is stronger than the one introduced in \cite{LeThi23}. When $G=\mathcal{I}$ and $h\equiv 0$, the set $\mathcal{S}^*$ is also strictly contained in the set of critical points adopted in  \cite{LeTai18Jota,LiuPong19,Nguyen17}. 
 \end{remark}

 Recall that $\vartheta_1$ and $\vartheta_2$ are strictly continuous and $F$ and $G$ are continuously differentiable, so $\Theta_1$ and $\Theta_2$ are directionally differentiable with $\Theta_1'(x;w)=\vartheta_1'(F(x);F'(x)w)$ and  $\Theta_2'(x;w)=\vartheta_2'(G(x);G'(x)w)$ at any $(x,w)\in\mathbb{X}\times\mathbb{X}$. 
 Inspired by \cite{Pang17}, we introduce the following definition of d(irectional)-stationary points.  
\begin{definition}\label{dspoint-def}
 A vector $x\in\mathbb{X}$ such that
 $\Theta_1'(x;w)-\Theta_2'(x;w)+h'(x;w)\ge 0$ for any $w\in\mathbb{X}$ is called a $d$-stationary point of the problem \eqref{prob}, and we denote by $\mathcal{D}^*$ the set of $d$-stationary points of \eqref{prob}.
\end{definition}

The following lemma implies that if $\vartheta_2$ is continuously differentiable, then it holds $\mathcal{S}^*\!\subset\mathcal{D}^*$.  
 \begin{lemma}\label{d-stationary}
 Consider any $\overline{x}\in{\rm dom}\,h$. If
$\partial\Theta_2(\overline{x})\subset\partial\Theta_1(\overline{x})+\partial h(\overline{x})$, then $\overline{x}\in\mathcal{D}^*$. 
 The converse also holds when $\overline{x}\in{\rm ri}({\rm dom}\,h)$.  
 \end{lemma}
 \begin{proof}\!\!\! .
  Suppose that $\partial\Theta_2(\overline{x})\subset\partial\Theta_1(\overline{x})+\partial h(\overline{x})$. Fix any $w\in\mathbb{X}$. As $\Theta_2$ is regular and strictly continuous on $\mathbb{X}$, from \cite[Exercise 9.15 \& Theorem 9.16]{RW98}, $d\Theta_2(\overline{x})(w)=\max_{z\in\partial \Theta_2(\overline{x})}\langle z,w\rangle=\Theta_2'(\overline{x};w)$. Then, 
  \begin{equation*}
   \Theta_2'(\overline{x};w)
   \le\sup_{z\in\partial\Theta_1(\overline{x})+\partial h(\overline{x})}\langle z,w\rangle\le\sup_{z'\in\partial\Theta_1(\overline{x})}\langle z',w\rangle +\sup_{v\in\partial h(\overline{x})}\langle v,w\rangle\le\Theta_1'(\overline{x};w)+h'(\overline{x};w).
  \end{equation*} 
  This by the arbitrariness of $w\in\mathbb{X}$ shows that $\overline{x}\in\mathcal{D}^*$. Conversely, let $\overline{x}\in\mathcal{D}^*\cap{\rm ri}({\rm dom}\,h)$. Suppose 
  on the contrary that the inclusion $\partial\Theta_2(\overline{x})\subset\partial\Theta_1(\overline{x})+\partial h(\overline{x})$ does not hold. Then, there must exist $u\in\partial\Theta_2(\overline{x})$ but $u\notin\partial\Theta_1(\overline{x})+\partial h(\overline{x})$. 
  Note that $\partial\Theta_1(\overline{x})=\nabla F(\overline{x})\partial\vartheta_1(F(\overline{x}))$ is a nonempty compact convex set, and $\partial h(\overline{x})$ is nonempty and convex due to $\overline{x}\in{\rm ri}({\rm dom}\,h)$. Hence,  $\partial(\Theta_1+h)(\overline{x})=\partial\Theta_1(\overline{x})+\partial h(\overline{x})$ is a nonempty closed convex set. This means that $\partial(\Theta_1+h)(\overline{x})$ and $\{u\}$ can be strongly separated, i.e., there exists $w\in\mathbb{X}\backslash\{0\}$ and $\varepsilon_0>0$ such that
  $\langle w,z\rangle+\varepsilon_0\le\langle w,u\rangle$ for all $z\in\partial(\Theta_1+h)(\overline{x})$. Consequently, 
  \[
   \sup_{z\in\partial\Theta_1(\overline{x})+\partial h(\overline{x})}\langle w,z\rangle+\varepsilon_0=\sup_{z\in\partial(\Theta_1+h)(\overline{x})}\langle w,z\rangle+\varepsilon_0\le \langle w,u\rangle\le d\Theta_2(\overline{x})(w)=\Theta_2'(\overline{x};w), 
  \]
  where the second inequality is due to $u\in\partial\Theta_2(\overline{x})$, the regularity of $\Theta_2$ and \cite[Exercise 8.4]{RW98}. In addition, from $\Theta_1'(\overline{x};w)=\vartheta_1'(F(\overline{x});F'(\overline{x})w)=\max_{z\in\partial\vartheta_1(F(\overline{x}))}\langle z,F'(\overline{x})w\rangle$, there must exist a point $\xi\in\partial\vartheta_1(F(\overline{x}))$ such that $\Theta_1'(\overline{x};w)=\langle\xi,F'(\overline{x})w\rangle=\langle\nabla F(\overline{x})\xi,w\rangle$. Obviously, $\nabla F(\overline{x})\xi\in\partial\Theta_1(\overline{x})$. Then, 
  \begin{align*}
  \Theta_1'(\overline{x};w)+h'(\overline{x};w)=\langle\nabla F(\overline{x})\xi,w\rangle+\sup_{z\in\partial h(\overline{x})}\langle z,w\rangle&=\sup_{z\in\partial h(\overline{x})}\langle z+\nabla F(\overline{x})\xi,w\rangle\\
   &\le\sup_{z'\in\partial\Theta_1(\overline{x})+\partial h(\overline{x})}\langle z',w\rangle.
  \end{align*}
  The above two equations imply that $\Theta_1'(\overline{x};w)+h'(\overline{x};w)-\Theta_2'(\overline{x};w)<0$, a contradiction to $\overline{x}\in\mathcal{D}^*$. 
  \end{proof}
\subsection{Metric $q$-subregularity and KL property}\label{sec2.2}

 The metric $q\ (q>0)$-subregularity of a multifunction and the Kurdyka-{\L}ojasiewicz (KL) property of a nonsmooth function play a key role in the convergence analysis of algorithms. The former was used to analyze the convergence rate of proximal point algorithm for seeking a root to a maximal monotone operator \cite{LiMor12}, and the local superlinear convergence rates of proximal Newton-type methods for convex and nonsmooth composite optimization \cite{Mordu23,LiuPanWY22}. Its formal definition is stated as follows.
\begin{definition}\label{Def-subregular}
(see \cite[Definition 3.1]{LiMor12}) A multifunction $\mathcal{F}:\mathbb{X}\rightrightarrows\mathbb{X}$ is called (metrically) $q\ (q>0)$-subregular at a point $(\overline{x},\overline{y})\in{\rm gph}\mathcal{F}$, the graph of $\mathcal{F}$, if there exist $\varepsilon>0$ and $\kappa>0$ such that 
\[
  {\rm dist}(x,\mathcal{F}^{-1}(\overline{y}))\le\!\kappa[{\rm dist}(\overline{y},\mathcal{F}(x))]^{q}\quad{\rm for\ all}\ x\in\mathbb{B}(\overline{x},\varepsilon).
 \]
 When $q=1$, this property is called the (metric) subregularity of $\mathcal{F}$ at $(\overline{x},\overline{y})$. 
\end{definition}

 Obviously, if $\mathcal{F}$ is subregular at $(\overline{x},\overline{y})\in{\rm gph}\,\mathcal{F}$, it is $q\in(0,1]$-subregular at this point. When $\mathcal{F}$ is the subdifferential mapping of a class of nonconvex composite functions, its $q\in(0,1]$-subregularity at a point $(\overline{x},0)\!\in{\rm gph}\,\mathcal{F}$ is closely related to the KL property of exponent $1/(2q)$ of this class of composite functions (see \cite[Section 2.3]{LiuPanWY22}). To introduce the KL property of an extended real-valued function, for every $\eta\in(0,\infty]$, we denote by $\Upsilon_{\eta}$ the set consisting of all continuous concave $\varphi:[0,\eta)\to\mathbb{R}_{+}$ that is continuously differentiable on $(0,\eta)$ with $\varphi(0)=0$ and $\varphi'(s)>0$ for all $s\in(0,\eta)$.
\begin{definition}\label{KL-def}
 A proper lsc function $f\!:\mathbb{X}\to\overline{\mathbb{R}}$ is said to satisfy the property at $\overline{x}\in{\rm dom}\,\partial\!f$ if there exist $\eta\in(0,\infty]$, $\varphi\in\Upsilon_{\!\eta}$ and a neighborhood $\mathcal{U}$ of $\overline{x}$ such that for all $x\in\mathcal{U}\cap\big[f(\overline{x})<f<f(\overline{x})+\eta\big]$,
 \[
  \varphi'(f(x)\!-\!f(\overline{x})){\rm dist}(0,\partial\!f(x))\ge 1.
 \]
 If $\varphi$ can be chosen as $\varphi(t)=ct^{1-p}$ with $p\in[0,1)$ for some $c>0$, then $f$ is said to satisfy the KL property of exponent $p$ at $\overline{x}$. If $f$ has the KL property (of exponent $p$) at each point of ${\rm dom}\,\partial\!f$, it is called a KL function (of exponent $p$).
\end{definition}
\begin{remark}\label{KL-remark}
 According to \cite[Lemma 2.1]{Attouch10}, a proper lsc function has the KL property of exponent $0$ at any non-critical point. Thus, to prove that a proper lsc $f\!:\mathbb{X}\to\overline{\mathbb{R}}$ is a KL function (of exponent $p$), it suffices to check its KL property (of exponent $p$) at critical points. The KL functions are extremely extensive by the discussions in \cite{Dries96} and \cite[Section 4]{Attouch09}, but it is not an easy task to verify the KL property of exponent $1/2$ except for some special class of nonsmooth functions; see \cite{LiPong18,YuLiPong21,WuPanBi21}.
\end{remark} 	
	
To close this section, we state the relation between $\Theta_i$ for $i=1,2$ and its local linearization.
\begin{lemma}\label{vtheta1-lemma}
 Consider any $x\in {\rm dom}\,h$. For any $\varepsilon>0$, there exists $\delta>0$ such that for all $z\in\mathbb{B}(x,\delta)$,
 \begin{subnumcases}{}\label{Theta1-Lip}
  \big|\Theta_1(z)\!-\!\vartheta_1\big(F(x)\!+\!F'(x)(z\!-\!x)\big)\big|
  \! \le\!(1/2)({\rm lip}\,\vartheta_1(F(x))+\varepsilon)({\rm lip}\,F'(x)+\varepsilon)\|z-x\|^2,\\
  \label{Theta2-Lip}
 \big|\Theta_2(z)\!-\!\vartheta_2\big(G(x)\!+\!G'(x)(z\!-\!x)\big)\big|
 \!\le\!(1/2)({\rm lip}\,\vartheta_2(G(x))+\varepsilon)({\rm lip}\,G'(x)+\varepsilon)\|z-x\|^2.
\end{subnumcases}
\end{lemma}
\begin{proof}\!\!\! .
 Fix any $\varepsilon>0$. By Assumption \ref{ass0} (ii), $\vartheta_1$ is strictly continuous at $F(x)$, so there exists $\varepsilon_1>0$ such that for all $y,y'\in\mathbb{B}(F(x),\varepsilon_1)$,
 \begin{equation}\label{ineq-vtheta1}
  |\vartheta_1(y)-\vartheta_1(y')|\le ({\rm lip}\,\vartheta_1(F(x))+\varepsilon)\|y-y'\|.
 \end{equation}
 By Assumption \ref{ass0} (i), there exists $\delta>0$ such that $\mathbb{B}(x,\delta)\subset\mathcal{O}$ and for all $z\in\mathbb{B}(x,\delta)$,
 $\|F(z)-F(x)\|\le\varepsilon_1$ and $\|F'(x)(z-x)\|\le\varepsilon_1$, and furthermore, the strict continuity of $F'$ on $\mathcal{O}\supset{\rm dom}\,h$ implies that
 \begin{equation}\label{Lip-Fdiff}
  \|F'(z)-F'(z')\|\le({\rm lip}\,F'(x)+\varepsilon)\|z-z'\|\quad\ \forall\, z,z'\in\mathbb{B}(x,\delta).
 \end{equation}
 Now fix any $z\in\mathbb{B}(x,\delta)$. Invoking the above \eqref{ineq-vtheta1} with $y=F(z)$ and $y'\!=\!F(x)\!+\!F'(x)(z-x)$ results in
 \begin{align*}
  |\Theta_1(z)\!-\!\vartheta_1(F(x)+F'(x)(z\!-\!x))|
  &\le ({\rm lip}\,\vartheta_1(F(x))+\varepsilon)\|F(z)-F(x)-F'(x)(z\!-\!x)\|\\
  &\!=\!({\rm lip}\,\vartheta_1(F(x))\!+\!\!\varepsilon)\Big\|\!\int_{0}^{1}\![F'(x\!+\!t(z\!-\!x))\!-\!\!F'(x)](z\!-\!x)dt\Big\|\\
  &\le\!\frac{1}{2}({\rm lip}\,\vartheta_1(F(x))+\varepsilon)({\rm lip}\,F'(x)+\varepsilon)\|z-x\|^2,
 \end{align*}	
 where the last inequality is due to \eqref{Lip-Fdiff} and $x+t(z-x)\in\mathbb{B}(x,\delta)$ for all $t\in[0,1]$. Inequality \eqref{Theta1-Lip} follows by the arbitrariness of $z\in\mathbb{B}(x,\delta)$. Using the same arguments results in inequality \eqref{Theta2-Lip}. 
\end{proof}
\section{Inexact linearized proximal algorithm}\label{sec3}

For any given $x\in\mathbb{X}$, let $\ell_{F}(\cdot,x)$ and $\ell_{G}(\cdot,x)$ be the linear approximation of $F$ and $G$ at $x$, respectively:
\begin{equation}\label{LinearFG}
 \ell_{F}(z,x):=F(x)+F'(x)(z-x)\ \ {\rm and}\ \ \ell_{G}(z,x):=G(x)+G'(x)(z-x)\quad\ \forall\,z\in\mathbb{X}.
\end{equation}
Let $x^k$ be the current iterate. Fix any $\varepsilon\!>\!0$. By invoking \eqref{Theta1-Lip} with $x\!=\!x^k$, for any $x$ close enough to $x^k$,
\begin{equation}\label{theta1-ineq0}
 \Theta_1(x)\le\vartheta_1(\ell_{F}(x,x^k))+\frac{1}{2}({\rm lip}\,\vartheta_1(F(x^k))+\varepsilon)({\rm lip}\,F'(x^k)+\varepsilon)\|x-x^k\|^2.
\end{equation}
Similarly, by invoking inequality \eqref{Theta2-Lip}, for any $x$ sufficiently close to $x^k$,
\begin{equation}\label{theta2-ineq0}
 -\Theta_2(x)\le-\vartheta_2(\ell_{G}(x,x^k))+\frac{1}{2}({\rm lip}\,\vartheta_2(G(x^k))+\varepsilon)({\rm lip}\,G'(x^k)+\varepsilon)\|x-x^k\|^2.
 \end{equation}
Pick any $\xi^k\in\partial(-\vartheta_2)(G(x^k))\subset -\partial\vartheta_2(G(x^k))$. From the convexity of $\vartheta_2$, it follows $\vartheta_2(\ell_{G}(x,x^k))\ge\Theta_2(x^k)-\langle\nabla G(x^k)\xi^k,x-x^k\rangle$ for any $x\in\mathbb{X}$. For each $k\in\mathbb{N}$, let ${\rm lip}\,\Theta(x^k):={\rm lip}\,\vartheta_1(F(x^k)){\rm lip}\,F'(x^k)+{\rm lip}\,\vartheta_2(G(x^k)){\rm lip}\,G'(x^k)$, and let $L_k$ be close enough to ${\rm lip}\,\Theta(x^k)$ from above. Along with \eqref{theta1-ineq0}-\eqref{theta2-ineq0}, 
\begin{align}\label{Theta-ineq30}
  \Theta(x)&=\Theta_1(x)-\Theta_2(x)\le \vartheta_1(\ell_{F}(x,x^k))-\vartheta_2(\ell_{G}(x,x^k))+(L_k/2)\|x-x^k\|^2\\
  &\le\vartheta_1(\ell_{F}(x,x^k))+\langle\nabla G(x^k)\xi^k,x-x^k\rangle+(L_k/2)\|x-x^k\|^2-\Theta_2(x^k)\nonumber
\end{align}
for any $x$ sufficiently close to $x^k$. Thus, by choosing a positive definite (PD) linear operator $\mathcal{Q}_k\!:\mathbb{X}\to\mathbb{X}$ with $\mathcal{Q}_k\succeq L_{k}\mathcal{I}$, we obtain the following strongly convex local majorization of $\Phi=\Theta+h$ at $x^k$:
\[
  q_{k}(x):=\vartheta_1(\ell_{F}(x,x^k))+\langle\nabla G(x^k)\xi^k,x-x^k\rangle+h(x)+\frac{1}{2}\|x-x^k\|_{\mathcal{Q}_k}^2-\Theta_2(x^k)
  \quad\ \forall x\in\mathbb{X}.
\]
At the $k$th iteration, our inexact LPA seeks an inexact minimizer of $q_k$ as the next iterate. Consider that ${\rm lip}\,\Theta(x^k)$ is generally unknown. Our method captures its upper estimation $L_k$ via backtracking at each iteration. Its  iteration steps are described as follows, where $\overline{x}^{k,j}$ is the unique optimal solution of \eqref{subprobkj}.   
\begin{algorithm}[h]
\caption{\label{iLPA}{\bf\,(Inexact linearized proximal algorithm)}}
 {\bf{1.}} Input: $\overline{\varrho}>1,\,0<\underline{\gamma}<\overline{\gamma}$, a large $\overline{\beta}>0$, a bounded positive sequence $\{\mu_k\}_{k\in\mathbb{N}}$, and 
 
 \ \ \ \ \ $x^0\in{\rm dom}\,h$. \\
{\bf{2. For}} {$k=0,1,2,\ldots$}\\
  {\bf{3.   }}   \label{step3} \   Choose $\xi^{k}\in\partial(-\vartheta_2)(G(x^k))$ and $\gamma_{k,0}\in[\underline{\gamma},\overline{\gamma}]$.
  \\
  {\bf{4.    \ \    For}}{$j=0,1,2,\ldots$}\\		
 {\bf{5.}} \label{step5} \ \ \  \ \ \   Choose a PD linear operator $\gamma_{k,j}\mathcal{I}\!\preceq\mathcal{Q}_{k,j}\!\preceq(\gamma_{k,j}+\overline{\beta})\mathcal{I}$. Seek an inexact solution 
 
 \qquad \ \ \ \ $x^{k,j}$ of
 \begin{equation}\label{subprobkj}
 \qquad\min_{x\in\mathbb{X}}\ q_{k,j}(x):=\vartheta_1(\ell_{F}(x,x^k))+\langle\nabla G(x^k)\xi^k,x-x^k\rangle +h(x)+\frac{1}{2}\|x-x^k\|_{\mathcal{Q}_{k,j}}^2-\Theta_2(x^k),
 \end{equation}
 \hspace{1.2cm} and a lower bound $q_{k,j}^{\rm LB}$ for the optimal value $q_{k,j}(\overline{x}^{k,j})$ of \eqref{subprobkj} such that 
 \begin{equation}\label{inexact-cond}
 q_{k,j}(x^{k,j})-q_{k,j}^{\rm LB}\le ({\mu_k}/{2})\|x^{k,j}-x^{k}\|^2.
 \end{equation}
			
{\bf{6. \label{step6} \ \ \  \ \ \   If }}  $\Theta(x^{k,j})\le\vartheta_1\big(\ell_{F}(x^{k,j},x^k)\big)-\vartheta_2\big(\ell_{G}(x^{k,j},x^k)\big)+\frac{1}{2}\|x^{k,j}-x^k\|_{\mathcal{Q}_{k,j}}^2$, go to step 9;

\qquad \ \ \ \ otherwise, let $\gamma_{k,j+1}=\overline{\varrho}\gamma_{k,j}$.	
 \\	
   {\bf{7. \ \   end (For)}}\\		
  {\bf{8.}} Set $j_k=j$, $x^{k+1}=x^{k,j_k},\overline{x}^{k+1}=\overline{x}^{k,j_k}$ and $\mathcal{Q}_{k}=\mathcal{Q}_{k,j_k}$. \\
 {\bf{9. end (For)}}
 \end{algorithm}
 \begin{remark}\label{remark-alg}
 
 {\bf(a)} For the sequence $\{\mu_k\}_{k\in\mathbb{N}}$, we assume that there exists $\overline{k}\in\mathbb{N}$ such that $\mu_k\in(0,{\underline{\gamma}}/{5}]$ for all $k\ge\overline{k}$. Such a restriction on $\{\mu_k\}_{k\in\mathbb{N}}$ is just for the convergence analysis of Algorithm \ref{iLPA}. In practical computation, one can choose $\{\mu_k\}_{k\in\mathbb{N}}$ to be any positive number sequence converging to $0$. 

 \noindent
 {\bf(b)} The inner for-end loop aims at seeking a tight upper estimation for ${\rm lip}\,\Theta(x^k)$ and computing an inexact minimizer $x^{k,j}$ of the associated $q_{k,j}$ and a lower bound $q_{k,j}^{\rm LB}$ for its minimum such that the condition \eqref{inexact-cond} holds. As will be shown in Lemma \ref{ls-welldef1} later, for each $k\in\mathbb{N}$ and $j\in\mathbb{N}$, such a pair $(x^{k,j},q_{k,j}^{\rm LB})$ exists and the inner loop stops within a finite number of steps, so Algorithm \ref{iLPA} is well defined. 
 	
 \noindent
 {\bf(c)} The inexactness criterion \eqref{inexact-cond} involves a lower bound $q_{k,j}^{\rm LB}$ for the optimal value of \eqref{subprobkj}. As will be shown in Section \ref{sec5}, when $\mathcal{Q}_{k,j}$ is chosen to have a special structure, the dual of \eqref{subprobkj} is an unconstrained smooth optimization problem. In this case, by the weak duality theorem, an algorithm for solving the dual problem will return such $q_{k,j}^{\rm LB}$ and $x^{k,j}$ at each iteration. This shows that our inexactness criterion is practical, though it is stronger than the one $q_{k,j}(x^{k,j})-q_{k,j}(\overline{x}^{k,j})\le (\mu_k/2)\|x^{k,j}-x^k\|^2$. The dPPASN developed in Section \ref{sec5} is an efficient dual solver to produce such $q_{k,j}^{\rm LB}$ and $x^{k,j}$.
 
  \noindent
  {\bf(d)} Algorithm \ref{iLPA} is an extension of the Gauss-Newton method \cite{Pauwels16} and the inexact LPA \cite[Algorithm 19]{HuYang16}, which are both proposed for the problem \eqref{prob} with $\vartheta_2\equiv 0$ and $h\equiv 0$. Compared with the Gauss-Newton method \cite{Pauwels16}, Algorithm \ref{iLPA} is practical because its strongly convex subproblems are allowed to be solved inexactly and the inexactness criterion $q_{k,j}(x^{k,j})-q_{k,j}^{\rm LB}\le\frac{\mu_k}{2}\|x^{k,j}-x^{k}\|^2$ is implementable according to part (c). Different from the inexact LPA of \cite{HuYang16}, our inexactness condition uses a lower bound for the optimal value of \eqref{subprobkj} and controls $q_{k,j}(x^{k,j})-q_{k,j}^{\rm LB}$ by the quadratic term $\|x^{k,j}\!-\!x^{k}\|^2$ rather than the more restrictive $\|x^{k-1}\!-\!x^{k}\|^{\alpha}$ with $\alpha>2$. In addition, Algorithm \ref{iLPA} also extends the proximal DC algorithm proposed in \cite{LiuPong19} for the problem \eqref{prob} with $F\equiv(f;\mathcal{I}),G\equiv\mathcal{I}$ and $\vartheta_1(t,y)=t+h(y)$ without requiring $\nabla\!f$ to be globally Lipschitz continuous.
	
  \noindent
  {\bf(e)} When $x^{k+1}=x^{k}$ for some $k\in\mathbb{N}$, both $x^{k+1}$ and $x^k$ are a stationary point of \eqref{prob}. Indeed, from the inexactness condition \eqref{inexact-cond} and the strong convexity of $q_{k,j}$ with modulus not less than $\underline{\gamma}$, it follows that $\underline{\gamma}\|x^{k+1}-\overline{x}^{k+1}\|^2\le\mu_k\|x^{k+1}-x^k\|^2=0$ and then $x^k=x^{k+1}=\overline{x}^{k+1}$. Together with $0\in\partial q_{k,j_k}(\overline{x}^{k+1})$, we have $0\in\partial q_{k,j_k}(x^{k})=\partial q_{k,j_k}(x^{k+1})$, which by Definition \ref{spoint-def} shows that $x^{k}$ and $x^{k+1}$ are a stationary point of \eqref{prob}. By virtue of this, we can adopt $\|x^{k+1}-x^k\|\le \epsilon_*$ for a tolerance $\epsilon_*>0$ as the stop condition. 
\end{remark}
\begin{lemma}\label{ls-welldef1}
 Fix any $k\in\mathbb{N}$ such that $x^k$ is not a stationary point of \eqref{prob}. Then, under Assumption \ref{ass0},  
 \begin{enumerate}
  \item[(i)] for each $j\in\mathbb{N}$, letting $\mathcal{Q}_{k,j}=\mathcal{B}_{k,j}^*\mathcal{B}_{k,j}$ with $\mathcal{B}_{k,j}$ being an injective linear mapping from $\mathbb{X}$ to a finite-dimensional space $\mathbb{U}$, and $D_{k,j}$ be the dual objective function of \eqref{subprobkj}, any $z\in{\rm dom}\,h$ close enough to $\overline{x}^{k,j}$ and $q_{k,j}^{\rm LB}=D_{k,j}(\zeta,y)$ for $(\zeta,y)\in(-{\rm dom}\vartheta_1^*\times\mathbb{U})\cap\mathcal{L}_{k,j}^{-1}({\rm dom}\,h^*)$ close enough to a dual optimal solution $(\overline{\zeta}^{k,j},\overline{y}^{k,j})$ together satisfy \eqref{inexact-cond}, where $\mathcal{L}_{k,j}(\zeta,y):=\nabla F(x^k)\zeta+\mathcal{B}_{k,j}^*y-b$; 

  \item[(ii)] the inner loop of Algorithm \ref{iLPA} stops within a finite number of steps.
 \end{enumerate}
\end{lemma}
\begin{proof} 
 {\bf(i)} Fix any $j\in\mathbb{N}$. To prove the conclusion, we take a closer look at the dual problem of \eqref{subprobkj}. Since $\mathcal{Q}_{k,j}=\mathcal{B}_{k,j}^*\mathcal{B}_{k,j}$, the subproblem \eqref{subprobkj} can equivalently be written as 
 \begin{align*}
  &\min_{x\in\mathbb{X},z\in\mathbb{Z},u\in\mathbb{U}}\ \vartheta_1(z)+\langle\nabla G(x^k)\xi^k,x-x^k\rangle +h(x)+(1/2)\|u\|^2-\Theta_2(x^k),\\
  &\quad\ \ {\rm s.t.}\quad F'(x^k)(x-x^k)+F(x^k)=z,\\
  &\qquad\qquad \mathcal{B}_{k,j}(x-x^k)=u.
 \end{align*}
 An elementary calculation yields the dual problem of \eqref{subprobkj} as follows
 \begin{align}\label{dual-prob}
&\max_{\zeta\in\mathbb{Z},y\in\mathbb{U}}D_{k,j}(\zeta,y)=\langle \zeta,F'(x^k)x^k-F(x^k)\rangle-\vartheta_1^*(-\zeta)-\frac{1}{2}\|y\|^2+\langle y,\mathcal{B}_{k,j}x^k\rangle\nonumber\\
  &\qquad\qquad\qquad\qquad\ -h^*(\nabla F(x^k)\zeta+\mathcal{B}_{k,j}^*y-\nabla G(x^k)\xi^k)-\langle \nabla G(x^k)\xi^k,x^k\rangle-\Theta_2(x^k).
 \end{align}
 By virtue of Assumption \ref{ass0} (ii)-(iii), the function $D_{k,j}:\mathbb{Z}\times\mathbb{U}\to\overline{\mathbb{R}}$ is continuous relative to its domain. 

 From Remark \ref{remark-alg} (a), we infer that  $x^k\ne\overline{x}^{k,j}$. If not, $0\in\partial q_{k,j}(\overline{x}^{k,j})=\partial q_{k,j}(x^k)$, which by the expression of $q_{k,j}$ and Definition \ref{spoint-def} shows that $x^k$ is a stationary point of \eqref{prob}.
 Consider 
\[ 
 \Psi(z,\zeta,y):=q_{k,j}(z)-D_{k,j}(\zeta,y)-(\mu_k/2)\|z-x^{k}\|^2\quad{\rm for}\ (z,\zeta,y)\in\mathbb{X}\times\mathbb{Z}\times\mathbb{U}. 
\]
Clearly, $\Psi$ is continuous relative to its domain ${\rm dom}\,h\times[(-{\rm dom}\vartheta_1^*\times\mathbb{U})\cap\mathcal{L}_{k,j}^{-1}({\rm dom}\,h^*)]$.   
 Note that  $\Psi(\overline{x}^{k,j},\overline{\zeta}^{k,j},\overline{y}^{k,j})<0$. Then, for any $(z,\zeta,y)\in{\rm dom}\,h\times[(-{\rm dom}\vartheta_1^*\times\mathbb{U})\cap\mathcal{L}_{k,j}^{-1}({\rm dom}\,h^*)]$ close enough to $(\overline{x}^{k,j},\overline{\zeta}^{k,j},\overline{y}^{k,j})$, it holds $q_{k,j}(z)-D_{k,j}(\zeta,y)\le(\mu_k/2)\|x-x^k\|^2$. The desired conclusion holds. 
 
 \noindent
 {\bf(ii)} Suppose on the contrary that the conclusion does not hold. Then, for sufficiently large $j\in\mathbb{N}$,
  \begin{equation}\label{aim-ineq1}
   \Theta(x^{k,j})-\vartheta_1(\ell_{F}(x^{k,j},x^k))+\vartheta_2(\ell_{G}(x^{k,j},x^k))>(1/2)\|x^{k,j}-x^k\|_{\mathcal{Q}_{k,j}}^2.
  \end{equation}
  From the definition of $x^{k,j}$ in the inner loop and the expression of $q_{k,j}$, for each $j$ large enough,
  \begin{align}\label{temp-ineq31}
  q_{k,j}^{\rm LB}+(\mu_k/2)\|x^{k,j}-x^{k}\|^2\ge q_{k,j}(x^{k,j})
  &=\vartheta_1(\ell_{F}(x^{k,j},x^{k}))\!+\!\langle\nabla G(x^k)\xi^k,x^{k,j}\!-\!x^k\rangle\nonumber\\
  &\quad+h(x^{k,j})+(1/2)\|x^{k,j}\!-\!x^k\|_{\mathcal{Q}_{k,j}}^2-\Theta_2(x^k).
  \end{align}
  For the proper and lsc convex function $h$, since ${\rm dom}\,h\ne\emptyset$, we have ${\rm ri}({\rm dom}\,h)\ne\emptyset$. Let $\widehat{x}\in{\rm ri}({\rm dom}\,h)$. From \cite[Theorem 23.4]{Roc70}, it follows that $\partial h(\widehat{x})\ne\emptyset$. Pick any $\widehat{v}\in\partial h(\widehat{x})$. Then, 
 \begin{equation}\label{hineq}
  h(x)\ge h(\widehat{x})+\langle\widehat{v},x-\widehat{x}\rangle\quad\ \forall\,x\in{\rm dom}\,h.
 \end{equation}
  In addition, from the finite convexity of $\vartheta_1$, $\partial\vartheta_1(F(\widehat{x}))\ne\emptyset$. Picking any $\widehat{\zeta}\in\partial\vartheta_1(F(\widehat{x}))$, we have 
  \begin{equation}\label{vtheta1-ineq}
   \vartheta_1(y)\ge\vartheta_1(F(\widehat{x}))+\langle\widehat{\zeta},y-F(\widehat{x})\rangle\quad\ \forall y\in\mathbb{Y}. 
  \end{equation}
  Substituting \eqref{vtheta1-ineq} with $y=\ell_{F}(x^{k,j},x^k)$ and \eqref{hineq} with $x=x^{k,j}$ into inequality \eqref{temp-ineq31} leads to
  \begin{align*}
  \Phi(x^k)\ge q_{k,j}^{\rm LB}
  &\ge \vartheta_1(F(\widehat{x}))+\langle\widehat{\zeta},F(x^k)-F(\widehat{x})\rangle+\langle\nabla\!F(x^k)\widehat{\zeta}+\nabla G(x^k)\xi^k, x^{k,j}-x^k\rangle\\
  &\quad+h(\widehat{x})+\langle\widehat{v},x^{k,j}-\widehat{x}\rangle+\frac{1}{2}\|x^{k,j}-x^k\|_{Q_{k,j}}^2-\frac{\mu_k}{2}\|x^{k,j}-x^{k}\|^2-\!\Theta_2(x^k),
 \end{align*} 	
 where the first inequality is due to $\Phi(x^k)=q_{k,j}(x^k)\ge q_{k,j}(\overline{x}^{k,j})\ge q_{k,j}^{\rm LB}$ for each $j\in\mathbb{N}$. Recall that the sequence $\{\mu_k\}_{k\in\mathbb{N}}\subset\mathbb{R}_{++}$ is bounded and $\mathcal{Q}_{k,j}\succ\gamma_{k,j}\mathcal{I}=\overline{\rho}^j\gamma_{k,0}\mathcal{I}$ for each $j$. The above inequality implies that $x^{k,j}\to x^k$ as $j\to\infty$. Consequently, by invoking \eqref{Theta-ineq30} with $x=x^{k,j}$, for sufficiently large $j$,
 \begin{equation*}
  \Theta(x^{k,j})-\vartheta_1\big(\ell_{F}(x^{k,j},x^k)\big)+\vartheta_2\big(\ell_{G}(x^{k,j},x^k)\big)\le(L_k/2)\|x^{k,j}-x^k\|^2,
 \end{equation*}
 which clearly contradicts the above \eqref{aim-ineq1} because $L_k\mathcal{I}\prec\mathcal{Q}_{k,j}$ when $j$ is large enough.
\end{proof}
\section{Convergence analysis of Algorithm \ref{iLPA}}\label{sec4}
 In this section, let $\{(x^k,\xi^k,\mathcal{Q}_{k})\}_{k\in\mathbb{N}}$ be the sequence generated by Algorithm \ref{iLPA}. For each $k\in\mathbb{N}$, write $w^{k}:=(\overline{x}^{k},x^{k-1},\xi^{k-1},\mathcal{Q}_{k-1})$. From the iteration of Algorithm \ref{iLPA}, obviously, $\{w^k\}_{k\in\mathbb{N}}\subset{\rm dom}\,h\times{\rm dom}\,h\times(-{\rm dom}\,\vartheta_2^*)\times\mathbb{S}_{++}$. To analyze the convergence of Algorithm \ref{iLPA}, we need to construct a potential function. Inspired by the work \cite{LiuPong19}, for each $w=(x,s,z,\mathcal{Q})\in\mathbb{W}:=\mathbb{X}\times\mathbb{X}\times\mathbb{Z}\times\mathbb{S}_{+}$, we define 
\begin{equation}\label{Xi-fun}
 \Xi(w):=\vartheta_1(\ell_{F}(x,s))+\langle \ell_{G}(x,s),z\rangle+h(x)
	+\vartheta_2^*(-z)+\|x-s\|_{\mathcal{Q}}^2+\chi_{\mathbb{S}_{+}}(\mathcal{Q}).
\end{equation}
The potential function $\Xi$ has a close relation with the objective function of \eqref{subprobkj}. Indeed, for any  $(x,s)\in\mathbb{X}\times\mathbb{X}$ and $\xi\in\partial(-\vartheta_2)(G(s))$, from the convexity of $\vartheta_2$ and \cite[Theorem 23.5]{Roc70}, $\vartheta_2(G(s))+\vartheta_2^*(-\xi)=-\langle \xi,G(s)\rangle$, which means that $\langle \ell_{G}(x,s),\xi\rangle+\vartheta_2^*(-\xi)=\langle\nabla G(s)\xi,x-s\rangle-\vartheta_2(G(s))$. By comparing with the expression of $\Xi$ and $q_{k,j}$ and using the definition of $w^k$, for each $k\in\mathbb{N}$, it holds that
\begin{align}
\Xi(w^k)&\!=\!\vartheta_1(\ell_{F}(\overline{x}^k,x^{k-1}))\!+\!\langle\nabla G(x^{k-1})\xi^{k-1},\overline{x}^k\!-\!x^{k-1}\rangle\!-\!\vartheta_2(G(x^{k-1}))\!+\!h(\overline{x}^k)\!+\!\|\overline{x}^k\!-\!x^{k-1}\|_{\mathcal{Q}_{k-1}}^2\nonumber\\  
 &=q_{k-1,j_{k-1}}(\overline{x}^k)+\frac{1}{2}\|\overline{x}^k-x^{k-1}\|_{\mathcal{Q}_{k-1}}^2\le q_{k-1,j_{k-1}}(x^{k-1})=\Phi(x^{k-1}),
 \label{relation}
\end{align}    
 where the inequality is by the strong convexity of $q_{k-1,j_{k-1}}$. For each $k\in\mathbb{N}$, from the step 3 of Algorithm \ref{iLPA}, $\xi^{k}\in\partial(-\vartheta_2)(G(x^k))\subset-\partial\vartheta_2(G(x^k))$, which by the convexity of $\vartheta_2$ and \cite[Theorem 23.5]{Roc70} implies  
\begin{subnumcases}{}\label{pre-equa0}	
 \vartheta_2(G(x^{k}))+\vartheta_2^*(-\xi^{k})=-\langle\xi^{k},G(x^{k})\rangle,\\
 \label{pre-equa1}  	
 -\vartheta_2(\ell_{G}(x^k,x^{k-1}))\le-\vartheta_2(G(x^{k-1}))+\langle\nabla G(x^{k-1})\xi^{k-1},x^{k}-x^{k-1}\rangle.
\end{subnumcases} 
 The relations in \eqref{relation} and \eqref{pre-equa0}-\eqref{pre-equa1} are often used in the subsequent convergence analysis.  
\subsection{Subsequential convergence of Algorithm \ref{iLPA}}\label{sec4.1}

 The following proposition proves the convergence of the sequences $\{\Xi(w^k)\}_{k\in\mathbb{N}}$ and $\{\Phi(x^k)\}_{k\in\mathbb{N}}$.
\begin{proposition}\label{prop1-xk}
 For the sequence $\{w^k\}_{k\!\in\mathbb{N}}$, the following three statements hold under Assumption \ref{ass0}.
 \begin{enumerate}
 \item [(i)] For each $k\in\mathbb{N}$, $\Xi(w^{k+1})\le\Xi(w^{k})-(\underline{\gamma}/4-\mu_{k-1})\|x^k-x^{k-1}\|^2$.
		
 \item[(ii)] For each $k\in\mathbb{N}$, $\Phi(x^k)+(\underline{\gamma}/4-\mu_{k-1})\|x^k-x^{k-1}\|^2\le\Xi(w^k)\leq \Phi(x^{k-1})$.
		
 \item[(iii)] The sequences $\{\Phi(x^k)\}_{k\in\mathbb{N}}$ and $\{\Xi(w^{k})\}_{k\in\mathbb{N}}$ are nonincreasing and convergent.
 \end{enumerate}
\end{proposition}
\begin{proof}\!\!\! .
{\bf(i)} Fix any $k\in\mathbb{N}$. From the step 6 of Algorithm \ref{iLPA} and the inequality \eqref{pre-equa1}, we have
 \begin{align*}
 \Theta(x^{k})
  &\le\vartheta_1(\ell_{F}(x^k,x^{k-1}))-\vartheta_2(\ell_{G}(x^k,x^{k-1}))+(1/2)\|x^k-x^{k-1}\|_{\mathcal{Q}_{k-1}}^2\\
  &\le\!\vartheta_1(\ell_{F}(x^k,x^{k-1}))\!-\!\Theta_2(x^{k-1})\!+\!
  \langle\nabla G(x^{k-1})\xi^{k-1},x^k-x^{k-1}\rangle+(1/2)\|x^k\!-\!x^{k-1}\|_{\mathcal{Q}_{k-1}}^2.
\end{align*}
Recall that $\Phi(x^{k})=\Theta(x^{k})+h(x^k)$ for each $k\in\mathbb{N}$. Along with the above inequality, it follows  
\[
 \Phi(x^k)\le\vartheta_1(\ell_{F}(x^k,x^{k-1}))-\Theta_2(x^{k-1})+
  \langle\nabla G(x^{k-1})\xi^{k-1},x^k-x^{k-1}\rangle+\frac{1}{2}\|x^k-x^{k-1}\|_{\mathcal{Q}_{k-1}}^2+h(x^k).
\] 
Combining with $\Xi(w^{k+1})\le\Phi(x^k)$ by \eqref{relation} and the expression of $q_{k-1,j_{k-1}}$, for each $k\in\mathbb{N}$, it holds 
\begin{align}
 \Xi(w^{k+1})&\!\le\!\vartheta_1\big(\ell_{F}(x^k\!,x^{k-1})\big)\!+\!\langle\nabla G(x^{k-1})\xi^{k-1},x^k\!-\!x^{k-1}\rangle+h(x^k)\!+\!\frac{1}{2}\|x^k\!-\!x^{k-1}\|_{\mathcal{Q}_{k-1}}^2\!\!\!-\!\Theta_2(x^{k\!-\!1})\nonumber\\
 \label{final-ineq40}
 &=q_{k-1,j_{k-1}}(x^{k})\stackrel{\eqref{inexact-cond}}{\le} q_{k-1,j_{k-1}}^{\rm LB}+\frac{\mu_{k-1}}{2}\|x^{k}-x^{k-1}\|^2\nonumber\\
 &\le q_{k-1,j_{k-1}}(\overline{x}^k)+\frac{\mu_{k-1}}{2}\|x^{k}-x^{k-1}\|^2\\
 \label{temp-ineq40}
 &\stackrel{\eqref{relation}}{=}\Xi(w^{k})-(1/2)\|\overline{x}^k-x^{k-1}\|_{\mathcal{Q}_{k-1}}^2+(\mu_{k-1}/2)\|x^{k}-x^{k-1}\|^2.
 \end{align}
 From Cauchy-Schwarz inequality, 
 $\frac{1}{4}\|{x}^k-x^{k-1}\|_{\mathcal{Q}_{k-1}}^2
 \le \frac{1}{2}\|{x}^k-\overline{x}^k\|_{\mathcal{Q}_{k-1}}^2+\frac{1}{2}\|\overline{x}^k-x^{k-1}\|_{\mathcal{Q}_{k-1}}^2$. While from the strong convexity of $q_{k-1,j_{k-1}}$ and the above \eqref{final-ineq40}, it follows that
 \begin{equation}\label{key-ineq40}
  \frac{1}{2}\|\overline{x}^k-x^{k}\|_{\mathcal{Q}_{k-1}}^2
 \le q_{k-1,j_{k-1}}(x^{k})-q_{k-1,j_{k-1}}(\overline{x}^{k})\le\frac{\mu_{k-1}}{2}\|x^{k}-x^{k-1}\|^2.
 \end{equation}
 Then, from the two sides, we immediately obtain that
 \begin{equation}\label{key-ineq41}
  \frac{1}{2}\|\overline{x}^k-x^{k-1}\|_{\mathcal{Q}_{k-1}}^2  \ge-\frac{\mu_{k-1}}{2}\|{x}^{k}-x^{k-1}\|^2+\frac{1}{4}\|{x}^k-x^{k-1}\|_{\mathcal{Q}_{k-1}}^2.
 \end{equation}
 Substituting this inequality into \eqref{temp-ineq40} and noting that $\mathcal{Q}_{k-1}\succeq\underline{\gamma}\mathcal{I}$ leads to the desired inequality.
	
 \noindent
 {\bf(ii)} Fix any $k\in\mathbb{N}$. From the inequality \eqref{final-ineq40} and the expression of $q_{k-1,j_{k-1}}$, it follows that
 \begin{align*}
  q_{k-1,j_{k-1}}(\overline{x}^{k})
  &\ge\vartheta_1(\ell_{F}({x}^{k},x^{k-1}))-\vartheta_2(G(x^{k-1}))+\langle\nabla G(x^{k-1})\xi^{k-1},x^{k}-x^{k-1}\rangle\\
  &\quad+h(x^k)+\!(1/2)\|x^k-x^{k-1}\|_{\mathcal{Q}_{k-1}}^2-(\mu_{k-1}/2)\|{x}^{k}-x^{k-1}\|^2\\
  &\stackrel{\eqref{pre-equa1}}{\ge}\vartheta_1(\ell_{F}({x}^{k},x^{k-1}))-\vartheta_2(\ell_G(x^k,x^{k-1}))+h(x^k)\\
  &\quad+\frac{1}{2}\|x^k-x^{k-1}\|_{\mathcal{Q}_{k-1}}^2-\frac{\mu_{k-1}}{2}\|{x}^{k}\!-\!x^{k-1}\|^2\\
  &\ \ge \Phi(x^k)-(\mu_{k-1}/2)\|{x}^{k}-x^{k-1}\|^2.
  \end{align*}
  Then, from \eqref{relation},
 \(
  \Phi(x^{k-1})\ge\Xi(w^k)=q_{k-1,j_{k-1}}(\overline{x}^{k})\ge\Phi(x^k)-\frac{\mu_{k-1}}{2}\|{x}^{k}-x^{k-1}\|^2+\frac{1}{2}\|\overline{x}^k-x^{k-1}\|_{\mathcal{Q}_{k-1}}^2,
 \)
 which together with \eqref{key-ineq41} and $\mathcal{Q}_{k-1}\succeq\underline{\gamma}\mathcal{I}$ implies the desired inequality.
	
 \noindent
 {\bf(iii)} By Remark \ref{remark-alg} (a), $\mu_{k-1}\le\underline{\gamma}/5$ for all $k\ge\overline{k}$, which by part (ii) implies the nonincreasing of $\{\Xi(w^k)\}_{k\ge\overline{k}}$ and $\{\Phi(x^k)\}_{k\ge\overline{k}}$. Recall that $\Phi$ is bounded from below by Assumption \ref{ass0} (iii), so is $\{\Xi(w^{k})\}_{k\in\mathbb{N}}$ is bounded from below. Then, $\{\Xi(w^k)\}_{k\in\mathbb{N}}$ and $\{\Phi(x^k)\}_{k\in\mathbb{N}}$ are convergent.
\end{proof}

 To achieve the subsequential convergence of $\{w^k\}_{k\in\mathbb{N}}$, we need the following assumption, which is rather weak since it requires $\Phi$ to have a bounded level set on $\Phi(x^0)$ instead of all real numbers.
\begin{assumption}\label{ass1}
 The level set $\mathcal{L}_{\Phi}(x^0):=\{x\in\mathbb{X}\ |\ \Phi(x)\le\Phi(x^0)\}$ is bounded. 
\end{assumption}

\begin{proposition}\label{prop2-xk}
 Under Assumptions \ref{ass0}-\ref{ass1}, the following assertions hold.
 \begin{enumerate}
 \item[(i)] The sequence $\{\gamma_k\}_{k\in\mathbb{N}}$ with $\gamma_k:=\gamma_{k,j_k}$ is bounded, so is the sequence $\{\|\mathcal{Q}_k\|\}_{k\in\mathbb{N}}$.
				
 \item[(ii)] The sequence $\{w^{k}\}_{k\in\mathbb{N}}$ is bounded and the set of its accumulation points, denoted by $\Gamma^*$, is a nonempty and compact.
		
 \item[(iii)] For each $\overline{w}=(\overline{x},\overline{s},\overline{\xi},\overline{\mathcal{Q}})\in\Gamma^*$, it holds that $\overline{x}=\overline{s}\in\mathcal{S}^*$ and $\Xi(\overline{w})=\overline{\Xi}:=\lim_{k\to\infty}\Xi(w^k)$. 
\end{enumerate}
\end{proposition}
\begin{proof}\!\!\! .
 We first argue that the sequence $\{(\overline{x}^k,x^{k},\xi^k)\}_{k\in\mathbb{N}}$ is bounded. Indeed, by Proposition \ref{prop1-xk} (ii), $\{x^k\}_{k\in\mathbb{N}}\subset\mathcal{L}_{\Phi}(x^0)$, so the boundedness of $\{x^k\}_{k\in\mathbb{N}}$ follows  Assumption \ref{ass1} (i). Recall that $-\vartheta_2$ is strictly continuous and $\xi^k\in\partial(-\vartheta_2)(G(x^k))$ for each $k$. Combining the boundedness of $\{x^k\}_{k\in\mathbb{N}}$ and \cite[Theorem 9.13 \& Proposition 5.15]{RW98} (d), we achieve the boundedness of $\{\xi^k\}_{k\in\mathbb{N}}$. By combining \eqref{key-ineq40} with $\mathcal{Q}_{k-1}\succeq\underline{\gamma}\mathcal{I}$ and the boundedness of $\{x^k\}_{k\in\mathbb{N}}$ and $\{\mu_k\}_{k\in\mathbb{N}}$, we infer that $\{\overline{x}^k\}_{k\in\mathbb{N}}$ is bounded.    

 \noindent
 {\bf(i)} Suppose on the contrary that the sequence $\{\gamma_k\}_{k\in\mathbb{N}}$ is unbounded. Then there exists an index set $\mathcal{K}\subset\mathbb{N}$ such that $\lim_{\mathcal{K}\ni k\to\infty}\gamma_k=\infty$. Without loss of generality, for each $k\in\mathcal{K}$, we assume $j_k\ge 1$ and write $\widetilde{\gamma}_k:=\overline{\varrho}^{-1}\gamma_k=\gamma_{k,j_k-1}$. From the inner loop of Algorithm \ref{iLPA} and $\mathcal{Q}_{k,j_k-1}\succeq \widetilde{\gamma}_{k}\mathcal{I}$, for each $k\in\mathcal{K}$, 
 \begin{equation}\label{aim-ineq40}
  \Theta(x^{k,j_k-1})
  >\vartheta_1\big(\ell_{F}(x^{k,j_k-1},x^k)\big)-\vartheta_2\big(\ell_{G}(x^{k,j_k-1},x^k)\big)+({\widetilde{\gamma}_k}/{2})\|x^{k,j_k-1}-x^k\|^2.
 \end{equation}
 Note that $x^{k,j_k-1}\ne x^{k}$ for each $k\in\mathcal{K}$ (if, by the step \ref{step5} of Algorithm \ref{iLPA}, $x^k=x^{k,j_k-1}=\overline{x}^{k,j_k-1}$, so $0\in\partial q_{k,j_k-1}(x^k)$ and $x^k$ is a stationary point). Now we prove that $\lim_{\mathcal{K}\ni k\rightarrow\infty}\|x^{k,j_k-1}-x^{k}\|=0$. For each $k\in\mathcal{K}$, from the optimality of $\overline{x}^{k,j_k-1}$ and the feasibility of $x^k$ to the subproblem \eqref{subprobkj} with $j=j_{k}-1$, 
 \begin{align*}
  \Phi(x^k)=q_{k,j_k-1}({x}^{k})&\ge q_{k,j_k-1}(\overline{x}^{k,j_k-1})\ge q_{k,j_k-1}^{\rm LB}\stackrel{\eqref{inexact-cond}}{\ge} q_{k,j_k-1}(x^{k,j_k-1})-(\mu_k/2)\|x^{k,j_k-1}-x^{k}\|^2\\
 &=\vartheta_1(\ell_{F}(x^{k,j_k-1},x^k))+\langle\nabla G(x^k)\xi^k,x^{k,j_k-1}-x^k\rangle+h(x^{k,j_k-1})-\Theta_2(x^k)\nonumber\\
  &\quad +(1/2)\|x^{k,j_k-1}-x^k\|_{\mathcal{Q}_{k,j_k-1}}^2-(\mu_k/2)\|x^{k,j_k-1}-x^{k}\|^2.
 \end{align*}
 Combining this inequality with \eqref{vtheta1-ineq} for $y=\ell_{F}(x^{k,j_k-1},x^k)$ and \eqref{hineq} for $x=x^{k,j_k-1}$ yields that
 \begin{align*}
  \Phi(x^k)&\ge \langle\nabla F(x^k)\widehat{\zeta}+\nabla G(x^k)\xi^k,x^{k,j_k-1}-x^k\rangle+\langle\widehat{v},x^{k,j_k-1}-\widehat{x}\rangle+\Theta_1(\widehat{x})+h(\widehat{x})-\Theta_2(x^k)\\
  &\quad+\langle\widehat{\zeta},F(x^k)\!-\!F(\widehat{x})\rangle+(\widetilde{\gamma}_{k}/2)\|x^{k,j_k-1}-x^k\|^2-(\mu_k/2)\|x^{k,j_k-1}-x^{k}\|^2
 \end{align*}
 where the inequality is also using $\mathcal{Q}_{k,j_k-1}\succeq\widetilde{\gamma}_{k}\mathcal{I}$ for each $k\in\mathcal{K}$. After a suitable rearrangement,
 \begin{align*}
  (1/2)(\widetilde{\gamma}_k-\mu_k)\|x^{k,j_k-1}-x^{k}\|^2
  &\le \Phi(x^k)+\Theta_2(x^k)-\Theta_1(\widehat{x})-\langle\widehat{\zeta},F(x^k)-F(\widehat{x})\rangle-h(\widehat{x})\\
  &\quad-\langle\nabla F(x^k)\widehat{\zeta}+\nabla G(x^k)\xi^k,x^{k,j_k-1}-x^k\rangle-\langle\widehat{v},x^{k,j_k-1}-\widehat{x}\rangle\\
  &\le\Phi(x^0)\!+\!\Theta_2(x^k)\!-\!\Theta_1(\widehat{x})-\!\langle\widehat{\zeta},F(x^k)\!-\!F(\widehat{x})\rangle-h(\widehat{x})\!-\!\langle\widehat{v},x^k\rangle\\
  &\quad+\!\langle\widehat{v},\widehat{x}\rangle\!+\big[\|\nabla F(x^{k})\widehat{\zeta}\|
   \! +\!\|\nabla G(x^{k})\xi^{k}\|\!+\!\|\widehat{v}\|\big]\|{x}^{k,j_k-1}-x^{k}\|,
 \end{align*}
 where the second inequality is due to $\Phi(x^k)\le\Phi(x^0)$ by Proposition \ref{prop1-xk} (ii). Note that $\lim_{\mathcal{K}\ni k\to\infty}\widetilde{\gamma}_k=\infty$. Passing the limit $\mathcal{K}\ni k\to\infty$, recalling that $x^{k,j_k-1}\ne x^{k}$ for each $k\in K$ and using the boundedness of $\{(x^k,\xi^k)\}_{k\in\mathbb{N}}$, we obtain the desired limit $\lim_{\mathcal{K}\ni k\to\infty}\|x^{k,j_k-1}\!\!-\!x^{k}\|=0$. Then, invoking the previous \eqref{Theta-ineq30} with $x=x^{k,j_k-1}$ for sufficiently large $k\in\mathcal{K}$, it follows that for large enough $k\in\mathcal{K}$,
 \begin{equation*}
  \Theta(x^{k,j_k-1})\le\vartheta_1\big(\ell_{F}(x^{k,j_k-1},x^k)\big)-\vartheta_2\big(\ell_{G}(x^{k,j_k-1},x^k)\big)+(L_k/2)\|x^{k,j_k-1}-x^k\|^2,
 \end{equation*} 
 where $L_k$ is a constant close enough to ${\rm lip}\,\Theta(x^k)$ from above. Note that the sequence $\{{\rm lip}\,\Theta(x^k)\}_{k\in\mathbb{N}}$ is bounded due to the boundedness of $\{x^k\}_{k\in\mathbb{N}}$ and \cite[Theorem 9.2]{RW98}, so is the sequence $\{L_k\}_{k\in\mathbb{N}}$. Then, for  sufficiently large $k\in\mathcal{K}$, the above inequality is a contradiction to \eqref{aim-ineq40}.
 
 \noindent
 {\bf(ii)} The boundedness of $\{w^k\}_{k\in\mathbb{N}}$ follows that of $\{(\overline{x}^k,x^{k},\xi^k)\}_{k\in\mathbb{N}}$ and part (i), which implies the nonemptyness and compactness of the set $\Gamma^*$.
	
 \noindent
 {\bf(iii)} Pick any $\overline{w}=(\overline{x},\overline{s},\overline{\xi},\overline{\mathcal{Q}})\in\Gamma^*$. There exists an index set $\mathcal{K}\subset\mathbb{N}$ such that $\lim_{\mathcal{K}\ni k\to\infty}w^k=\overline{w}$. From Proposition \ref{prop1-xk}  (ii)-(iii), $\lim_{\mathcal{K}\ni k\to\infty}\|x^k-x^{k-1}\|=0$. Along with \eqref{key-ineq40} and $\mathcal{Q}_k\ge\underline{\gamma}\mathcal{I}$, we have $\lim_{\mathcal{K}\ni k\to\infty}\|\overline{x}^k-x^{k}\|=0$, so $\overline{x}=\overline{s}$. From the definition of $\overline{x}^k$ and \cite[Theorem 10.6]{RW98}, for each $k\in\mathbb{N}$, 
 \begin{equation*}
  0\in\nabla F(x^{k-1})\partial\vartheta_1(\ell_{F}(\overline{x}^k,x^{k-1}))+\nabla G(x^{k-1})
  \partial(-\vartheta_2)(G(x^{k-1}))+\mathcal{Q}_{k-1}(\overline{x}^k-x^{k-1})+\partial h(\overline{x}^k).
 \end{equation*}
 Passing the limit $\mathcal{K}\ni k\to\infty$ to the last inclusion and using the outer semicontinuity of $\partial\vartheta_1,\partial(-\vartheta_2)$ and $\partial h$ yields that $0\in\nabla F(\overline{x})\partial\vartheta_1(F(\overline{x}))+\nabla G(\overline{x})\partial(-\vartheta_2)(G(\overline{x}))+\partial h(\overline{x})$, so $\overline{x}$ is a stationary point of problem \eqref{prob}. We next argue that $\Xi(\overline{w})=\lim_{k\to\infty}\Xi(w^k)$. Note that $-\overline{\xi}\in\partial\vartheta_2(G(\overline{x}))$. Hence, $\langle G(\overline{x}),\overline{\xi}\rangle+\vartheta_2^*(-\overline{\xi})=-\vartheta_2(G(\overline{x}))$. By the expression of $\Xi$ and $\overline{x}=\overline{s}$, we have $\Xi(\overline{w})=\Theta_1(\overline{x})-\Theta_2(\overline{x})+h(\overline{x})$. From the feasibility of $\overline{x}$ and the optimality of $\overline{x}^k$ to the $(k-1)$-th subproblem, it holds that
 \begin{align*}
 &\vartheta_1(\ell_{F}(\overline{x}^k,x^{k-1}))+\langle\nabla G(x^{k-1})\xi^{k-1},\overline{x}^{k}-x^{k-1}\rangle+h(\overline{x}^{k})+\frac{1}{2}\|\overline{x}^{k}-x^{k-1}\|_{\mathcal{Q}_{k-1}}^2\\
 &\le\vartheta_1(\ell_{F}(\overline{x},x^{k-1}))+\langle\nabla G(x^{k-1})\xi^{k-1},\overline{x}-x^{k-1}\rangle+h(\overline{x})+\frac{1}{2}\|\overline{x}-x^{k-1}\|_{\mathcal{Q}_{k-1}}^2,
 \end{align*}
 which by the continuity of $\vartheta_1,\,F'$ and $G'$ implies that $\limsup_{\mathcal{K}\ni k\to\infty}h(\overline{x}^{k})\le h(\overline{x})$. Along with the lower semicontinuity of $h$, $\lim_{\mathcal{K}\ni k\to\infty}h(\overline{x}^{k})\!=\!h(\overline{x})$. By \eqref{pre-equa0}, $\lim_{\mathcal{K}\ni k\to\infty}\vartheta_2^*(-\xi^k)\!=-\vartheta_2(G(\overline{x}))-\langle G(\overline{x}),\overline{\xi}\rangle$. Thus, by the expression of $\Xi$,  $\lim_{\mathcal{K}\ni k\to\infty}\Xi(w^k)=\vartheta_1(F(\overline{x}))-\vartheta_2(G(\overline{x}))+h(\overline{x})=\Xi(\overline{w})$. 
\end{proof}

\subsection{Full convergence of Algorithm \ref{iLPA}}\label{sec4.2}

To achieve the full convergence of $\{x^k\}_{k\in\mathbb{N}}$, we need  the twice continuous differentiability of $F$ and $G$ on an open set $\mathcal{V}\supset\Gamma^*$ as in Assumption \ref{ass2}, a little stronger than the $\mathcal{C}^{1,1}$ assumption on $F$ in \cite{LeThi23}.  
\begin{assumption}\label{ass2}
 There exists an open set $\mathcal{V}\supset\Gamma^*$ such that $F'$ and $G'$ are continuously differentiable on $\mathcal{V}_{x}$ where, for a given set $\Gamma\subset\mathbb{W}$, the sets $\Gamma_{\!x}$ and $\Gamma_{\!z}$ are the projection of $\Gamma$ onto $\mathbb{X}$ and $\mathbb{Z}$, respectively.
\end{assumption}

The rest of this section focuses on the following full convergence result under Assumptions \ref{ass0}-\ref{ass2}.
\begin{theorem}\label{global-converge}
 Suppose that Assumptions \ref{ass0}-\ref{ass2} hold, and that $\Xi$ is a KL function. Then, the sequence $\{x^k\}_{k\in\mathbb{N}}$ is convergent and its limit  $x^*$ is a stationary point of the problem \eqref{prob}.
\end{theorem}

To prove the result of Theorem \ref{global-converge}, we first characterize the subdifferential of $\Xi$ at any $w\in{\rm dom}\,\Xi$.
\begin{lemma}\label{subdiff-LemXi}
 Fix any $w=(x,s,z,\mathcal{Q})\in{\rm dom}\,h\times{\rm dom}\,h\times(-{\rm dom}\,\vartheta_2^*)\times\mathbb{S}_{+}$. Suppose that $F'$ and $G'$ are strictly differentiable at $s$. Then $\Xi$ is regular at $w$ with $\partial\Xi(w)=T_1(w)\times T_2(w)\times T_3(w)$, where
 \begin{align*}
 T_1(w)&=\left(\begin{matrix}
          \nabla\!F(s)\\
 	[D^2F(s)(x-s,\cdot)]^*\\	
      \end{matrix}\right)\partial\vartheta_1(\ell_{F}(x,s))
      +\left(\begin{matrix}
        \nabla G(s)z+\partial h(x)+2\mathcal{Q}(x-s)\\
	[D^2G(s)(x-s,\cdot)]^*z+2\mathcal{Q}(s-x)\\
      \end{matrix}\right),\\
 T_2(w)&=\ell_{G}(x,s)-\partial\vartheta_2^*(-z)\ \ {\rm and}\ \ T_3(w)=(s-x)(s-x)^{\top}+\mathcal{N}_{\mathbb{S}_{+}}(\mathcal{Q}).	 	
 \end{align*}
\end{lemma}
\begin{proof}\!\!\! .
 Let $\psi(x',s'):=\vartheta_1(\ell_{F}(x',s'))$ for $(x',s')\in\mathbb{X}\times\mathbb{X}$. For any $w'=(x',s',z',\mathcal{Q}')\in\mathbb{W}$, define 
 \[
  f(w'):=\psi(x',s')+h(x')+\vartheta_2^*(-z')+ \chi_{\mathbb{S}_{+}}(\mathcal{Q}')\ \ {\rm and}\ \ g(w'):=\langle\ell_{G}(x',s'),z'\rangle+\|x'-s'\|_{\mathcal{Q}'}^2.
 \]
 Then $\Xi=f+g$. The strict differentiability of $G'$ at $s$ implies that of $g$ at $w$, so $\partial\Xi(w)=\partial\!f(w)+\nabla g(w)$ and $\widehat{\partial}\Xi(w)=\widehat{\partial}\!f(w)+\nabla g(w)$. Recall that $\vartheta_1$ is strictly continuous at $\ell_{F}(x,s)$ and ${\rm dom}\,\vartheta_1=\mathbb{Y}$. From the strict differentiability of $F'$ at $s$ and \cite[Theorem 10.6]{RW98},  $\widehat{\partial}\psi(x,s)=\partial\psi(x,s)=\nabla\ell_{F}(x,s)\partial\vartheta_1(\ell_{F}(x,s))$.
 Since $\psi$ is strictly continuous at $w$, by invoking \cite[Exercise 10.10 \& Proposition  10.5]{RW98}, we have
 \begin{equation}\label{temp-equa40}
   \partial\!f(w)=\widehat{\partial}\!f(w) =\partial\psi(x,s)\times\{0\}\times\{0\}+\partial h(x)\times\{0\}\times[-\partial\vartheta_2^*(-z)]\times\mathcal{N}_{\mathbb{S}_{+}}(\mathcal{Q}).
 \end{equation}  
 This shows that the function $\Xi$ is regular at $w$, 
 and the second part also holds. 
\end{proof}

By using Lemma \ref{subdiff-LemXi}, we can provide the relative error condition for $w^k$ to be a stationary point of $\Xi$. 
\begin{proposition}\label{prop3-xk}
 Under Assumptions \ref{ass0}-\ref{ass2}, it holds $\Gamma^*\subset{\rm crit}\,\Xi$, the set of critical points of $\Xi$, and there exist $\widehat{k}\in\mathbb{N}$ and a constant $b>0$ such that for all $k\ge\widehat{k}$, ${\rm dist}(0,\partial\Xi(w^{k}))\le b\|x^k-x^{k-1}\|$.
\end{proposition}
\begin{proof}\!.
 Pick any $\overline{w}=(\overline{x},\overline{s},\overline{\xi},\overline{\mathcal{Q}})\in\Gamma^*$. From Proposition \ref{prop2-xk} (iii) and its proof, we have $\overline{x}=\overline{s}$, $	 0\in\nabla\!F(\overline{x})\partial\!\vartheta_1\!(F(\overline{x}))\!+\nabla G(\overline{x})\partial(-\vartheta_2)(G(\overline{x}))\!+\!\partial h(\overline{x})$ and $-\overline{\xi}\in\partial\vartheta_2(G(\overline{x}))$. Together with Lemma \ref{subdiff-LemXi}, we have $0\in\partial\Xi(\overline{w})$, i.e., $\overline{w}\in{\rm crit}\,\Xi$. The inclusion $\Gamma^*\subset{\rm crit}\,\Xi$ follows  the arbitrariness of $\overline{w}\in\Gamma^*$.

 For the second part, recall that $\Gamma^*$ is the set of cluster points, so $\lim_{k\to\infty}{\rm dist}(w^k,\Gamma^*)=0$. Then, there exists $\widehat{k}>\overline{k}$ such that $w^k\in\mathcal{V}$ for all $k\ge\widehat{k}$, where $\overline{k}$ is the same as in Remark \ref{remark-alg} (a). For each $k\ge\widehat{k}$, by the optimality of $\overline{x}^{k}$ to the $(k\!-\!1)$th subproblem, there is $v^{k-1}\in\partial\vartheta_1(\ell_{F}(\overline{x}^{k},x^{k-1}))$ such that
 \begin{equation}\label{temp-inclusion1}
  0\in\nabla F(x^{k-1})v^{k-1}
    +\nabla G(x^{k-1})\xi^{k-1}+\partial h(\overline{x}^{k})+\mathcal{Q}_{k-1}(\overline{x}^{k}-x^{k-1}).
 \end{equation}
 Note that $\mathcal{N}_{\mathbb{S}_{+}}(\mathcal{Q}_{k-1})=\{0\}$ since $\mathcal{Q}_{k-1}\succeq\underline{\gamma}\mathcal{I}$, and $G(x^{k-1})\in\partial\vartheta_2^*(-\xi^{k-1})$. For each $k\ge\widehat{k}$, define
 \[
   \zeta^k\!:=\!\!\left(\begin{matrix}
   \mathcal{Q}_{k-1}(\overline{x}^k-x^{k-1})\\
   \![D^2F(x^{k-1})(\overline{x}^{k}\!-\!x^{k-1},\cdot)]^*v^{k-1}\!+\![D^2G(x^{k-1})(\overline{x}^{k}\!-\!x^{k-1},\cdot)]^*\xi^{k-1}\!+\!2\mathcal{Q}_{k-1}(x^{k-1}\!-\!\overline{x}^{k})\!\\
    \ell_{G}(\overline{x}^{k},x^{k-1})-G(x^{k-1})\\
   (x^{k-1}-\overline{x}^{k})(x^{k-1}-\overline{x}^{k})^{\top}
  \end{matrix}\right).
 \]
 Using \eqref{temp-inclusion1} and comparing with the expression of $\partial\Xi$ in Lemma \ref{subdiff-LemXi}, we obtain $\zeta^k\in\partial\Xi(w^k)$. Since $\vartheta_1$ and $\vartheta_2$ are strictly continuous, from \cite[Theorem 9.13]{RW98}, $\|v^{k-1}\|\le{\rm lip}\,\vartheta_1(\ell_{F}(\overline{x}^k,x^{k-1}))$  and $\|\xi^{k-1}\|\le{\rm lip}\,\vartheta_2(G(x^{k-1}))$. Recall that $\{(\overline{x}^k,x^k)\}_{k\in\mathbb{N}}$ is bounded, so are the sequences $\{\ell_{F}(\overline{x}^k,x^{k-1})\}_{k\in\mathbb{N}}$ and $\{G(x^{k-1})\}_{k\in\mathbb{N}}$. Together with \cite[Theorem 9.2]{RW98}, there necessarily exists a constant $b_1>0$ such that $\max(\|v^{k-1}\|,\|\xi^{k-1}\|)\le b_1$ for all $k\ge\widehat{k}$. Along with the boundedness of $\{\|\mathcal{Q}_{k-1}\|\}_{k\in\mathbb{N}}$ and the expression of $\zeta^k$, there exists $b_2>0$ such that
 $\|\zeta^k\|\le b_2\|x^{k-1}-\overline{x}^{k}\|\le b_2[\|x^{k-1}-x^k\|+\|\overline{x}^{k}-x^k\|]$. Furthermore, combining  $\mathcal{Q}_{k-1}\succeq\underline{\gamma}\mathcal{I}$ with \eqref{key-ineq40} and Remark \ref{remark-alg} (a) leads to $\|\overline{x}^k-x^k\|\le\!\sqrt{\mu_{k-1}\underline{\gamma}^{-1}}\|x^k-x^{k-1}\|\le\frac{1}{\sqrt{5}}\|x^k-x^{k-1}\|$ for all $k\ge\widehat{k}$. The result then holds with $b=\frac{b_2(\sqrt{5}+1)}{\sqrt{5}}$. 
\end{proof}

 Now we are ready to prove the convergence of the sequence $\{x^k\}_{k\in\mathbb{N}}$. Its proof is similar to that of \cite[Theorem 1]{Bolte14}, and we here include it to show that its proof depends only on Propositions \ref{prop1-xk}-\ref{prop3-xk}. 

 \medskip
 \noindent
 {\it The proof of Theorem \ref{global-converge}.}
 If there exists some $k_0\in\mathbb{N}$ such that $\Xi(w^{k_0})=\Xi(w^{k_0+1})$, then $x^{k_0}=x^{k_0-1}$ follows Proposition \ref{prop1-xk} (i). By Remark \ref{remark-alg} (e), Algorithm \ref{iLPA} finds a stationary point within a finite number of steps. Therefore, in view of Proposition \ref{prop1-xk} (i), it suffices to consider that $\Xi(w^k)>\Xi(w^{k+1})>\overline{\Xi}$ for all $k\in\mathbb{N}$. By Proposition \ref{prop2-xk} (ii)-(iii), the set $\Gamma^*$ is nonempty and compact, and $\Xi(w)=\overline{\Xi}$ for all $w\in\Gamma^*$. From the KL property of $\Xi$ on $\Gamma^*$ and \cite[Lemma 6]{Bolte14}, there exist $\varepsilon>0,\eta\in(0,\infty]$ and $\varphi\in\Upsilon_{\!\eta}$ such that 
 \[
  \varphi'(\Xi(w)-\overline{\Xi}){\rm dist}(0,\partial\Xi(w))\ge 1
 \] 
 for all $w\in[\overline{\Xi}<\Xi<\overline{\Xi}+\eta]\cap\mathfrak{B}(\Gamma^*,\varepsilon)$ with $\mathfrak{B}(\Gamma^*,\varepsilon)\!:=\big\{w\in\mathbb{W}\,|\,{\rm dist}(w,\Gamma^*)\le\varepsilon\big\}$. Recall that $\lim_{k\to\infty}\Xi(w^k)=\overline{\Xi}$ by Proposition \ref{prop1-xk} (iii) and $\lim_{k\to\infty}{\rm dist}(w^k,\Gamma^*)=0$. There exists $\mathbb{N}\ni\widetilde{k}\ge\widehat{k}+1$ such that for all $k\ge \widetilde{k}$, $w^k\in[\overline{\Xi}<\Xi<\overline{\Xi}\!+\!\eta]\cap\mathfrak{B}(\Gamma^*,\varepsilon)$. From the above inequality, it then follows 
 \[
  \varphi'\big(\Xi(w^k)\!-\!\overline{\Xi}\big){\rm dist}(0,\partial\Xi(w^k))\geq 1\quad\forall k\ge\widetilde{k}. 
 \]
 Along with Proposition \ref{prop3-xk}, for all $k\ge\widetilde{k}$, 
 $b\varphi'\big(\Xi(w^{k})-\overline{\Xi}\big)\|x^k-x^{k-1}\|\ge 1$. From the concavity of $\varphi$,
 \begin{align*}
 \Delta_{k,k+1}&:=\varphi(\Xi(w^k)-\overline{\Xi})-\varphi(\Xi(w^{k+1})-\overline{\Xi})
  \ge\varphi'(\Xi(w^k)-\overline{\Xi})(\Xi(w^{k})\!-\!\Xi(w^{k+1}))\\
  &\ge\frac{\Xi(w^{k})\!-\!\Xi(w^{k+1})}{b\|x^k-x^{k-1}\|}
  \ge (\underline{\gamma}/20)b^{-1}\|x^{k}-x^{k-1}\|\quad\forall k\ge\widetilde{k},
 \end{align*}
 where the third inequality is due to Proposition \ref{prop1-xk}. 
 Then, $\|x^{k}\!-\!x^{k-1}\|\le 20 b\underline{\gamma}^{-1}\Delta_{k,k+1}$ for all $k\ge\widetilde{k}$. Summing the above inequality from any $k\ge\widetilde{k}$ to any $l>k$ leads to 
 \begin{align}\label{ineq-rate}
 {\textstyle\sum_{i=k}^l}\|x^{i+1}-x^i\|
 \le 20b\underline{\gamma}^{-1}{\textstyle\sum_{i=k}^l}\Delta_{i,i+1}&\le 20b\underline{\gamma}^{-1}\big[\varphi\big(\Xi(w^{k})-\overline{\Xi}\big)-\varphi\big(\Xi(w^{l+1})-\overline{\Xi}\big)\big],\\
 &\le 20b\underline{\gamma}^{-1}\varphi\big(\Xi(w^{k})-\overline{\Xi}\big),\nonumber
 \end{align}
 where the second inequality is due to the nonnegativity of $\varphi$. Passing the limit $l\to\infty$ to the above inequality results in $\sum_{k=0}^{\infty}\|x^{k+1}\!-x^k\|<\infty$. Thus, we complete the proof. \qed
 
\begin{remark}
 From \cite[Theorem 4.1]{Attouch10}, if $\Xi$ is definable in an o-minimal structure over the real field $(\mathbb{R},+,\cdot)$, then it has the KL property at each point of ${\rm dom}\,\partial\Xi$. Note that $\chi_{\mathbb{S}_{+}}(\cdot)$ is semialgebraic according to Appendix A. Together with the expression of $\Xi$ and \cite[Sections 2 $\&$\,3]{Ioffe09}, when $F,G$ and $\vartheta_1,\vartheta_2,h$ are definable in the same o-minimal structure over $(\mathbb{R},+,\cdot)$, $\Xi$ is definable in this o-minimal structure. As discussed in \cite[Section 4]{Attouch10}, definable functions in an o-minimal structure are very rich, which cover semialgebraic functions and globally subanalytic functions. 
\end{remark}
\subsection{Local convergence rate of Algorithm \ref{iLPA}}\label{sec4.3}

In view of Theorem \ref{global-converge} and Propositions \ref{prop1-xk}-\ref{prop3-xk}, using the above \eqref{ineq-rate} and following the same arguments as those for \cite[Theorem 2]{Attouch09} yields the following local convergence rate result.  
\begin{theorem}\label{convergece-rate}
 Suppose that Assumptions \ref{ass0}-\ref{ass2} hold, and that $\Xi$ has the KL property of exponent $p\in[1/2,1)$ on the set $\Gamma^*$. Then, $\{x^k\}_{k\in\mathbb{N}}$ is convergent with limit $x^*$, and the following assertions hold:
 \begin{enumerate}
 \item[(i)]  when $p=1/2$, there exist $c_1>0$ and $\tau\in(0,1)$ such that $\|x^k-x^*\|\leq c_1\tau^k$;
		
 \item[(ii)] when $p\in({1}/{2},1)$, there exists $c_1>0$ such that $\|x^k-x^*\|\le c_1k^{-\frac{1-p}{2p-1}}$.
 \end{enumerate}
 \end{theorem}

 As is well known, to verify the KL property of exponent $1/2$ for a nonconvex and nonsmooth function is not an easy task because there are no convenient rules to identify it. Next we focus on the verifiable conditions for the KL property of $\Xi$ with exponent $p\in[1/2,1)$. To this end, we introduce
\begin{equation}\label{wXi-fun}
 \widetilde{\Xi}(x,s,z):=\vartheta_1(\ell_{F}(x,s))+\langle \ell_{G}(x,s),z\rangle+h(x)+\vartheta_2^*(-z)\quad\forall(x,s,z) \in\mathbb{X}\times\mathbb{X}\times\mathbb{Z}.
\end{equation}
Under the assumption of Lemma \ref{subdiff-LemXi}, for any $(x,s,z)\in{\rm dom}\,h\times{\rm dom}\,h\times(-{\rm dom}\,\vartheta_2^*)$,
\begin{equation}\label{subdiff-wXi}
 \partial\widetilde{\Xi}(x,s,z)=
  \left(\begin{matrix}
  \left(\begin{matrix}
   \nabla F(s)	\\ [D^2F(s)(x-s,\cdot)]^*\\ 
   \end{matrix}\right)\partial\vartheta_1(\ell_{F}(x,s))
    +\left(\begin{matrix}
   \nabla G(s)z+\partial h(x) \\ [D^2G(s)(x-s,\cdot)]^*z
     \end{matrix}\right)\\
   \ell_{G}(x,s)-\partial\vartheta_2^*(-z)
  \end{matrix}\right).
\end{equation}
By comparing \eqref{subdiff-wXi} with the expression of $\partial\Xi$ in Lemma \ref{subdiff-LemXi}, we see that $(\overline{x},\overline{x},\overline{z})$ is a critical point of $\widetilde{\Xi}$ if and only if $(\overline{x},\overline{x},\overline{z},\mathcal{Q})$ for a certain PD linear mapping $\mathcal{Q}\!:\mathbb{X}\to\mathbb{X}$ is a critical point of $\Xi$. By combining \eqref{subdiff-wXi} with Definition \ref{spoint-def}, if $\overline{x}$ is a stationary point of \eqref{prob}, there exists $\overline{z}\in\partial(-\vartheta_2)(G(\overline{x}))$ such that $(\overline{x},\overline{x},\overline{z})$ is a critical point of $\widetilde{\Xi}$; and if $(\overline{x},\overline{x},\overline{z})$ is a critical point of $\widetilde{\Xi}$ and $\overline{z}\in\partial(-\vartheta_2)(G(\overline{x}))$, then $\overline{x}$ is a stationary point of \eqref{prob}. The following proposition provides a condition for the KL property of $\widetilde{\Xi}$ and $\Xi$ with exponent $p\in[1/2,1)$ by using that of an almost separable nonsmooth function.  
\begin{proposition}\label{prop-KL}
 Consider any $\overline{\omega}=(\overline{x},\overline{x},\overline{z})\in(\partial\widetilde{\Xi})^{-1}(0)$. Suppose that the mappings $F'$ and $G'$ are continuously differentiable on an neighborhood of $\overline{x}$, that the following function
 \begin{equation}\label{ffun}
  f(y,s,u,z):=\vartheta_1(y)+h(s)+\langle u,z\rangle+\vartheta_2^*(-z)\ \ {\rm for}\  (y,s,u,z)\in\mathbb{Y}\times\mathbb{X}\times\mathbb{Z}\times\mathbb{Z}
 \end{equation}
 satisfies the KL property of exponent $p\in[0,1)$ at $(\overline{y},\overline{x},\overline{u},\overline{z})$ with $\overline{y}=F(\overline{x})$ and $\overline{u}=G(\overline{x})$, and that 
 \begin{align}\label{key-cond}
 {\rm Ker}[\nabla F(\overline{x})\ \ \mathcal{I}\ \ \nabla G(\overline{x})]\cap\Big[\limsup_{y\to\overline{y}}{\rm pos}(\partial\vartheta_1(y))\times&\limsup_{x\xrightarrow[h]{}\overline{x}}{\rm pos}(\partial h(x))\times\limsup_{\zeta\to\overline{u}}(-\partial\vartheta_2(\zeta))\Big]\nonumber\\
 &=\!\big\{(0,0,0)\big\},
 \end{align}
 where, for a set $C$, ${\rm pos}(C):=\big\{\lambda x\,|\,x\in C,\lambda\ge 0\big\}$ is the positive hull of $C$. Then, the function $\widetilde{\Xi}$ satisfies the KL property of exponent $p$ at $\overline{\omega}$, so does the function $\Xi$ at  $(\overline{\omega},\overline{\mathcal{Q}})$ for any PD linear mapping $\overline{\mathcal{Q}}$ from $\mathbb{X}$ to $\mathbb{X}$ if $p\in[1/2,1)$.
\end{proposition}
\begin{proof}\!\!.
 Suppose on the contrary that $\widetilde{\Xi}$ does not have the KL property of exponent $p$ at $\overline{\omega}$. By Definition \ref{KL-def}, there exists a sequence $\{\omega^k\}_{k\in\mathbb{N}}$ with $\omega^k=(x^k,s^k,z^k)\to\overline{\omega}$ as $k\to\infty$ and $\widetilde{\Xi}(\overline{\omega})<\widetilde{\Xi}(\omega^k)<\widetilde{\Xi}(\overline{\omega})+1/k$ for each $k$ such that the following inequality holds with $y^k:=\ell_{F}(x^k,s^k)$ and $u^k:=\ell_{G}(x^k,s^k)$:  
 \[
  {\rm dist}\big(0,\partial\widetilde{\Xi}(x^k,s^k,z^k)\big)
  <(1/k)\big(f(y^k,x^k,u^k,z^k)-f(\overline{y},\overline{x},\overline{u},\overline{z})\big)^{p}.
 \]
 From $\widetilde{\Xi}(\omega^k)<\widetilde{\Xi}(\overline{\omega})+1/k$, we infer that $\{(x^k,s^k,z^k)\}_{k\in\mathbb{N}}\subset{\rm dom}\,\widetilde{\Xi}$. Combining the above inequality with \eqref{subdiff-wXi}, for each $k\in\mathbb{N}$, there exist $v^k\in\partial\vartheta_1(y^k)$, $\xi^k\in\partial h(x^k)$ and $\zeta^k\in\partial\vartheta_2^*(-z^k)$ such that
 \begin{equation}\label{ineq1-KL}
  \left\|\left(\begin{matrix}
      \nabla F(s^k)v^k+\xi^k+\nabla G(s^k)z^k\\
   [D^2F(s^k)(x^k-s^k,\cdot)]^*v^k+[D^2G(s^k)(x^k-s^k,\cdot)]^*z^k\\	
			u^k-\zeta^k
  \end{matrix}\right)\right\|<\frac{1}{k}\Big(f(y^k,x^k,u^k,z^k)-\overline{f}\Big)^{p}
 \end{equation}
 with $\overline{f}=f(\overline{y},\overline{x},\overline{u},\overline{z})$.
 Since $f$ has the KL property of exponent $p$ at $(\overline{y},\overline{x},\overline{u},\overline{z})$, there exist $\delta'>0,\eta'>0$ and $c'>0$ such that for all $(y,x,u,z)\in\mathbb{B}((\overline{y},\overline{x},\overline{u},\overline{z}),\delta')\cap[\overline{f}<f<\overline{f}+\eta']$,
 \begin{equation}\label{vartheta-KL}
 {\rm dist}(0,\partial\!f(y,x,u,z))\ge c'\big(f(y,x,u,z)-\overline{f}\big)^{p}.
 \end{equation}
 From $\omega^k\to\overline{\omega}$ as $k\to\infty$, we have $x^k\in\mathcal{V}_{x}$ for sufficiently large $k$, where $\mathcal{V}_x$ is the same as the one in Assumption \ref{ass2}. Thus, by invoking Assumption \ref{ass0} and recalling that $\widetilde{\Xi}(\overline{\omega})<\widetilde{\Xi}(\omega^k)<\widetilde{\Xi}(\overline{\omega})+(1/k)$ for each $k\in\mathbb{N}$, it follows that  $(y^k,x^k,u^k,z^k)\in\mathbb{B}((\overline{y},\overline{x},\overline{u},\overline{z}),\delta')\cap[\overline{f}<f<\overline{f}+\eta']$ for all $k$ large enough. Together with the above inequalities \eqref{vartheta-KL} and \eqref{ineq1-KL}, it follows that for sufficiently large $k$,
  \begin{align}\label{ineq2-KL}
   &\left\|\left(\begin{matrix}
    \nabla F(s^k)v^k+\xi^k+\nabla G(s^k)z^k\\
   [D^2F(s^k)(x^k-s^k,\cdot)]^*v^k+[D^2G(s^k)(x^k-s^k,\cdot)]^*z^k\\		
			u^k-\zeta^k
     \end{matrix}\right)\right\|\nonumber\\
   &<\frac{1}{kc'}{\rm dist}\big(0,\partial\!f(y^k,x^k,u^k,z^k)\big)
   \le\frac{1}{kc'}\|(v^k,\xi^k,z^k,u^k\!-\zeta^k)\|,
  \end{align}
 where the second inequality is obtained by using $(v^k,\xi^k,z^k,u^k-\zeta^k)\in\partial\!f(y^k,x^k,u^k,z^k)$, implied by  $v^k\in\partial\vartheta_1(y^k)$, $\xi^k\in\partial h(x^k)$ and $\zeta^k\in\partial\vartheta_2^*(-z^k)$ for each $k\in\mathbb{N}$. Combining  $(v^k,\xi^k,z^k,u^k-\zeta^k)\in\partial\!f(y^k,x^k,u^k,z^k)$ with \eqref{vartheta-KL} yields that $\|(v^k,\xi^k,z^k,u^k-\zeta^k)\|\ge c'\big(f(y^k,x^k,u^k,z^k)-\overline{f}\big)^{p}>0$. For sufficiently large $k$, by letting $(\widetilde{v}^k,\widetilde{\xi}^k,\widetilde{z}^k,\widetilde{\eta}^k):=\frac{(v^k,\xi^k,z^k,u^k-\zeta^k)}{\|(v^k,\xi^k,z^k,u^k-\zeta^k)\|}$, from \eqref{ineq2-KL} it follows that 
  \begin{equation}\label{ineq3-KL}
   \left\|\left(\begin{matrix}
   \nabla F(s^k)\widetilde{v}^k+\widetilde{\xi}^k+\nabla G(s^k)\widetilde{z}^k\\
    [D^2F(s^k)(x^k-s^k,\cdot)]^*\widetilde{v}^k+[D^2G(s^k)(x^k-s^k,\cdot)]^*\widetilde{z}^k\\ \widetilde{\eta}^k
   \end{matrix}\right)\right\|\le\frac{1}{kc'}.
  \end{equation}
  From the boundedness of $\{(\widetilde{v}^k,\widetilde{\xi}^k,\widetilde{z}^k,\widetilde{\eta}^k)\}_{k\in\mathbb{N}}$, there necessarily exists an index set $\mathcal{K}\subset\mathbb{N}$ such that the subsequence $\{(\widetilde{v}^k;\widetilde{\xi}^k,\widetilde{z}^k,\widetilde{\eta}^k)\}_{k\in\mathcal{K}}$ is convergent with the limit $(\widetilde{v},\widetilde{\xi},\widetilde{z},\widetilde{\eta})$ satisfying $\|(\widetilde{v},\widetilde{\xi},\widetilde{z},\widetilde{\eta})\|\!=\!1$.
	
  We next argue that the sequence $\{(\xi^k,u^k-\zeta^k)\}_{k\in\mathbb{N}}$ is bounded. If not, by noting that $\{(v^k,z^k)\}_{k\in\mathbb{N}}$ is bounded, we infer from the unboundedness of $\{(\xi^k,u^k-\zeta^k)\}_{k\in\mathbb{N}}$ that $\widetilde{v}=0$, $\widetilde{z}=0$ and $\|(\widetilde{\xi},\widetilde{\eta})\|=1$. However, passing the limit $\mathcal{K}\ni k\to\infty$ to the inequality \eqref{ineq3-KL} and using $\widetilde{v}=0,\widetilde{z}=0$ leads to $\widetilde{\xi}=\widetilde{\eta}=0$, a contradiction to $\|(\widetilde{\xi},\widetilde{\eta})\|=1$. Thus, $\{(\xi^k,u^k-\zeta^k)\}_{k\in\mathbb{N}}$ is bounded. Note that $u^k\to \overline{u}$ as $k\to\infty$, so the sequence $\{\zeta^k\}_{k\in\mathbb{N}}$ is bounded. If necessary by taking a subsequence, we assume $\lim_{\mathcal{K}\ni k\to\infty}\zeta^k=\overline{\zeta}$. In view of the expression of $(\widetilde{v}^k,\widetilde{\xi}^k,\widetilde{z}^k)$ and $z^k\in-\partial\vartheta_2(\zeta^k)$, for sufficiently large $k$,
  \[
   (\widetilde{v}^k,\widetilde{\xi}^k,\widetilde{z}^k)\in{\rm pos}(\partial\vartheta_1(y^k))\times{\rm pos}(\partial h(x^k))\times{\rm pos}(-\partial\vartheta_2(\zeta^k)).
  \]
  Also, from Assumption \ref{ass0} (ii)-(iii), $\vartheta_1(y^k)\to\vartheta_1(\overline{y})$ and $\vartheta_2(\zeta^k)\to\vartheta_2(\overline{\zeta})$, which along with $\widetilde{\Xi}(\omega^k)\to\widetilde{\Xi}(\overline{\omega})$ implies $h(x^k)\to h(\overline{x})$. By the definition of outer limits of multifunctions (see \cite[Chapter 5.B]{RW98}), 
 \[  (\widetilde{v},\widetilde{\xi},\widetilde{z})\in\limsup_{y\to\overline{y}}{\rm pos}(\partial\vartheta_1(y))\times\limsup_{x\xrightarrow[h]{}\overline{x}}{\rm pos}(\partial h(x))\times\limsup_{\zeta\to\overline{\zeta}}(-\partial\vartheta_2(\zeta)).
 \]
 In addition, passing the limit $\mathcal{K}\ni k\to\infty$ to the inequality \eqref{ineq3-KL} yields that $ \nabla\!F(\overline{x})\widetilde{v}+\widetilde{\xi}+\nabla\!G(\overline{x})\widetilde{z}=0$ and $\widetilde{\eta}=0$. By using $\widetilde{\eta}=0$ and the expression of $\widetilde{\eta}^k$ and recalling that $u^k\to\overline{u}$ as $k\to\infty$, it is not hard to deduce that $\overline{\zeta}=\overline{u}$. Together with the above inclusion, it follows that
 \[
  0\!\ne\!(\widetilde{v},\widetilde{\xi},\widetilde{z})\!\in\!{\rm Ker}[\nabla\!F(\overline{x})\  \mathcal{I}\ \nabla G(\overline{x})]\cap\Big[\limsup_{y\to\overline{y}}{\rm pos}(\partial\vartheta_1(y))\times\limsup_{x\xrightarrow[h]{}\overline{x}}{\rm pos}(\partial h(x))\times\limsup_{\zeta\to\overline{u}}(-\partial\vartheta_2(\zeta))\Big],
 \]
 a contradiction to the assumption in \eqref{key-cond}. Therefore, $\widetilde{\Xi}$ satisfies the KL property of exponent $p$ at $\overline{\omega}$. 
 
 Notice that $\Xi(w)=\widetilde{\Xi}(x,s,z)+\|x-s\|_\mathcal{Q}^2+\chi_{\mathbb{S}_{+}}(\mathcal{Q})$ for $w=(x,s,z,\mathcal{Q})\in\mathbb{W}$, and $(\overline{\omega},\overline{\mathcal{Q}})$ with $\overline{\mathcal{Q}}\succ\underline{\gamma}\mathcal{I}$ is a critical point of $\Xi$. When $p\in[1/2,1)$, following the same arguments as those for \cite[Lemma 1]{LiuPanWY22}, we can prove that $\Xi$ satisfies the KL property of exponent $p$ at $(\overline{\omega},\overline{\mathcal{Q}})$. We complete the proof.
\end{proof}
\begin{remark}\label{remark42-KL}
 {\bf(a)} From the proof of Proposition \ref{prop-KL}, we see that its conclusions do not require the convexity of $\vartheta_1,\,\vartheta_2$ and $h$. The outer limit sets appearing in \eqref{key-cond} can be characterized especially when $\vartheta_1,h$ and $\vartheta_2$ are separable. Therefore, the condition \eqref{key-cond} is checkable. 
 
 \noindent
 {\bf(b)} Observe that $\widetilde{\Xi}=f\circ H$, where  $H:\mathbb{X}\times\mathbb{X}\times\mathbb{Z}\to\mathbb{Y}\times\mathbb{X}\times\mathbb{Z}\times\mathbb{Z}$ is a mapping defined by 
 \[
   H(x,s,z):=(\ell_{F}(x,s);x;\ell_{G}(x,s);z).
 \]
 It is easy to obtain that ${\rm Ker}\,\nabla H(\overline{x},\overline{x},\overline{z})={\rm Ker}\,[\nabla F(\overline{x})\ \ \mathcal{I}\ \ \nabla G(\overline{x})]$. In addition, from the definition of horizon subdifferential $\partial^{\infty}h(\overline{x})$ (see \cite[Definition 8.3]{RW98}), it is not hard to check that $\partial^{\infty}h(\overline{x})\subset\limsup_{x\xrightarrow[h]{}\overline{x}}{\rm pos}(\partial h(x))$ and the inclusion is strict. Then, the condition \eqref{key-cond} is stronger than 
 \[
   {\rm Ker}\,\nabla H(\overline{x},\overline{x},\overline{z})\cap[\partial^{\infty}\vartheta_1(\overline{y})\times\partial^{\infty}h(\overline{x})\times(-\partial^{\infty}\vartheta_2(\overline{u}))]=\{(0,0,0)\},
 \]
 which is precisely the metric regularity constraint qualification for the system $H(x,s,z)\in{\rm dom}\,f$ at $(\overline{x},\overline{x},\overline{z})$. Notice that Li and Pong proposed a criterion in \cite[Theorem 3.2]{LiPong18} to identify the KL property of exponent $p\in[0,1)$ for a general composite function, which is applicable to the function $\widetilde{\Xi}$. When $f$ has the KL property of exponent $p\in[0,1)$ at $(\overline{y},\overline{x},\overline{u},\overline{z})$, their criterion requires the surjectivity of $H'(\overline{x},\overline{x},\overline{z})$ or ${\rm Ker}\,[\nabla F(\overline{x})\ \ \mathcal{I}\ \ \nabla G(\overline{x})]=\{(0,0,0)\}$ to ensure that $\widetilde{\Xi}$ satisfies the KL property of exponent $p$ at $\overline{\omega}$. Since the outer limit sets in \eqref{key-cond} are generally not the whole space, the condition in Proposition \ref{prop-KL} for $\widetilde{\Xi}$ to have the KL property of exponent $p$ at $\overline{\omega}$ is weaker than that of \cite[Theorem 3.2]{LiPong18}. 
    
 \noindent
 {\bf(c)} Let $\psi(u,z):=\langle u,z\rangle+\vartheta_2^*(-z)$ for $(u,z)\in\mathbb{Z}\times\mathbb{Z}$. By the expression of $f$ in \eqref{ffun}, when $\psi,\vartheta_1$ and $h$ are the KL functions of exponent $p\in[1/2,1)$, the function $f$ is a KL function of exponent $p\in[1/2,1)$. By \cite[Proposition 2.2 (i) $\&$ Remark 2.2]{LiuPanWY22}, when $\psi,\,\vartheta_1$ and $h$ are KL functions, their KL property of exponent $p\in[1/2,1)$ is implied by the $1/(2p)$-subregularity of their subdifferential mappings. Thus, by invoking \cite[Proposition 1]{Robinson81}, we conclude that $\vartheta_1$ and $h$ are KL functions of exponent $p\in[1/2,1)$ if they are piecewise linear-quadratic (PLQ) KL functions, and $\psi$ is a KL function of exponent $p\in[1/2,1)$ if $\vartheta_2$ is a PLQ function and $\psi$ is a KL function. When they are not PLQ functions, one can use the criteria in \cite{Geferee11} to check the subregularity of their subdifferential mappings at the reference point.

 \noindent
 {\bf(d)} Combining part (b) with part (c), we conclude that $\widetilde{\Xi}$ satisfies the KL property of exponent $[1/2,1)$ at $\overline{\omega}=(\overline{x},\overline{x},\overline{z})$ when $\vartheta_1,h$ are PLQ functions with KL property, $\vartheta_2$ is a PLQ definable function in an o-minimal structure over the real field, and the condition \eqref{key-cond} holds. In this case, the mapping $\partial\widetilde{\Xi}$ is metrically subregular at $(\overline{\omega},0)$ if in addition $\overline{\omega}$ is a local minimizer of $\widetilde{\Xi}$, by noting that $\widetilde{\Xi}$ has the composite form considered in \cite{LiuPanWY22} and invoking \cite[Proposition 2.2 c]{LiuPanWY22}. To the best of our knowledge, there are no convenient rules for identifying the subregularity of  $\partial\widetilde{\Xi}$ at $(\overline{\omega},0)$.

\end{remark}
\begin{corollary}\label{corollary-KL}
 Consider \eqref{prob} with $\vartheta_2\equiv 0$ and any  $(\overline{x},\overline{x})\in(\partial\widetilde{\Xi})^{-1}(0)$ with $\widetilde{\Xi}(x,s)=\vartheta_1(\ell_{F}(x,s))+h(x)$ for $(x,s)\in\mathbb{X}\times\mathbb{X}$. Suppose that $F'$ is continuously differentiable on an neighborhood of $\overline{x}$, that 
 \begin{equation}\label{ffun1}
  f(y,s):=\vartheta_1(y)+h(s)\ \ {\rm for}\  (y,s)\in\mathbb{Y}\times\mathbb{X}
 \end{equation}
 has the KL property of exponent $p\in[0,1)$ at $(\overline{y},\overline{x})$ with $\overline{y}=F(\overline{x})$, and that the following condition holds
 \begin{equation}\label{key-cond1}
  {\rm Ker}\,[\nabla\!F(\overline{x})\ \ \mathcal{I}]\cap\big[\limsup_{y\to\overline{y}}{\rm pos}(\partial\vartheta_1(y))\times\limsup_{x\xrightarrow[h]{}\overline{x}}{\rm pos}(\partial h(x))\big]=\big\{(0,0)\big\}.
 \end{equation}
 Then, the function $\widetilde{\Xi}$ satisfies the KL property of exponent $p$ at $(\overline{x},\overline{x})$, so does the corresponding $\Xi$ at $(\overline{x},\overline{x},\overline{\mathcal{Q}})$ for any PD linear mapping $\overline{\mathcal{Q}}$ if $p\in[1/2,1)$.
\end{corollary}

Note that the function $f$ defined in \eqref{ffun1} has the KL property of exponent $1/2$ at $(F(\overline{x}),\overline{x})$ iff $\vartheta_1$ and $h$ have the KL property of exponent $1/2$ at $F(\overline{x})$ and $\overline{x}$, respectively. By combining \cite[Theorem 5 (ii)]{Bolte17} and \cite[Theorem 3.3]{Artacho08}, the KL property of $\vartheta_1$ with exponent $1/2$ at $\overline{y}=F(\overline{x})$ is equivalent to the subregularity of $\partial\vartheta_1$ at $(\overline{y},0)$, and that of $h$ is equivalent to the subregularity of $\partial h$ at $(\overline{x},0)$. Such a condition is equivalent to requiring that $\mathop{\arg\min}_{(y,s)\in\mathbb{Y}\times\mathbb{X}}f(y,s)$ is the set of local weak sharp minima of order 2 for $f$, which for $h\equiv 0$ is the one used in \cite[Theorem 20]{HuYang16} to achieve a local R-linear rate. 

Next we take a closer look at the condition \eqref{key-cond1} and discuss its relation with the regularity used in \cite{HuYang16} for the problem \eqref{prob} with $\vartheta_2\equiv 0$. For this purpose, for a closed convex set $S\subset\mathbb{Y}$, we denote by $S^{\ominus}$ its negative polar, i.e., $S^{\ominus}:=\{y^*\in\mathbb{Y}\ |\ \langle y^*,y\rangle\le 0\ {\rm for\ all}\ y\in S\}$. Define the sets
\[
 C:=\mathop{\arg\min}_{y\in\mathbb{Y}}\vartheta_1(y)\ \ {\rm and}\ \ D:=\mathop{\arg\min}_{x\in\mathbb{X}}h(x). 
 \]
 We claim that $\limsup_{y\to\overline{y}}{\rm pos}(\partial\vartheta_1(y))\subset [C-\overline{y}]^{\ominus}$ and $\limsup_{x\xrightarrow[h]{}\overline{x}}{\rm pos}(\partial h(x))\subset [D-\overline{x}]^{\ominus}$. Indeed, pick any $\xi\in\limsup_{y\to\overline{y}}{\rm pos}(\partial\vartheta_1(y))$. There exist $y^k\to\overline{y}$ and $\xi^k\to\xi$ with $\xi^k\in{\rm pos}(\partial\vartheta_1(y^k))$ for each $k\in\mathbb{N}$. Obviously, for each $k\in\mathbb{N}$, there exists $t_k\ge 0$ and $v^k\in\partial\vartheta_1(y^k)$ such that $\xi^k=t_kv^k$. Then, for each $z\in C$, it holds 
$0\ge \vartheta_1(z)-\vartheta_1(y^k)\ge \langle v^k,z-y^k\rangle$ for each $k\in\mathbb{N}$, which implies $\langle\xi^k,z-y^k\rangle\le 0$ for each $k\in\mathbb{N}$. Passing the limit $k\to\infty$ to this inequality leads to $\langle \xi,z-\overline{y}\rangle\le 0$. Consequently, $\xi\in[C-\overline{y}]^{\ominus}$ follows the arbitrariness of $z\in C$. Thus, $\limsup_{y\to\overline{y}}{\rm pos}(\partial\vartheta_1(y))\subset [C-\overline{y}]^{\ominus}$. Similarly, the inclusion $\limsup_{x\xrightarrow[h]{}\overline{x}}{\rm pos}(\partial h(x))\subset [D-\overline{x}]^{\ominus}$ also holds. It is worth pointing out that the inclusion is generally strict; for example, consider $\vartheta_1(y)=|y_1|+y_2^2$ for $y\in\mathbb{R}^2$ and $F(x)\!=\!(x,-x)^{\top}$ for $x\in\mathbb{R}$. For $\overline{x}\!=\!1$, we have $\limsup_{y\to\overline{y}}{\rm pos}(\partial\vartheta_1(y))\subset\mathbb{R}_{+}\times\mathbb{R}_{-}$, but $[C\!-\!F(\overline{x})]^{\ominus}\!=\!\{y\in\mathbb{R}^2\ |\ y_2-y_1\le 0\}$. Thus, from the above discussion, we conclude that  the condition \eqref{key-cond1} is weaker than the following one
 \[
  {\rm Ker}\big([\nabla\!F(\overline{x})\ \ \mathcal{I}]\big)\cap\big([C-F(\overline{x})]^{\ominus}\times[D-\overline{x}]^{\ominus}\big)=\big\{(0,0)\big\},
 \]
which is precisely the regularity condition used in \cite[Theorem 18]{HuYang16} and \cite[Section 3]{Burke95}.

 Suppose that $h\equiv 0$. For any given $x\in\mathbb{X}$, let $D(x):=\{d\in\mathbb{X}\,|\,F(x)+F'(x)d=0\}$. From \cite[Definition 7 (c)]{HuYang16}, a vector $\overline{x}\in\mathbb{X}$ is called a quasi-regular point of inclusion $F(x)\in C$ if there exist $r>0$ and $\beta_{r}>0$ such that for all $x\in\mathbb{B}(\overline{x},r)$, ${\rm dist}(0,D(x))\le\beta_{r}{\rm dist}(F(x),C)$. Now there is no clear implication relation between the condition \eqref{key-cond1} and the quasi-regularity condition. When $F$ is polyhedral but $C\ne\{0\}$, there are examples for which the quasi-regularity condition does not hold; for example, $\overline{x}=1$, $F(x)=(x;0)^{\top}$ for $x\in\mathbb{R}$ and $\vartheta_1(y)=(y_1\!-\!1)^4+|y_2|$ for $y\in\mathbb{R}^2$. The condition \eqref{key-cond1} does not hold for this example either. However, when $F$ is polyhedral and $\vartheta_1$ is a PLQ convex function, $\widetilde{\Xi}$ is necessarily a KL function of exponent $1/2$ by \cite[Theorem 5 (ii)]{Bolte17}, \cite[Theorem 3.3]{Artacho08} and \cite[Proposition 1]{Robinson81}, so is the function $\Xi$ by the second part of Corollary \ref{corollary-KL}. When $F$ is non-polyhedral, there are examples for which the condition \eqref{key-cond1} holds but the quasi-regularity condition does not hold; for instance, $F(x)=(x^2,-x)^{\top}$ for $x\in\mathbb{R}$ and $\vartheta_1(y)=\|y\|^2$ for $y\in\mathbb{R}^2$. It is easy to check that \eqref{key-cond1} holds at  $\overline{x}=0$, which along with the strong convexity of $\vartheta_1$ shows that $\widetilde{\Xi}$ and $\Theta_1$ are the KL functions of exponent $1/2$, but $\overline{x}$ is not a quasi-regular point of the inclusion $F(x)\in C$ because $D(x)=\{d\in\mathbb{R}\ |\ (x^2+2xd;-x-d)=0\}=\emptyset$ for all $x\ne 0$ sufficiently close to $\overline{x}$. Now it is unclear whether there is an example with nonlinear $F$ for which the quasi-regularity holds but the condition \eqref{key-cond1} does not hold.
\section{Dual PPA armed with SNCG to solve subproblems}\label{sec5}

 The efficiency of Algorithm \ref{iLPA} depends heavily on the computation of subproblem \eqref{subprobkj}. In this section, we develop an efficient solver (named dPPASN) to compute subproblem \eqref{subprobkj} by applying the proximal point algorithm (PPA) armed with semismooth Newton to solve its dual. To this end, for a proper closed convex function $f:\mathbb{X}\to\overline{\mathbb{R}}$ and a constant $\tau>0$, let $\mathcal{P}_{\!\tau f}$ and $e_{\tau f}$ denote the proximal mapping and the Moreau envelope of $f$ associated with $\tau$, and for each $k\in\mathbb{N}$ and $j\in[j_k]$, introduce the notation:
 \begin{equation}\label{Akmap}
  \mathcal{A}_k:=F'(x^k),\,c^k:=F(x^k),\,b^k:=\mathcal{A}_kx^k-c^k,\,u^k:=\nabla G(x^k)\xi^k\ \ {\rm and}\ \   \mathcal{Q}_{k,j}:=\gamma_{k,j}\mathcal{I}+\alpha_k\mathcal{A}_k^*\mathcal{A}_k
 \end{equation}
 where $\alpha_k>0$ is specified in the experiments. The structure of $\mathcal{Q}_{k,j}$ in \eqref{Akmap} ensures that the subproblem \eqref{subprobkj} has a simple dual. Indeed, with the notation in \eqref{Akmap}, it can equivalently be written as 
 \begin{align}\label{Esubprobkj}
 &\min_{x\in\mathbb{X},y\in\mathbb{Y}}\!\vartheta_1(y)+\langle u^k,x-x^k\rangle+h(x)+({\gamma_{k,j}}/{2})\|x-x^k\|^2+({\alpha_k}/{2})\|y-c^k\|^2-\Theta_2(x^k) \nonumber\\
 &\ \ \ {\rm s.t.}\ \ \mathcal{A}_kx-b^k=y
\end{align}
whose dual, after an elementary calculation, takes the following form
 \begin{equation}\label{dEsubprobkj}
 \min_{\zeta\in\mathbb{Y}}\ \Psi_{k,j}(\zeta):=\frac{\|\zeta\|^2}{2\alpha_k}-e_{\!\alpha_k^{-1}\vartheta_1}\big(\alpha_k^{-1}\zeta+c^k\big)+\frac{\|\mathcal{A}_k^*\zeta+u^k\|^2}{2\gamma_{k,j}} -e_{\gamma_{k,j}^{-1}h}\big(x^k-\gamma_{k,j}^{-1}(\mathcal{A}_k^*\zeta+u^k)\big)+\Theta_2(x^k).
 \end{equation}
 Due to the strong convexity of \eqref{subprobkj} or \eqref{Esubprobkj}, the strong duality holds for \eqref{Esubprobkj} and \eqref{dEsubprobkj}. Since $\Psi_{k,j}$ is a smooth convex function with Lipschitz continuous gradient, seeking an optimal solution of the dual problem \eqref{dEsubprobkj} is equivalent to finding a root to the nonsmooth system
 \begin{equation}\label{system}
  0=\nabla\Psi_{k,j}(\zeta)=\mathcal{P}_{\!\alpha_k^{-1}\vartheta_1}(\alpha_k^{-1}\zeta+c^k)+b^k-\gamma_{k,j}^{-1}\mathcal{A}_k\mathcal{P}_{\!\gamma_{k,j}^{-1}h}(x^k-\gamma_{k,j}^{-1}(\mathcal{A}^*_{k}\zeta+u^k)).
 \end{equation}
 That is, with any $\zeta^*\in(\nabla\Psi_{k,j})^{-1}(0)$, one can recover the unique optimal solution $(\overline{x}^{k,j},\overline{y}^{k,j})$ of \eqref{Esubprobkj} via

 \begin{equation}\label{xbarkj}
  \overline{x}^{k,j}=\mathcal{P}_{\!\gamma_{k,j}^{-1}h}(x^k-\gamma_{k,j}^{-1}(\mathcal{A}^*_{k}\zeta^{*}+u^k)) \ \ {\rm and}\ \ \overline{y}^{k,j}=\mathcal{P}_{\!\alpha_k^{-1}\vartheta_1}(\alpha_k^{-1}\zeta^{*}+c^k).
 \end{equation}
 
 From \cite[Theorem 1]{Bolte09} and \cite[Proposition 3.1]{Ioffe09}, when $\vartheta_1$ and $h$ are definable in the same o-minimal structure on the real field $(\mathbb{R},+,\cdot)$, the system \eqref{system} is semismooth. However, a direct application of the semismooth Newton method to it faces the difficulty caused by the potential singularity of the generalized Hessian of $\Psi_{k,j}$. Inspired by this, we apply the inexact PPA armed with the semismooth Newton method to solving \eqref{dEsubprobkj}, whose iterations are described as follows.
\begin{algorithm}[h]
 \renewcommand{\thealgorithm}{A}
 \caption{\label{PPASN}{\bf\ Inexact dPPA with semismooth Newton (dPPASN)}}
 \textbf{1:} Initialization: Fix $k\in\mathbb{N}$ and $j\in[j_k]$. Choose $\tau_0>\underline{\tau}>0, \varsigma\in(0,1)$ and an initial 
 
 \ \ \  \  $\zeta^0\in\mathbb{Y}$.\\
 \textbf{2: For} {$l=0,1,2,\ldots$} \\
   \textbf{3:       }\textbf{    }\textbf{    }  Seek an inexact minimizer $\zeta^{l+1}$ of \eqref{PPA-subprob} with Algorithm \ref{SNCG} described later:
    \begin{equation}\label{PPA-subprob}
    \min_{\zeta\in\mathbb{Y}}\ \widetilde{\Psi}_{k,j}(\zeta):=\Psi_{k,j}(\zeta)+({\tau_{l}}/{2})\|\zeta-\zeta^{l}\|^2.
    \end{equation}
   		
  \textbf{4:    }\textbf{    }\textbf{    }  Update the parameter $\tau_{l+1}$ by $\tau_{l+1}=\min\{\underline{\tau},\varsigma\tau_{l}\}$. \\
 \textbf{5: end (For)}
\end{algorithm}

 From the weak duality theorem between \eqref{Esubprobkj} and \eqref{dEsubprobkj}, for each $l\in\mathbb{N}$, $-\Psi_{k,j}(\zeta^{l})\le q_{k,j}(\overline{x}^{k,j})$. This means that, when Algorithm \ref{PPASN} is applied to solve the dual \eqref{dEsubprobkj}, if some iterate $\zeta^{l}$ is such that  
 \begin{equation}\label{stop-cond-dPPASN}
 q_{k,j}(x^{k,j}(\zeta^{l}))+\Psi_{k,j}(\zeta^{l})<\frac{\mu_k}{2}\|x^{k,j}-x^k\|^2\ {\rm for}\ x^{k,j}(\zeta^{l}):=\!\mathcal{P}_{\!\gamma_{k,j}^{-1}h}(x^k-\gamma_{k,j}^{-1}(\mathcal{A}_{k}^*\zeta^{l}+u^k)),
 \end{equation}
 then 
 $x^{k,j}(\zeta^{l})$ and $q_{k,j}^{\rm LB}=-\Psi_{k,j}(\zeta^{l})$ satisfy the inexactness condition \eqref{inexact-cond}. Inspired by this, we use \eqref{stop-cond-dPPASN} as the stop condition of Algorithm \ref{PPASN} for numerical tests. It is worth pointing out that, when the sequence ${\rm dist}(\zeta^{l},(\nabla\Psi_{k,j})^{-1}(0))$ yielded by Algorithm \ref{PPASN} converges to $0$ with a Q-linear rate, the corresponding $\{x^{k,j}(\zeta^{l})\}_{l\in\mathbb{N}}$ converges to $\overline{x}^{k,j}$ defined in \eqref{xbarkj} with a R-linear rate by noting that  
 \[
  \|x^{k,j}(\zeta^{l+1})-\overline{x}^{k,j}\|\le \gamma_{k,j}^{-1}\|\mathcal{A}_k^*\|\,{\rm dist}(\zeta^{l+1},(\nabla\Psi_{k,j})^{-1}(0)). 
 \]
 From \cite[Theorem 1]{Roc76} and \cite[Theorem 2.1]{Luque84}, we obtain the following convergence result. 
\begin{theorem}\label{theorem-AlgA}
 Fix any $k\in\mathbb{N}$ and $j\in[j_k]$. For each $l\in\mathbb{N}$, let $x^{k,j}(\zeta^{l})$ be defined as in \eqref{stop-cond-dPPASN}, where $\{\zeta^{l}\}_{l\in\mathbb{N}}$ is the sequence generated by Algorithm \ref{PPASN} with the following inexactness criterion
 \[
  \|\nabla\widetilde{\Psi}_{k,j}(\zeta^{l+1})\|\le\delta_l\tau_l\|\zeta^{l+1}-\zeta^{l}\|\quad{\rm with}\ \sum_{l=1}^{\infty}\delta_{l}<\infty.
 \]
 Then, the sequence $\{\zeta^{l}\}_{l\in\mathbb{N}}$ converges to some $\zeta^{k,*}\in(\nabla\Psi_{k,j})^{-1}(0)$. If the mapping $\nabla\Psi_{k,j}$ is metrically subregular at $\zeta^{k,*}$ for $0$ with modulus $\kappa$, then the sequence ${\rm dist}(\zeta^{l},(\nabla\Psi_{k,j})^{-1}(0))$ converges to $0$ with a Q-linear rate bounded from above by $\kappa/\sqrt{\kappa^2+\underline{\tau}^{-2}}$, and the sequence $\{x^{k,j}(\zeta^{l})\}_{l\in\mathbb{N}}$ converges to $\overline{x}^{k,j}$ in \eqref{xbarkj} with a R-linear rate. 
\end{theorem}
\begin{remark}\label{remark-AlgA}
 When $\vartheta_1$ and $h$ are the PLQ convex functions, the graph of $\partial\vartheta_1$ and $\partial h$ are the union of finitely many polyhedral convex sets, so are the mappings $\mathcal{P}_{\!\alpha_k^{-1}\vartheta_1}$ and $\mathcal{P}_{\!\gamma_{k,j}^{-1}h}$. By the expression of $\nabla\Psi_{k,j}$ in \eqref{system}, the  graph of $(\nabla\Psi_{k,j})^{-1}$ is also the union of finitely many polyhedral convex sets, so is locally upper Lipschitzian at $0$ by \cite[Proposition 1]{Robinson81}. Then, $\nabla\Psi_{k,j}$ is metrically subregular at $\zeta^{k,*}$ for $0$. 
\end{remark}

 Since $\widetilde{\Psi}_{k,j}$ in \eqref{PPA-subprob} has a Lipschitz continuous gradient mapping, we define its generalized Hessian at $\zeta$ by the Clarke Jacobian of $\nabla\widetilde{\Psi}_{k,j}$ at $\zeta$, i.e., $\partial^2\widetilde{\Psi}_{k,j}(\zeta)\!:=\partial_{C}(\nabla\widetilde{\Psi}_{k,j})(\zeta)$. From \cite[Page 75]{Clarke83}, it follows that $\partial^2\widetilde{\Psi}_{k,j}(\zeta)d=\widehat{\partial}^2\widetilde{\Psi}_{k,j}(\zeta)d$ for all $d\in\mathbb{Y}$ with
 \[
  \widehat{\partial}^2\widetilde{\Psi}_{k,j}(\zeta)
 \!:=\!\tau_{l}\mathcal{I}+\alpha_k^{-1}\partial_{C}\mathcal{P}_{\!\alpha_k^{-1}\vartheta_1}\big(\alpha_k^{-1}\zeta+c^k\big)
+\gamma_{k,j}^{-1}\mathcal{A}_k\partial_{C}\mathcal{P}_{\!\gamma_{k,j}^{-1}h}
\big(x^k\!-\!\gamma_{k,j}^{-1}(\mathcal{A}_k^*\zeta+u^k)\big)\mathcal{A}_k^{*}.
\]
By mimicking the proof in \cite[Section 3.3.4]{Ortega70}, every $\mathcal{U}\in\partial_{C}\mathcal{P}_{\!\alpha_k^{-1}\vartheta_1}\big(\alpha_k^{-1}\zeta+c^k\big)$ and every $\mathcal{V}\in\partial_{C}\mathcal{P}_{\!\gamma_{k,j}^{-1}h}
\big(x^k-\gamma_{k,j}^{-1}(\mathcal{A}_k^*\zeta+u^k)\big)$ are positive semidefinite. Along with $\tau_l>0$ for each $l\in\mathbb{N}$, the operator $\tau_l\mathcal{I}+\alpha_{k}^{-1}\mathcal{U}+\gamma_{k,j}^{-1}\mathcal{A}_k\mathcal{V}\mathcal{A}_k^*$ is positive definite, so every element in $\widehat{\partial}^2\widetilde{\Psi}_{k,j}(\zeta)$ is nonsingular. Thus, Algorithm \ref{SNCG} described as follows is well defined and has the following convergence result. 

\begin{theorem}
 (see \cite[Theorem 3.5]{ZhaoST10}) Fix any $k\in\mathbb{N},j\in[j_k]$ and $\tau_l\in[\underline{\tau},\tau_0]$. Let $\overline{\zeta}^{l+1}$ be an accumulation point of the infinite sequence $\{\zeta^{\nu}\}_{\nu\in\mathbb{N}}$ generated by Algorithm \ref{SNCG} for solving \eqref{PPA-subprob}. Then the whole sequence $\{\zeta^{\nu}\}_{\nu\in\mathbb{N}}$ converges to $\overline{\zeta}^{l+1}$ and $\|\zeta^{\nu+1}-\overline{\zeta}^{l+1}\|=O(\|\zeta^{\nu}-\overline{\zeta}^{l+1}\|^{1+\varsigma})$.
\end{theorem}

\begin{algorithm}[h]
 \renewcommand{\thealgorithm}{A.1}
 \caption{\label{SNCG}{\bf\ Semismooth Newton method}}
 \textbf{1:}  Initialization: Fix $k\in\mathbb{N},j\in[j_k]$ and $\tau_l\in[\underline{\tau},\tau_0]$. Choose  $\underline{\eta},\beta\in(0,1),\varsigma\in(0,1]$ and 
 
 \ \ \ \ $0<c_1<c_2<\frac{1}{2}$. 
	\\
 \textbf{2: For} {$\nu=0,1,2,\ldots$} \\
 \textbf{3:       }\textbf{    }Choose  $\mathcal{U}^{\nu}\in\partial_{C}\mathcal{P}_{\!\alpha_k^{-1}\vartheta_1}\big(\alpha_k^{-1}\zeta^{\nu}+c^k\big)$
		and $\mathcal{V}^{\nu}\in\partial_{C}\mathcal{P}_{\!\gamma_{k,j}^{-1}h}\big(x^k-\gamma_{k,j}^{-1}(\mathcal{A}_k^*\zeta^{\nu}+u^k)\big)$.
		
\textbf{4:    }  Apply the practical conjugate gradient (CG) algorithm to find an approximate 

\ \ \ \ \  \ solution $d^{\nu}$ to
		\[
		\mathcal{W}^{\nu}d=-\nabla\widetilde{\Psi}_{k,l}(\zeta^{\nu})\ \ {\rm with}\ \ 		\mathcal{W}^{\nu}=\tau_l\mathcal{I}+\alpha_{k}^{-1}\mathcal{U}^{\nu}+\gamma_{k,j}^{-1}\mathcal{A}_k\mathcal{V}^{\nu}\mathcal{A}_k^*
		\]
	 \hspace{0.6cm} such that  $\|\mathcal{W}^{\nu}d^{\nu}+\nabla\widetilde{\Psi}_{k,l}(\zeta^{\nu})\|\le\min(\underline{\eta},\|\nabla\widetilde{\Psi}_{k,l}(\zeta^{\nu})\|^{1+\varsigma})$.

 	\textbf{5:   }\ \ Let $m_{\nu}$ be the smallest nonnegative integer $m$ such that
		\begin{align*}
		\widetilde{\Psi}_{k,l}(\zeta^{\nu}\!+\!\beta^md^{\nu})\leq\widetilde{\Psi}_{k,l}(\zeta^{\nu})+c_1\beta^m\langle\nabla\widetilde{\Psi}_{k,l}(\zeta^{\nu}),d^{\nu}\rangle,\\
			\vert\langle\nabla\widetilde{\Psi}_{k,l}(\zeta^{\nu}\!+\!\beta^md^{\nu}),d^{\nu}\rangle\vert\le c_2\vert\langle\nabla\widetilde{\Psi}_{k,l}(\zeta^{\nu}),d^{\nu}\rangle\vert.
		\end{align*}
		\textbf{6:   }\ \ Set $\zeta^{\nu+1}=\zeta^{\nu}+\beta^{m_{\nu}}d^{\nu}$. \\
 	\textbf{7: end (For)}
\end{algorithm}

 Note that the local convergence rate of $\{x^k\}_{k\in\mathbb{N}}$ has no direct relation with that of Algorithm A. To understand this better, we can consider Algorithm 1 without involving the inner for-end loop, which aims at seeking a tight upper estimation for ${\rm lip}\,\Theta(x^k)$ and then a good local majorization of $\Phi$ at $x^k$. Then, it is clear to see that the convergence rate of $\{x^k\}_{k\in\mathbb{N}}$ mainly depends on the KL property of the potential function $\Xi$ with exponent $p$. A smaller $p$ leads to a better convergence rate. Of course, the iteration complexity of Algorithm A has a great influence on the overall complexity of Algorithm 1.
Now let us take a closer look at the complexity for solving the dual problem \eqref{dEsubprobkj}. From the above discussions, the evaluations of $\widetilde{\Psi}_{k,j}$ and $\nabla\widetilde{\Psi}_{k,j}$ at $\xi$ are relatively simple since they directly call the outer $F(x^k),G(x^k)$ and $F'(x^k),G'(x^k)$ and do not involve any computation on $F$ and $G$ and their Jacobian matrices $F'$ and $G'$ at new points. In that sense, the computation cost for one step of dPPASN mainly comes from the iteration complexity of Algorithm \ref{SNCG}. Consider that Algorithm \ref{SNCG} admits a local superlinear convergence rate. The evaluation of a semismooth Newton step dominates the computation cost of one step of dPPASN. When the practical CG (see \cite[Algorithm 1]{ZhaoST10}) is applied to seek a semismooth Newton step, the computation complexity is about $O(N^2)$, where $N$ is the dimension of $d^{\nu}$.     

\section{Numerical experiments}\label{sec6}

We validate the efficiency of Algorithm \ref{iLPA} armed with dPPASN (iLPA, for short) by applying it to solve DC programs with nonsmooth components and matrix completion with outliers under non-uniform sampling. All tests are performed in MATLAB 2024a on a laptop computer running on 64-bit Windows Operating System with an Intel(R) Core(TM) i9-13905H CPU 2.60GHz and 32 GB RAM. 
 

\subsection{Choice of parameters $\underline{\gamma}$ and $\mu_k$}\label{sec6.1}

The parameter $\underline{\gamma}$, as a uniform lower bound for all $\gamma_{k,j}$, ensures the uniformly positive definiteness of all $\mathcal{Q}_{k,j}$ and is mainly used for theoretical analysis. In view of this, we choose $\underline{\gamma}=0.01$. Recall that the convergence analysis of Algorithm \ref{iLPA} requires the positive real number sequence $\{\mu_k\}_{k\in\mathbb{N}}$ to satisfy $\mu_k\in(0,\underline{\gamma}/5]$ for all $k\in\mathbb{N}$; see Remark \ref{remark-alg} (a). An immediate choice for it is $\mu_k=c_{\mu}k^{-1}$ with $c_{\mu}>0$. Figure \ref{figmuk1} shows how the value of $c_{\mu}$ affects the number of iterations of Algorithm \ref{iLPA} and the average number of iterations of Algorithm \ref{PPASN}. From the discussion at the end of Section \ref{sec5}, the latter actually reflects the average number of function evaluations at each iteration of Algorithm \ref{iLPA}. The curves in Figure \ref{figmuk1} are plotted by applying Algorithm \ref{iLPA} with $\mu_k=c_{\mu}k^{-1}$ to solve the problem \eqref{SCAD-loss}, where $\mathcal{A}$ and $b$ are generated randomly in the same way as in Section \ref{sec6.2} under noise of type I with $n_1=n_2=1000,r^*=15,SR=0.15$. We see that, Algorithm \ref{iLPA} with $\mu_k$ for $c_{\mu}\in\{10,10^2\}$ requires fewer iterations than it with $\mu_k$ for $c_{\mu}\in\{10^3,5\times 10^3, 10^4, 5\times10^4, 10^5, 5\times10^5, 10^6, 5\times10^6, 10^7\}$, but Algorithm \ref{PPASN} corresponding to the former needs more iterations.  Consider that the semismooth Newton step is relatively expensive. Inspired by Figure \ref{figmuk1}, we suggest $\mu_k=10^6/k$ for testing matrix completion with outliers. For Algorithm \ref{iLPA} with $\mu_k=10^6/k$, Figure \ref{figmuk2} (a) shows the number of subproblems solved by each iteration of Algorithm \ref{iLPA}, and Figure \ref{figmuk2} (b) plots the maximum number of iterations required by  Algorithm A to solve the total $j_k+1$ subproblems. We see that now each step of Algorithm \ref{iLPA} needs solving only one subproblem with Algorithm \ref{PPASN}, and the maximum number of iterations for Algorithm \ref{PPASN} is at most $3$.       

\begin{figure}[h]
\centering
\includegraphics[width=\textwidth]{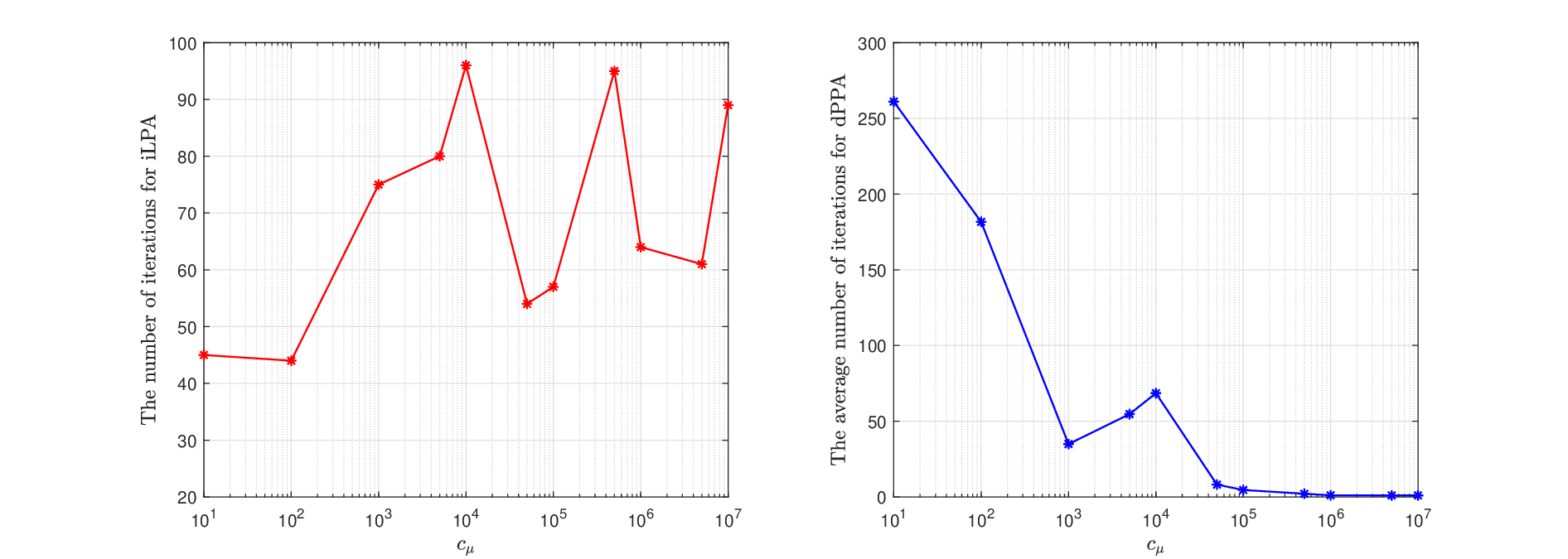}
 \caption{The performance of iLPA with $\mu_k=c_{\mu}k^{-1}$ for different $c_{\mu}$ and that of its inner solver dPPASN}
\label{figmuk1}
\end{figure}
\begin{figure}[h]
\centering
\includegraphics[width=\textwidth]{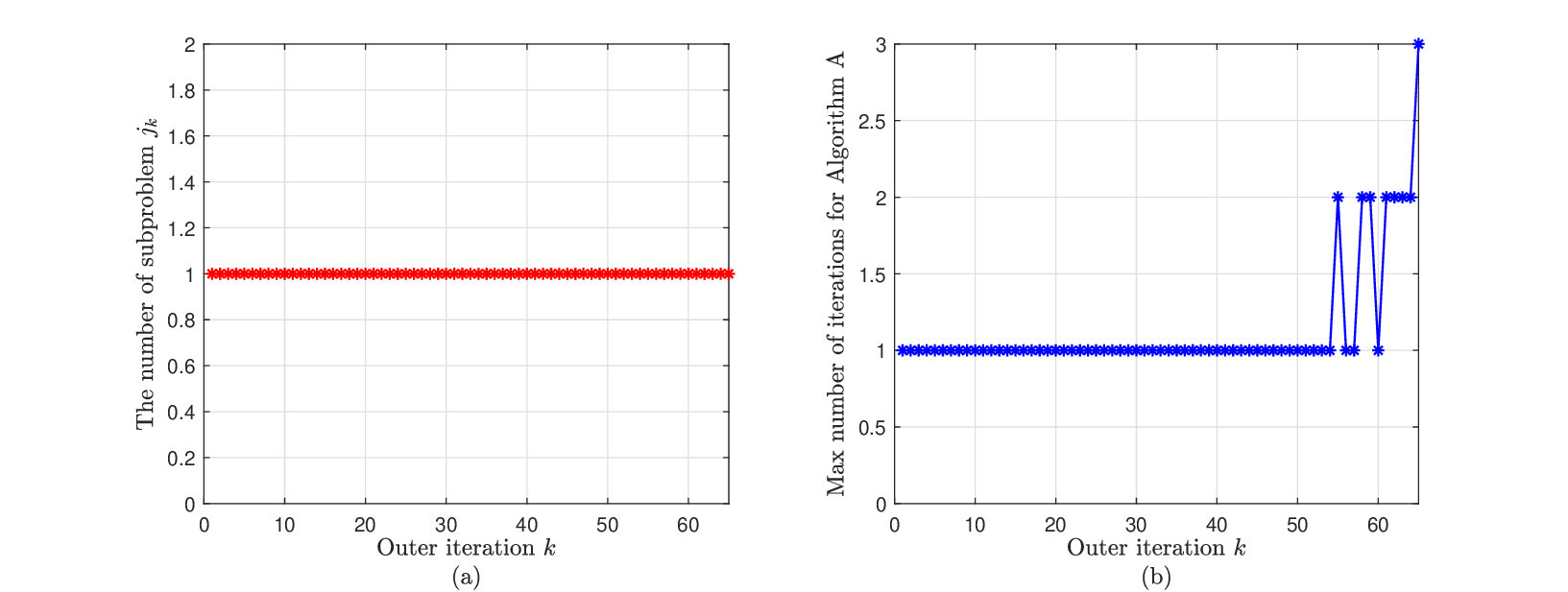}
\caption{The number of subproblems for each step of iLPA with $\mu_k=10^6 k^{-1}$ and the maximum number of iterations of dPPASN}
\label{figmuk2}
\end{figure}

\subsection{DC programs with nonsmooth components}\label{sec6.2}
 We first apply the iLPA to solve the test examples in \cite{Artacho20}. These examples take the form of \eqref{prob} with $h\equiv 0$, and Appendix B includes their description in terms of \eqref{prob}. We see that, the functions $\vartheta_1$ and $\vartheta_2$ corresponding to these examples are piecewise linear convex functions which, by \cite[Proposition 1]{Robinson81} and \cite[Proposition 2.2 (i)]{LiuPanWY22}, are the KL function of exponent $1/2$; the mappings $F$ and $G$ associated with them are at least twice continuously differentiable, which implies that $F'$ and $G'$ are strictly continuous. After checking, the conjugate $\vartheta_1^*$ of $\vartheta_1$ is continuous relative to its domain. Thus, Assumption \ref{ass0} (i)-(iii) and Assumption \ref{ass2} hold for these examples, and the associated potential function $\Xi$ is a KL function. In addition, according to the original description in \cite{Artacho20} for these examples, it is not hard to check that the level set $\mathcal{L}_{\Phi}(x^0)$ is bounded. Then, the whole sequence $\{x^k\}_{k\in\mathbb{N}}$ is convergent by Theorem \ref{global-converge}, and has the R-linear convergence rate if the condition \eqref{key-cond} holds with $h\equiv 0$.

 For these examples, since their objective values have a small scale, we choose  $\mu_k=10/k$ for Algorithm \ref{iLPA}, which yields better objective values. The other parameters are chosen to be $\overline{\varrho}= 2,\,\overline{\gamma}=10^6$ and $\gamma_{k,0}\equiv\underline{\gamma}$. The parameter $\alpha_k$ involved in $\mathcal{Q}_{k,j}$ is chosen to be $\alpha_k\equiv\min\big\{10^{-4},10/\max\{1,\|F'(x^0)\|\}\big\}$. We  
 compare the performance of Algorithm \ref{iLPA} with that of the non-monotone boosted DC algorithm (nmBDCA) proposed in \cite{Ferreria21}. The nmBDCA is a non-monotone version of the boosted DC algorithm proposed in \cite{Artacho20} and able to deal with DC programs with nonsmooth components. The parameters of nmBDCA are set as the default ones in the code of nmBDCA. For a fair comparison, we terminate the iteration of iLPA and nmBDCA at $x^k$ whenever $\|x^k-x^{k-1}\|\le 10^{-8}$ or $k>10^3$.

 Table \ref{DCtab1} reports the average results of $\textbf{100}$ times running for each example. Among others, $\min\Phi$ and $\max\Phi$ denote the minimum and maximum objective values among $100$ running, and \textbf{Nopt} records the number of solutions whose objective values have the absolute difference to the optimal value less than $10^{-5}$. In each running, the two solvers start from the same random initial point generated by MATLAB command $x^0=20*\textrm{rand(n,1)}-10$. From Table \ref{DCtab1}, the iLPA returns more global optimal solutions than the nmBDCA except Example B.\ref{examA.3}. For Example B.\ref{examA.2}, the number of optimal solutions yielded by the iLPA is more twice than the one given by the nmBDCA. The average objective values yielded by the iLPA are better than those given by the nmBDCA except for Example B.\ref{examA.3}.
\setlength{\tabcolsep}{1mm}
\begin{table}[h]
\setlength{\belowcaptionskip}{-0.01cm}
 \caption{Numerical results of iLPA and nmBDCA for Examples B.1-B.6 in Appendix B}\label{DCtab1}
 \centering 
 \scriptsize
 \begin{tabular}{|c|c|c|c|c|l|c|c|c|c|c|}
 \hline
 \multicolumn{1}{|c|}{}&\multicolumn{5}{c|}{iLPA}&\multicolumn{5}{c|}{nmBDCA}\\
 \hline
  Example  &$\min\Phi$  &$\max\Phi$ &$ {\rm ave}\ \Phi^{\rm out}$ & Nopt &time(s)&$\min\Phi $&$\max\Phi$ &$ {\rm ave}\ \Phi^{\rm out}$ & Nopt&time(s)\\
			\hline
			B.1    &2.0000   & 2.0000&2.0000 &100 &0.025 & 2.0000     & 2.0000  &  2.0000  &100 &  0.015 \\
			B.2    &2.01e-11  &2.0000 &0.6400 & 68  & 0.013 & 3.36e-8 & 13.100 &  2.3130  & 28 &  0.015 \\
			B.3    &1.50e-11   &1.0000 &0.0800 & 92  &0.014 & 1.82e-09 & 1.0000  &  0.0200  &98  &  0.009 \\
			B.4    &0.5000     &0.5000 &0.5000 & 100  & 0.017& 0.5000     & 1.0000  &  0.7350  &53  &  0.032 \\
			B.5    &3.5000     &3.7500 &3.5275 & 89  &0.145 & 3.5000     & 3.9405  &  3.5594  &77  &  0.012 \\
			B.6    &-1.1250   &-1.1250 &-1.1250 &100 &0.044 &-1.1250     &-1.1250  & -1.1250   &100 &  0.016 \\			
   \Xhline{1.5pt}
   \end{tabular}
   \end{table}

 We also apply the above two solvers to the $\ell_1$-norm penalty problems of DC constrained problems for a fixed penalty parameter $\beta>0$. This class of problems takes the following form  
 \begin{equation}\label{L1-penalty}
 \min_{x\in\mathbb{R}^n} f(x)+\beta\sum_{i=1}^q\max\big\{0,c_i(x)-d_i(x)\big\}+\chi_{\Omega}(x),
 \end{equation}
 where $\Omega\subset\mathbb{R}^n$ is a simple closed convex set. Obviously, \eqref{L1-penalty} has the form \eqref{prob} with $\mathbb{X}=\mathbb{R}^n,\mathbb{Y}=\mathbb{R}^q$, $\mathbb{Z}=\mathbb{R}, \vartheta_1(y):=\beta{\textstyle\sum_{i=1}^q}\max(y_i,0), \vartheta_2(t):=t$, $h(x):=\delta_{\Omega}(x),F(x):=(c_1(x)-d_1(x),\ldots,c_q(x)-d_q(x))^{\top}$ and $G(x):=-f(x)$.
  The test examples include \textbf{mistake},  \textbf{hs108}, \textbf{hesse} and \textbf{Colville}; see \cite{Barkova21} for their description. The first two examples do not involve the hard constraint $x\in\Omega$, but we impose a soft box set containing their feasible set. Since nmBDCA is inapplicable to the extended real-valued convex functions, we apply it to the following equivalent form of \eqref{L1-penalty}:
  \begin{equation}\label{EL1-penalty}
  \min_{x\in\mathbb{R}^n} \underbrace{\beta\sum_{i=1}^{q'}\max\big\{0,c_i(x)-d_i(x)\big\}+\frac{\alpha}{2}\|x\|^2}_{f_0(x)}-\underbrace {(-f(x)+\frac{\alpha}{2}\|x\|^2)}_{f_1(x)},
  \end{equation}
 where $\alpha>0$ is a constant such that $f_0$ and $f_1$ are convex, and $q'\ge q$ is the number of constraints (including the hard constraints for the last two examples). For this group of examples, Algorithm \ref{iLPA} uses the same parameters as above except that $\gamma_{k,0}=\min\{\|F'(x^0)\|,100\}$. We terminate the two solvers either $\|x^k\!-x^{k-1}\|\le 10^{-8}\& \textbf{infea}^k\!=\sum_{i=1}^{q'}\max\big\{0,c_i(x^k)-d_i(x^k)\big\}\le 10^{-6}$ or $k>10^3$. 

 Table \ref{DCtab2} reports the average results of $\textbf{100}$ times running for each example, where \textbf{Nopt} records the number of solutions whose objective values have the absolute difference to the best known value less than $10^{-5}$. In each running, the two solvers start from the same initial point generated by MATLAB command $x^0=20*\textrm{rand(n,1)}-10$. For nmBDCA, the parameter $\alpha$ in \eqref{EL1-penalty} is chosen to be $0.1$. Such $\alpha$ does not necessarily guarantee the convexity of the corresponding $f_0$ and $f_1$, but it leads to a better result. From Table \ref{DCtab2}, iLPA returns more best known solutions than nmBDCA for ``mistake'', ``hesse'' and ``ore'', and it returns better average objective value for ``hs108'' though does not produce a best known solution. The feasibility violation yielded by iLPA is a little worse than the one given by nmBDCA.

 \setlength{\tabcolsep}{1mm}
 \begin{table}[h]
 \setlength{\belowcaptionskip}{-0.01cm}
 \scriptsize
 \centering
 \caption{Numerical results of iLPA and nmBDCA for DC constrained test examples}\label{DCtab2}
 \begin{tabular}{|c|c|c|c|c|c|l|c|c|c|c|c|}
 \hline
  \multicolumn{2}{|c|}{}&\multicolumn{5}{c|}{iLPA}&\multicolumn{5}{c|}{nmBDCA}\\
 			\hline
 			Problem &$\beta$  & max & ave & Nopt & Infea&time(s)  & max & ave & Nopt & Infea & time(s)\\
 			\hline
 			mistake &$10$     &-0.5063   & -1.0000   & 60 &2.656e-4 &0.167  & -0.4593   & -0.9102  & 1  & 3.936e-9 &  0.078 \\
 			hs108   &$10$   &-0.4996  &-0.7660  & 0 & 3.999e-5 &0.227 & -0.3810  & -0.7466  & 2  & 2.129e-6 &  0.094 \\
 			hesse   &\ $10^4$ &-36.0000    &-197.58  &2  & 3.144e-10 & 0.010 &-21.971  & -188.003 & 0  & 5.052e-9 &  0.027 \\
 			ore     &\ $10^2$ & -0.9198    &-1.0719    &46 &6.378e-8 &0.182  & -0.9167  &-1.0808    &0  & 2.228e-11 &  0.042 \\
 			\Xhline{1.5pt}
 	\end{tabular}
 \end{table} 
\subsection{Matrix completion with outliers under non-uniform sampling}\label{sec6.3}

 We apply the iLPA to solve the problem \eqref{SCAD-loss} with $\vartheta$ being the SCAD function. From the discussion in  Example \ref{exam3}, it has the form \eqref{prob} with $F(x)=\mathcal{A}(UV^{\top})-b=G(x),h(x)=\lambda(\|U\|_{2,1}+\|V\|_{2,1})$ for $x=(U,V)\in\mathbb{X}=\mathbb{R}^{n_1\times r}\times\mathbb{R}^{n_2\times r}$ and $\vartheta_1(y)=\|y\|_1$, $\vartheta_2(y)=\frac{1}{\rho}\sum_{i=1}^m\theta_a(\rho|y_i|)$ for $y\in\mathbb{R}^m$. Obviously, Assumptions \ref{ass0}-\ref{ass2} hold for this model, and the corresponding potential function $\Xi$ is semialgebraic. Notice that $\vartheta_2$ is a PLQ definable function in an
o-minimal structure over the real field, so is the function $\mathbb{Z}\times\mathbb{Z}\ni(u,z)\mapsto \langle u,z\rangle+\vartheta_2^*(-z)$. Then, it is a KL function of exponent $1/2$ by virtue of \cite[Proposition 1]{Robinson81} and \cite[Proposition 2.2 (i)]{LiuPanWY22}, and so is the function $\vartheta_1$. In addition, the mapping $\partial h$ is metrically subregular by virtue of \cite{ZhouSo17}, which by \cite[Proposition 2.2 (i) \& Remark 2.2]{LiuPanWY22} means that $h$ is also a KL function of exponent $1/2$. Thus, the associated function $f$ in \eqref{ffun} is a KL function of exponent $1/2$. Consequently, the sequence $\{x^k\}_{k\in\mathbb{N}}$ yielded by Algorithm \ref{iLPA} is convergent, and has a R-linear convergence rate if the condition \eqref{key-cond} is satisfied. For such $F$ and $G$, the mapping $\mathcal{A}_k$ and the matrix $u^k$ in \eqref{Akmap} are specified as 
 \[
  \mathcal{A}_k(G,H):=\mathcal{A}(U^kH^{\top}+G(V^k)^{\top})\ \ {\rm for}\ \ (G,H)\in\mathbb{X}\ \ {\rm and}\ \ 
  u^k=((\mathcal{A}^*\xi^k)V^k,(\mathcal{A}^*\xi^k)^{\top}U^k).
 \]

 To formulate the sampling operator $\mathcal{A}$, a random index set $\Omega=\big\{(i_t,j_t)\ |\ t=1,\ldots,m\big\}$ is assumed to be available, and the samples of the indices are drawn independently from a general distribution $\Pi\!=\{\pi_{kl}\}_{k\in[n_1],l\in[n_2]}$ on $[n_1]\times[n_2]$. We adopt the non-uniform sampling scheme used in \cite{Fang18}, i.e., 
\begin{equation}\label{sampling-scheme}
 \pi_{kl}=p_kp_l\ \ {\rm for\ each}\ (k,l)\ \ {\rm with}\ \ p_k=\left\{\begin{array}{ll}
      2p_0& {\rm if}\ k\le\frac{n_1}{10}, \\
      4p_0& {\rm if}\ \frac{n_1}{10}\le k\le \frac{n_1}{5},\\
      p_0& {\rm otherwise},\\
     \end{array}\right.
\end{equation}
where $p_0>0$ is a constant such that $\sum_{k=1}^{n_1}p_k=1$ or $\sum_{l=1}^{n_2}p_l=1$. Then, the mapping $\mathcal{A}$ is defined by
$\mathcal{A}(X):=(X_{i_1,j_1},\,X_{i_2,j_2},\ldots,X_{i_m,j_m})^{\top}$ for $X\in\mathbb{R}^{n_1\times n_2}$, and $b=\mathcal{A}(M_{\Omega})$ where $M_{\Omega}$ is an $n_1\times n_2$ matrix with \begin{equation}\label{observe}
 [M_{\Omega}]_{i_t,j_t}=\left\{\begin{array}{cl}
  0 & {\rm if}\ (i_t,j_t)\notin\Omega,\\
 M_{i_t,j_t}^*+\varpi_t &{\rm if}\ (i_t,j_t)\in\Omega
 \end{array}\right.\ {\rm for}\ \ t=1,2,\ldots,m.
 \end{equation}
 Here, $M^*$ is the true matrix of rank $r^*$ for synthetic data, and for real data it is a matrix drawn from the original incomplete data matrix, and $\varpi=(\varpi_1,\ldots,\varpi_m)^{\top}$ is a sparse noisy vector. The nonzero entries of $\varpi$ obey one of the following distributions: {\bf(I)} $N(0,10^2)$; {\bf(II)} Student's t-distribution with $4$ degrees of freedom scaled by $\sqrt{2}$; {\bf(III)} Cauchy distribution with density $d(u)=\frac{1}{\pi(1+u^2)}$; {\bf(IV)} mixture normal distribution $N(0,\sigma^2)$ with $\sigma\sim {\rm Unif}(1,5)$; {\bf(V)} Laplace distribution with density $d(u)=0.5\exp(-|u|)$. 
 
For synthetic data, we evaluate the effect of matrix recovery in terms of the relative error (RE), defined by $\frac{\|X^{\rm out}-M^*\|_F}{\|M^*\|_F}$, where $X^{\rm out}=U^{\rm out}(V^{\rm out})^{\top}$ represents the output of a solver. For real data, we adopt the normalized mean absolute error (NMAE) to measure the accuracy; see Section \ref{sec6.3.3} for its definition. In addition, we also record the sparsity ratio (SPR) of the DC loss term at the output, i.e., the percentage of the number of zero components of $\mathcal{A}(X^{\rm out})-b$ in the number of sampling, where the number of zero components of a vector $z\in\mathbb{R}^m$ is calculated by $|\{i\in\{1,\ldots,m\}\ |\ |z_i|\le 10^{-4}\|z\|_{\infty}\}|$. 
\subsubsection{Choice of parameters and stop condition }\label{sec6.3.1}

 We first focus on the choice of parameters in Algorithm \ref{iLPA}. As suggested in Section \ref{sec6.1}, we choose $\mu_k\equiv 10^6/k$. The others are the same as in Section \ref{sec6.2} except that $\gamma_{k,0}\equiv\max\{10,\lfloor10^{-2}\min\{n_1,n_2\}\rfloor\}$. The parameter $\alpha_k$ in $\mathcal{Q}_{k,j}=\gamma_{k,j}\mathcal{I}+\alpha_k\mathcal{A}_k^*\mathcal{A}_k$ is chosen by the following rule with $\alpha_0=1.0$:
 \begin{equation}\label{alpk-update}
  \alpha_{k}=\left\{\begin{array}{cl}
  \max\{\alpha_{k-1}/1.2,10^{-3}\}&{\rm if}\ {\rm mod}(k,3)=0,\\
  \alpha_{k-1} &{\rm otherwise}.
  \end{array}\right.
 \end{equation} 
 By Remark \ref{remark-alg} (e), we terminate Algorithm \ref{iLPA} at $x^k=(U^k,V^k)$ when one of the three conditions hold  
\begin{equation}\label{stop-cond}
 \frac{\|x^k-x^{k-1}\|_F}{1\!+\|b\|}\le\varepsilon_1,\ \frac{\max_{j\in\{1,\ldots,9\}}|\Phi(x^k)-\Phi(x^{k-j})|}{\max\{1,\Phi(x^k)\}}\le \varepsilon_2\ \ {\rm for}\ k\ge 10\ \ {\rm and}\ k>k_{\rm max}.
\end{equation}
Unless otherwise stated, $\varepsilon_1=10^{-5},\varepsilon_2=5\times 10^{-4}$ and $k_{\rm max}=500$ are used for the subsequent tests.

Next we take a look at the choice of parameters in model \eqref{SCAD-loss}. As the term $\lambda(\|U\|_{2,1}\!+\|V\|_{2,1})$ is used to reduce rank by promoting column sparsity, we choose \[r=\min(100,\lfloor\frac{1}{2}\min(n_1,n_2)\rfloor)\ \ {\rm and} \ \ \lambda=c_{\lambda}\|b\|,\]
where $c_\lambda>0$ is specified in the experiments. For the constant $a$ in $\vartheta_2$, we always choose $a=4$, which is close to $3.7$ suggested in \cite{Fan01}. The parameter $\rho$ in $\vartheta_2$ has influence on the relative error and the sparsity of the vector $\mathcal{A}(X^{\rm out})-b$. As shown by Figure \ref{fig0} below, the relative error has tiny variation when $\rho\in[0.008,0.5]$, and it becomes worse as $\rho$ increases in $(0.6,1.2]$; the sparsity ratio is desirable when $\rho\in[0.008,0.7]$, but as $\rho$ increases in $(0.7,1.2]$ it decreases rapidly and is close to zero. This means that the concave composition term $-\vartheta_2(\mathcal{A}(UV^{\top})-b)$ plays an active role when $\rho\in[0.008,0.7]$. After making trade-off between the relative error and the sparsity of $\mathcal{A}(X^{\rm out})-b$, we always choose $\rho=10^{-2}$ for the subsequent experiments.
\begin{figure}[h]
\centering
\includegraphics[width=\textwidth]{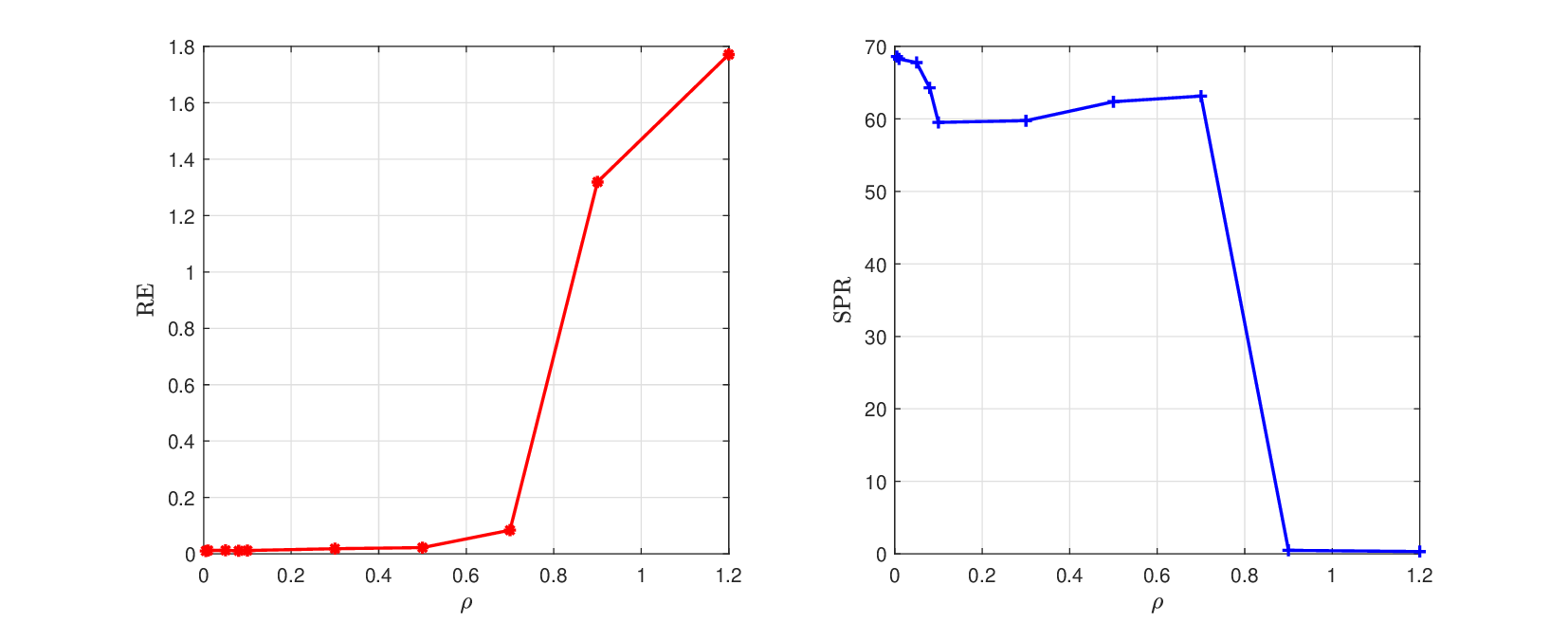}
 \caption{The relative error and sparsity ratio curves of iLPA under noise of type IV with $n_1=n_2=1000,r^*=10,SR=0.15$}
  \label{fig0}
\end{figure}

 We compare the performance of the iLPA with that of the Polyak subgradient method (subGM for short) and that of the PAM method in Appendix C. Considering that $\min_{x\in\mathbb{X}}\Phi(x)$ in step \ref{step4} of Algorithm \ref{subGrad} is unavailable in practice, we replace the step-size $\frac{\Phi(x^k)-\min_{x\in\mathbb{X}}\Phi(x)}{\|\zeta^k\|^2}$ with $\frac{0.05|\Phi(x^k)|}{\|\zeta^k\|^2}$. Algorithm \ref{subGrad} with such an approximate one does not admit the subsequential convergence, and we use it just for numerical comparison. In addition, as mentioned in Remark \ref{remark-BCD} of Appendix C, when Algorithm \ref{PAM} is applied to solve the problem \eqref{SCAD-loss}, its iterate sequence lacks a full convergence certificate, and we use it just for numerical comparison. For the parameters $\alpha_{i,k}$ and $\gamma_{i,k}$ involved in the PAM, we update $\alpha_{i,k}$ for $i=1,2$ by the same rule as for $\alpha_k$ in \eqref{alpk-update} with $\alpha_{1,0}=\alpha_{2,0}=10$, and update $\gamma_{i,k}$ for $i=1,2$ by the rule
\[
   \gamma_{i,k}=\left\{\begin{array}{cl}
      \max\{\gamma_{i,k-1}/1.2,10^{-3}\}&{\rm if}\ \textrm{mod}(k,3)=0,\\
      \gamma_{i,k-1} &{\rm otherwise}.      
   \end{array}\right.
\] 
The accuracy $\epsilon_k$ for solving subproblems is updated via $\epsilon_k=\max\{10^{-6},0.95\epsilon_{k-1}\}$ with $\epsilon_0=10$. For fair comparison, the SubGM and PAM use the same starting point $x^0$ and stop condition as for iLPA.
\subsubsection{Numerical results for synthetic data}\label{sec6.3.2}

 We generate randomly the true matrix $M^*=M_{L}^*(M_{R}^*)^{\top}\!\in\mathbb{R}^{n_1\times n_2}$ by sampling the entries of $M_{L}^*\in\mathbb{R}^{n_1\times r^*}$ and $M_{R}^*\in\mathbb{R}^{n_2\times r^*}$ independently from the standard normal distribution $N(0,1)$. The number of nonzero entries of the noise vector $\varpi$ is set to be $\lfloor0.3m\rfloor$. We choose $x^0=(U_{1}\Sigma_{r}^{1/2},V_{1}\Sigma_{r}^{1/2})$ as the starting point of iLPA and PAM, where $U_{1}\in\mathbb{R}^{n_1\times r}$ and $V_1\in\mathbb{R}^{n_2\times r}$ are the matrix consisting of the first $r$ largest left and right singular vectors of $M_{\Omega}$, respectively, and $\Sigma_{r}$ is the diagonal matrix consisting of the first $r$ largest singular values of $M_{\Omega}$ arranged in an nonincreasing order. 
 
 Before testing the performance of the three solvers on synthetic data, we take a look at their iteration behaviors and how the relative errors yielded by them vary with $\lambda$. Figure \ref{fig1} shows the relative errors of the successive iterations of the three solvers. We see that the relative errors by the iLPA approach to zero faster than those by the PAM and the subGM, which means that the iterate sequence of iLPA has better global convergence than that of PAM and subGM. Figure \ref{fig2} below plots the relative error and rank curves of the three solvers as the parameter $\lambda$ varies, by using the average results for running $5$ examples generated randomly with noise of type V, $n_1=n_2=1000,r^*=5$ and $\textrm{SR}=0.25$. We see that, there exists an interval of $\lambda$ such that the stationary points yielded by them with such $\lambda$ have the satisfactory relative errors and the true rank $r^*$. Such an interval of iLPA is remarkably larger than the other two solvers, which means that the iLPA has better robustness with respect to $\lambda$.
 \begin{figure}[h]
 \centering
\includegraphics[width=\textwidth]{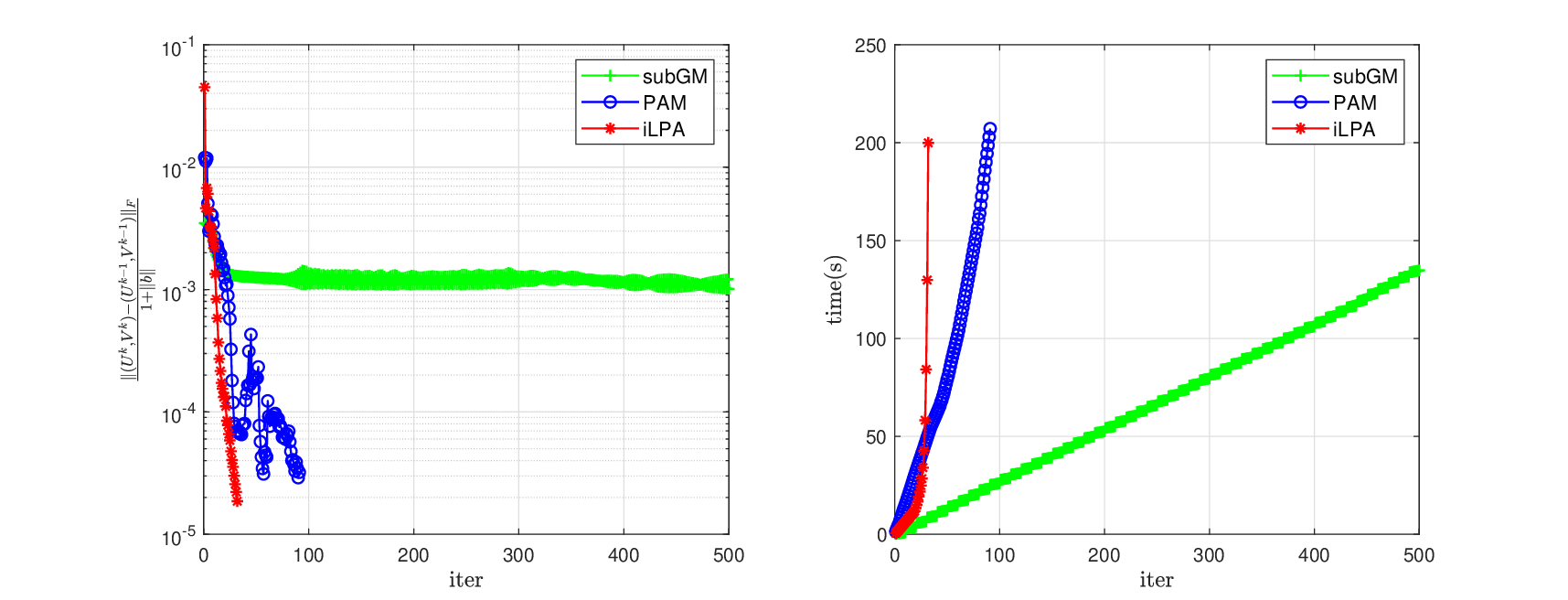}
 \caption{Iteration and time curves of iLPA, PAM and subGM under noise of type IV with $n_1=n_2=3000, r^*=10$ and $SR=0.15$}
  \label{fig1}
\end{figure}
\begin{figure}[h]
 \centering
\includegraphics[width=\textwidth]{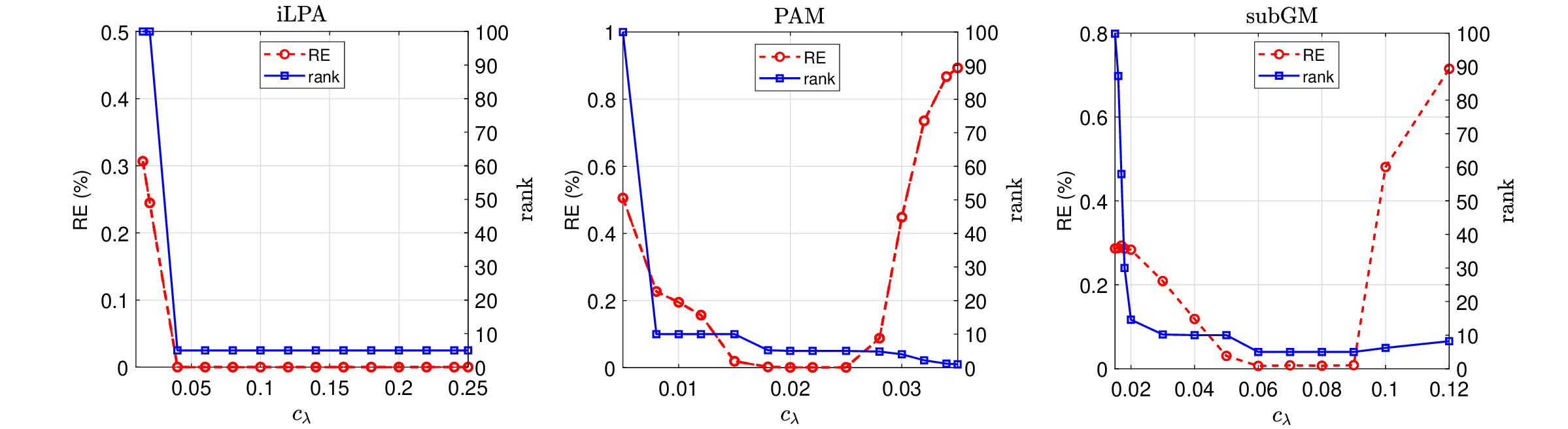}
 \caption{The relative error and rank curves of iLPA, PAM and subGM as the parameter $c_{\lambda}$ or $\lambda$ increases}
  \label{fig2}
\end{figure}

 Now we are in a position to show the recovery effect and the running time (in seconds) of the three solvers under different $n_1\!=n_2,\,r^*$ and $\textrm{SR}$. Table \ref{tabSubG} reports the average results for running $5$ examples generated randomly in each setting, where the results with lower REs and higher SPR are marked in black. Consider that the interval of $\lambda$ for the PAM and the subGM to have the satisfactory relative errors is smaller, so in Table \ref{tabSubG} we choose the value of $c_{\lambda}$ by referring to Figure \ref{fig2}. From Table \ref{tabSubG}, the iLPA is superior to the other two solvers in terms of relative error and SPR for most examples, and its running time is much less than that of the PAM for $n_1=6000$ and is comparable with that of the latter for $n_1=1000$. Observe that under noise of type III with $n_1=6000$ and ${\rm SR}=0.15$, the running time of the PAM is $8572$ seconds since the PPASN attains the maximum number of iterations $100$ for one of the examples. The subGM needs the least running time but yields the worst relative error and SPR. In addition, the ranks returned by the iLPA and the PAM all coincide with the true one, but the ranks returned by the subGM are generally higher than the true one.
\setlength{\tabcolsep}{1mm}
\setlength{\abovecaptionskip}{-0.1cm}
\begin{sidewaystable}[h]
	\setlength{\belowcaptionskip}{-0.01cm}
	\setlength\tabcolsep{0.78pt}
	\renewcommand\arraystretch{1.3}
	\centering
	\tiny
	\caption{\small Average RE and running time of three solvers for test examples generated randomly}\label{tabSubG}
		\begin{tabular*}{\textwidth}{@{\extracolsep{\fill}}cc|lcccc|lcccc|lcccc||lcccc|lcccc|lcccc@{\extracolsep{\fill}}}	
			\hline
			& & \multicolumn{5}{l}{\  \ iLPA($n_1=1000$,$r^*=10$)}&\multicolumn{5}{l}{ \ \ PAM ($n_1=1000$,$r^*=10$)}&\multicolumn{5}{l||}{\  \ SubGM ($n_1=1000$,$r^*=10$)}& \multicolumn{5}{l}{\  \ iLPA ($n_1=6000$,$r^*=15$)}&
			\multicolumn{5}{l}{\ PAM ($n_1=6000$,$r^*=15$)}&\multicolumn{5}{l}{\  SubGM ($n_1=6000$,$r^*=15$)}\\
			\hline
			
			$\varpi$ & SR&$c_{\lambda}$& RE & rank  &SPR&time &$c_{\lambda}$& RE & rank  &SPR&time &$ c_{\lambda}$&RE & rank &SPR& time&$c_{\lambda}$& RE & rank  &SPR&time&$c_{\lambda}$&  RE & rank &SPR& time&$c_{\lambda}$&  RE & rank &SPR& time\\
			\hline			
	&0.15 &0.06& 3.54e-2 & 10 & 65.6 &6.29&0.06& {\bf 3.49e-2}&10 &52.5& 11.0&  0.06  &9.96e-2 &11&1.37  &5.66 
    & 0.08&{\bf 4.74e-5}& 15&{\bf 70.0}& 191.9&0.08&1.85e-3&15 &53.8& 793.0  &0.15&7.02e-2&16&2.84&249.3\\
I	&0.25 &0.06& {\bf 1.67e-4} & 10 & {\bf 69.9} &12.2&0.06&3.03e-3&10 &48.7& 9.95& 0.06 &4.51e-2 &11&2.20&6.90 
    & 0.08& {\bf 6.62e-5}& 15&{\bf 70.0}&194.4&0.08&1.68e-3&15 &54.8& 998.3 & 0.15&3.42e-2&15  & 4.02 &  296.0 \\
   \hline
   
    &0.15 &0.07& 5.80e-3& 10 & 69.8 & 17.4 &0.07& 5.53e-3& 10 &67.4&31.2&  0.18  &1.38e-1 &11&2.28  &5.11& 
     0.08& {\bf 3.11e-5} & 15 & 70.2&527.9&0.08& 6.76e-4& 15 & 69.8& 1122&0.40& 1.52e-1 &16& 5.42 & 260.4\\
II	&0.25 &0.07& {\bf 8.78e-5}& 10 & {\bf 70.1}& 19.2 &0.07&2.88e-3&10&43.1&10.7 &0.18  &2.33e-2 &10&6.73 &6.63 
    &0.08& {\bf 2.28e-5} & 15 & {\bf70.2}& 557.9&0.08&5.01e-4& 15 &70.1& 1199&0.40& 3.15e-2 &16&11.9& 345.4\\
			\hline
    &0.15 &0.025& {\bf3.95e-2}& 10 & {\bf 95.1}&46.1&0.025 & 4.14e-2&10& 95.1&82.5& 0.05 &  4.20e-1  &11  &82.5 & 5.19&
    0.08&{\bf 2.45e-4} &15&{\bf 99.8}&642.8 &0.08&7.05e-4 &15&99.8& 8572&0.40& 2.03e-1 &16&92.9& 245.2\\
III	&0.25 &0.025&{\bf5.16e-4}& 10 & {\bf 96.7}&55.9&0.025 & 1.83e-3&10&96.7& 84.1&  0.05  & 4.71e-1 & 12  & 96.2&6.56&
    0.08& {\bf 1.39e-4} &15&{\bf 99.8}& 816.1&0.08&5.40e-4 &15&99.8&1630&0.40& 2.13e-2 &15&99.8& 304.2\\
			
			\hline
	&0.15 &0.07&  8.37e-3& 10 & 69.1 & 14.1&0.07& 8.25e-3 & 10 &47.4& 24.7& 0.15&1.12e-1 & 13& 0.94& 5.91  &       
 0.08& {\bf 3.11e-5} & 15 & {\bf 70.0}& 420.2 &0.07& 7.60e-4 &15& 50.6 & 907.5&0.40& 1.90e-1 &17& 0.97 & 243.4\\
IV	&0.25 &0.07& {\bf 1.00e-4}& 10 &{\bf 69.9}& 18.6&0.07 &3.13e-3 & 10 &26.5& 12.7& 0.15&2.46e-2 &10&3.25  &6.65&     
  0.08& {\bf 1.72e-5} & 15 &{\bf 70.0}& 607.1&0.07& 8.83e-4 &15& 48.1 & 721.0&0.40& 4.26e-2 &16& 2.11 & 310.5\\
  \hline
			
	&0.15 &0.07& {\bf 2.47e-3}& 10 &{\bf 69.3} & 31.1&0.07 & 2.82e-3 & 10 & 48.2&42.3 &0.20&1.47e-1&11 &0.81&   5.44 &
    0.08& {\bf 4.84e-5} & 15 & {\bf 70.0}& 481.9 &0.08& 5.30e-4 &15& 50.2 & 1473&0.40& 1.30e-1 &16& 0.83 & 246.6\\
V   &0.25 &0.07& {\bf7.46e-5}& 10 & {\bf69.9 }&26.1&0.07 &  3.40e-3 & 10 & 15.1&9.18& 0.20 &2.35e-2&10&2.44&  6.47&  
   0.08& {\bf 3.68e-5} & 15 &{\bf 70.0}& 510.1&0.08& 1.09e-3 &15& 34.0 & 899.0&0.40& 2.65e-2 &15& 2.13 & 314.5\\
			\hline
	\end{tabular*}
\end{sidewaystable}

\subsubsection{Numerical results for real data}\label{sec6.3.3}
 
 We test the performance of iLPA, PAM and subGM on matrix completion with real data sets, including the jester joke, movieLens and netflix datasets. For each dataset, let $M^0$ denote the original incomplete data matrix such that the $i$th row of $M^0$ corresponds to the ratings given by the $i$th user. Since many entries are unknown, we cannot compute the relative error as we did for the simulated data. Instead, we take the metric of the normalized mean absolute error (NMAE) to measure the accuracy:
 \[
 {\rm NMAE}=\frac{\sum_{(i,j)\in\Gamma\backslash\Omega}|X^{\rm out}_{i,j}-M_{i,j}|}
 {|\Gamma\backslash\Omega|(r_{\rm max}-r_{\rm min})}\ \ {\rm with}\ \ X^{\rm out}=U^{\rm out}(V^{\rm out})^{\top},
 \]
 where $\Gamma:=\{(i,j)\in[n_{1}]\times[n_2]\ |\ M_{ij}\ \textrm{is given}\}$ denotes the set of indices for which $M_{ij}$ is given, and $r_{\rm min}$ and $r_{\rm max}$ denote the lower and upper bounds of the ratings, respectively. 

 Before testing the performance of the three solvers, we utilize the netflix dataset from \url{https://www.kaggle.com/netflix-inc/netflix-prize-data\#qualifying.txt} to examine their iteration behaviors. We first randomly select $n_1=3000$ users and their $n_2=3000$ column ratings from $M^0$, sample the observed entries with the sampling scheme \eqref{sampling-scheme}, and then obtain $M_{\Omega}$ via \eqref{observe} with $M^*=M^0$. Figure \ref{fig3} shows the relative errors of the successive iterations of the three solvers. We see that their iteration behaviors are similar to those on synthetic data in Figure \ref{fig2}, but the running time of the PAM increases more quickly as the number of iterations increases. To ensure that the PAM can be used to test real data of large scale, we relax its stopping condition by replacing $\varepsilon_1=10^{-5}$ with $\varepsilon_1=10^{-4}$, $\varepsilon_2=5\times 10^{-4}$ with $\varepsilon_2=10^{-3}$, and $k_{\rm max}=100$ with $k_{\rm max}=40$ for the subsequent tests on real datasets. 
 \begin{figure}[h]
 \centering
\includegraphics[width=\textwidth]{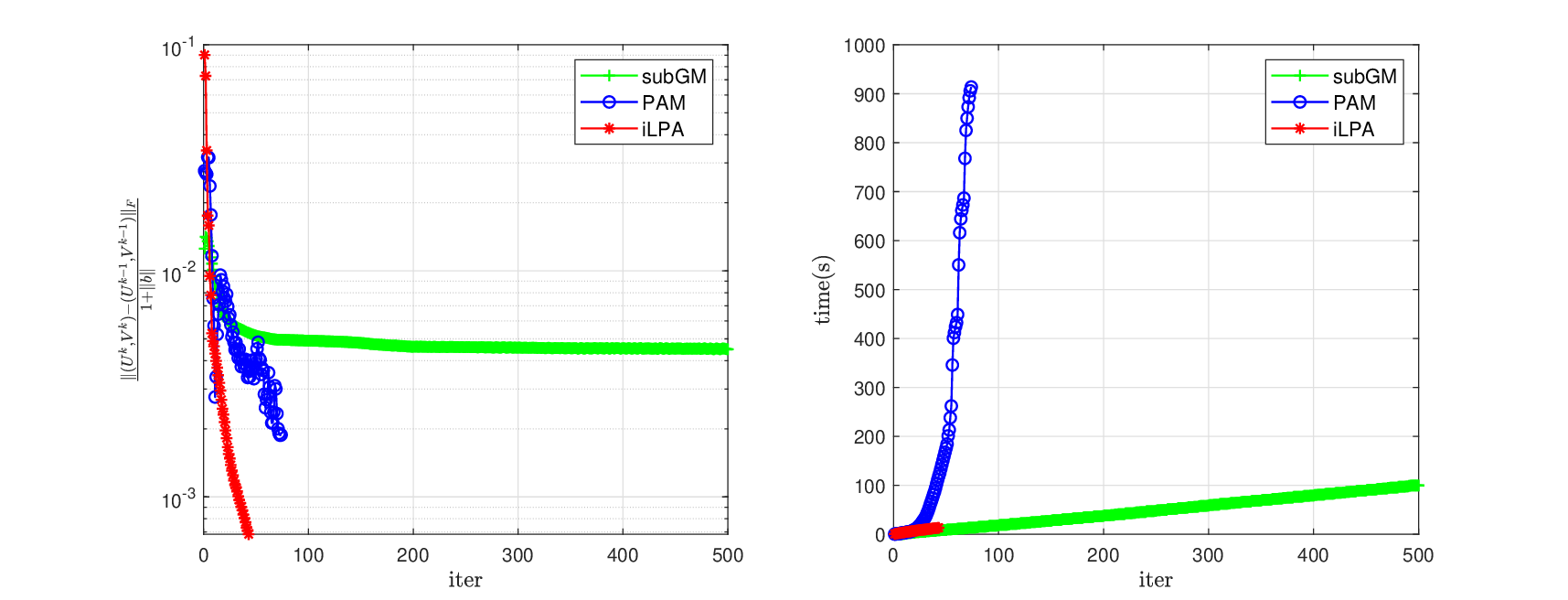}
 \caption{Iteration and time curves of iLPA, PAM and subGM under noise of type IV for the netflix dataset with $n_1=n_2=3000$}
  \label{fig3}
\end{figure}
 
We also check how the NMAEs yielded by the three solvers vary with $\lambda$ by using the movie-100K dataset, which is contained in the movieLens dataset from \url{http://www.grouplens.org/node/73}. Just like \cite{Toh10}, we consider the data matrix $\widetilde{M}^0\!=M^0-3$, and obtain $M_{\Omega}$ via \eqref{observe} with $M^*\!=\widetilde{M}^0$. From Figure \ref{fig4} below, the NMAE yielded by the iLPA with $c_{\lambda}\in[0.2,0.8]$ has a comparable variation with the NMAE yielded by the PAM with $c_{\lambda}\in[0.15,0.38]$ and the NMAE yielded by the subGM with $c_{\lambda}\in[0.15,0.5]$, though the best and worst NMAEs by the former with $c_{\lambda}\in[0.2,0.8]$ are a little higher than those by the PAM with $c_{\lambda}\in[0.15,0.38]$ and those by the subGM with $c_{\lambda}\in[0.15,0.5]$. Unlike Figure \ref{fig2} for synthetic data, the ranks yielded by the three solvers decrease quickly to $1$, and then keep unchanged for a certain range of $\lambda$. Clearly, such an interval of $\lambda$ for the iLPA is larger than the one for the other two solvers. This means that for real data the iLPA still has better robustness with respect to $\lambda$ than the PAM and the subGM. Since the intervals of $\lambda$ for the three solvers to return better NMAEs may be disjoint, for the fairness of comparison, in the subsequent testing, we will choose different $c_{\lambda}$ for the three solvers.
  
\begin{figure}[h]
 \centering
\includegraphics[width=\textwidth]{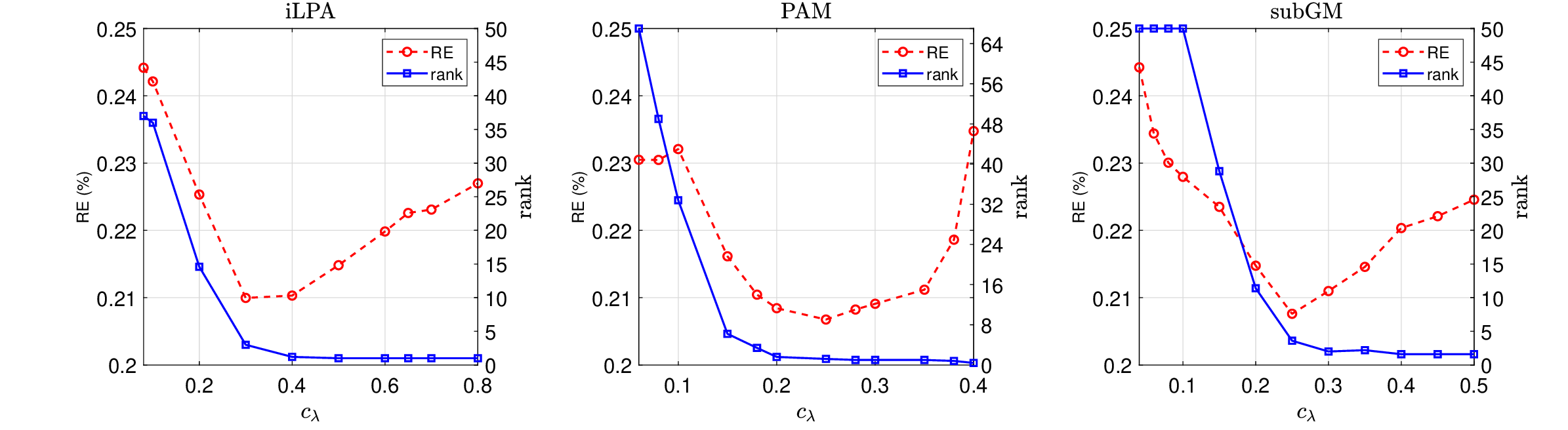}
 \caption{The relative error and rank curves of iLPA, PAM and subGM as $c_{\lambda}$ or $\lambda$ increases for the movie-100K dataset}
  \label{fig4}
\end{figure}
 
We start with testing the netflix dataset. Due to the large number of users, we randomly select $n_1$ users with $n_1=6000$ (resp. $10000$) and their $n_2=n_1$ column ratings from $M^0$, sample the observed entries with the sampling scheme in \eqref{sampling-scheme}, and then obtain $M_{\Omega}$ via \eqref{observe} with $M^*=M^0$. Preliminary tests indicate that as $c_{\lambda}$ or $\lambda$ increases, the three solvers have similar performance as they do in Figure \ref{fig4}, so we report their results for solving \eqref{SCAD-loss} with their respective $c_{\lambda}$ such that the target rank is $1$. Table \ref{tabNetflix} reports the average NMAE, rank, SPR and running time (in seconds) obtained by running $5$ times for each setting. We see that the three solvers yield the comparable NMAEs though the NMAEs by the PAM are the best and the NMAEs by the subGM are the worst; the running time of the iLPA is comparable with that of the subGM, but the running time of the PAM is the most. 

\setlength{\tabcolsep}{1mm}
\begin{sidewaystable}[h]
	\setlength{\belowcaptionskip}{-0.01cm}
	\setlength\tabcolsep{1.0pt}
	\renewcommand\arraystretch{1.3}
	\centering
	\tiny
	\caption{\small Average NMAE and running time of iLPA, PAM and subGM for netflix dataset}\label{tabNetflix}
	
  \begin{tabular*}{\textwidth}{@{\extracolsep{\fill}}cc|lcccc|lcccc|lcccc||lcccc|lcccc|lcccc@{\extracolsep{\fill}}}
			\hline
			& & \multicolumn{5}{l}{\  \  iLPA($n_{1}=6000$)}&\multicolumn{5}{l}{\  \ \ PAM($n_1=6000$)}&\multicolumn{5}{l||}{\  \ subGM($n_1=6000$)}& \multicolumn{5}{l}{\  \ \ iLPA($n_{1}=10000$)}&
			\multicolumn{5}{l}{\  \ \ PAM($n_{1}\!\!=\!10000$)}&\multicolumn{5}{l}{\  \ \ subGM($n_{1}\!\!=\!10000$)}\\
			\hline
			$\varpi$ & SR\ &\ $c_{\lambda}$& NMAE & rank&SPR  &time&$c_{\lambda}$&  NMAE & rank&SPR & time&\ $c_{\lambda}$& NMAE & rank&SPR  &time&$c_{\lambda}$& NMAE & rank  &SPR&time&$c_{\lambda}$& NMAE & rank&SPR & time&$c_{\lambda}$& NMAE & rank&SPR & time\\
			\hline	
	    &0.15 &0.1& 0.2268 & 1.6 &10.3& 112.2 &0.06&0.2257 &1.2&14.7&337.4 &0.10&0.2286 &3.2&5.76&62.2  
       & 0.1& 0.2171 & 2.2 &7.97&246.0 &0.06& 0.2111 & 1.0 &13.5& 546.0 & 0.1& 0.2175 & 3.5 &4.73&207.8\\
  I	&0.25 &0.1& 0.2140 & 1.6 &9.94& 103.3 &0.06&0.2091 &1.0  &13.6&313.0 &0.10&0.2145 &4.0&5.07&62.9  
      & 0.1& 0.2051 & 2 &5.87&158.8 &0.06& 0.2006 & 1.0 &11.4& 470.4 & 0.1& 0.2050 & 5.8 &3.50&268.3\\
			\hline
			
		&0.15 &0.3& 0.2177&2.6&9.30&113.3& 0.2& 0.2140 &1.2 &13.6&315.7  &0.40&0.2200 &2.6&6.10&90.7  
        &0.3&0.2094 &2.8&7.97&279.2&0.2& 0.2055 & 1.0 &12.2& 652.9& 0.4& 0.2214 & 2.0 &7.27&126.0\\
	II	&0.25 &0.3& 0.2068&2.2&7.91&114.0& 0.2& 0.2037 &1.0 &12.8& 333.6 &0.40 &0.2092 &2.4&5.53&80.6   
    &0.3&0.1979&2.4 &4.77&203.6&0.2& 0.1969&1&9.84 & 639.6& 0.4& 0.2092 & 2.6 &4.75& 165.6\\
			\hline
			
  &0.15 &0.02& 0.2233 &1.0 & 47.1&51.1&0.01&0.2138 &1 &51.7 & 158.7  &0.01  &0.2184  &6.2 &61.1 &  97.3    
        &0.03 &0.2113& 1.2 & 53.8 & 178.9 &0.018&  0.2045 & 1.0 & 55.5 & 321.4  &0.03 &0.2096& 20 & 76.1 & 314.9\\
  III   &0.25 &0.02& 0.2072 &1.4& 52.2&55.1&0.01&0.2012&1 & 52.8& 144.6  & 0.01 &0.2024  &9.6 &65.9 &131.3    
        &0.03& 0.1954 &2.0 & 62.6 & 187.2 &0.018 &  0.1944 & 1.0 & 58.4 & 751.8  &0.03 &0.1973& 50 & 84.4 & 320.9 \\
			\hline
			
		&0.15 &0.3& 0.2231&1.2&9.91&62.9 &0.18& 0.2174 & 1.0 &11.9& 265.4 &0.30& 0.2292 & 1.6&6.31 & 70.7  
        &0.3&  0.2152 & 1.4&7.66 & 223.9 &0.15& 0.2072 & 1.0 &10.7&851.4 & 0.3& 0.2208 & 2.2 &4.73&149.6\\
	IV	&0.25 &0.3& 0.2126&1.0&9.12&55.2&0.18& 0.2064 & 1.0 &11.1& 289.4 &0.30& 0.2192 & 2.2&5.07 &50.6
        &0.3&0.2046 & 1.6 &5.45& 173.2 &0.15& 0.1979 & 1.0 & 9.48&779.5 & 0.3& 0.2102 & 2.6 &3.51&139.0\\
			\hline
			
   &0.15 &0.3&  0.2152 &3.6&10.2 &179.3 &0.18& 0.2155 & 3.0 &6.83&350.6 &0.30& 0.2169 & 4.2&3.06 & 100.9  
   &0.35&  0.2083 & 2.6 & 3.77 & 268.9&0.2& 0.2032 & 1.2& 8.04 & 738.6&0.3& 0.2083 & 3.4 &2.48&358.2\\
V  &0.25 &0.3&  0.2042 &3.0&3.92 &92.9 &0.18& 0.2022 & 2.6 &4.93&352.4  &0.30& 0.2054 & 3.0&2.76 & 103.2  
   &0.35&  0.1970 & 2.2 & 2.89 & 210.8 &0.2& 0.1954 &1.0&7.19 & 708.2&0.3& 0.1959 & 3.0 &1.81&255.4 \\
			
			\hline
	\end{tabular*}
\end{sidewaystable}

 Next we test the jester joke dataset from \url{http://www.ieor.berkeley.edu/~goldberg/jester-data/}. For this dataset, we randomly select $n_1$ users' ratings from $M^0$, and then randomly permute the ratings from the users to generate $M^*\!\in\mathbb{R}^{{n_1}\times n_2}$ with $n_2=100$. We generate a set of observed indices $\Omega$ with the sampling scheme in \eqref{sampling-scheme}, and then the observation matrix $M_{\Omega}$ via \eqref{observe}. Since we can only observe those entries $M_{jk}$ with $M_{jk}$ available and $(j,k)\in\Omega$, the actual sampling ratio is less than the input SR. We consider different settings of $n_1$ and SR, and report the average NMAE, rank, SPR and running time (in seconds) obtained by running $5$ times for each setting in Table \ref{tabJester3}. Consider that the subGM yields the worst NMAEs among the three solvers, and requires the comparable running time with the iLPA, we do not report its result for this dataset and thee movie-1M later. From Table \ref{tabJester3}, for all the jester-3 examples, the iLPA and the PAM yield the comparable NMAEs and require almost the same running time. The ranks returned by the iLPA is a little higher than those given by the PAM, but the sparsity ratios returned by the former are generally higher than those yielded by the latter.
\setlength{\tabcolsep}{1mm}
\begin{table}[h]
\setlength{\belowcaptionskip}{-0.01cm}
	\setlength\tabcolsep{1.4pt}
	\renewcommand\arraystretch{1.3}
	\centering
	\tiny
	\caption{\small Average NMAE and running time of iLPA and PAM for the jester-3 dataset}\label{tabJester3}
	
          \begin{tabular*}{\textwidth}{@{\extracolsep{\fill}}cc|lcccc|lcccc||lcccc|lcccc@{\extracolsep{\fill}}}
			\hline
			& & \multicolumn{5}{l}{\quad  iLPA ($n_{1}\!=\!1000$)}&\multicolumn{5}{l||}{\quad PAM ($n_1\!=\!1000$)}& \multicolumn{5}{l}{\quad iLPA ($n_{1}\!=\!5000$)}&\multicolumn{5}{l}{\quad PAM ($n_{1}\!\!=\!5000$)}\\
			\hline
			$\varpi$ & SR\ &\ $c_{\lambda}$& NMAE & rank&SPR  &time&$c_{\lambda}$&  NMAE & rank&SPR & time&$c_{\lambda}$& NMAE & rank&SPR  &time&$c_{\lambda}$& NMAE & rank &SPR& time\\
			\hline	
	&0.15 &0.2& 0.2259 & 2.8 & 20.2&0.59&0.2& {\bf 0.2255} & 2.6 &23.3& 0.52 &0.2& {\bf 0.2229} &5 &40.8& 3.15 &0.2& 0.2250 & 2.8 &31.0&4.05\\
I	&0.25 &0.2& {\bf 0.2173} & 4.4 & 28.1&0.28 &0.2& 0.2181 & 2.6 &14.9& 0.51 &0.2&  0.2175 &6.4 &42.5& 4.56 &0.2&  0.2175 & 2.8 &20.3&3.77\\
   \hline
     &0.15 &0.2& 0.2171 & 3.2 &20.7 &0.24 &0.2&{\bf 0.2170 }&4 &17.0& 0.49 &0.2&{\bf 0.2154} & 5.6& 43.8 &5.06&0.2& 0.2158 & 5.4 & 32.7& 4.31\\
II   &0.25 &0.2& {\bf 0.2062} & 6 &29.1 &0.40 &0.2&0.2068 &4.8 &22.8& 0.73 &0.2& {\bf 0.2086} & 7.8& 41.4 & 6.01 &0.2& 0.2095 & 6.4 &30.5& 4.89\\	
			\hline
	  &0.15 &0.03& 0.2273 & 8.2 & 61.9&0.33&0.03& {\bf 0.2248} & 9 &52.0& 0.31 &0.03& 0.2195 & 9.2 &66.9& 2.88&0.03& 0.2195 & 8.4 & 64.7 & 2.41\\
 III    &0.25 &0.03& 0.2254 & 8.4 & 53.5&0.54&0.03& {\bf 0.2228} & 9.4 &49.4& 0.23 &0.03& 0.2077 & 6 &56.8& 2.72&0.03& {\bf 0.2072} & 2.8 & 42.7 & 1.96\\	
			\hline
			
   &0.15 &0.2& 0.2182& 3.4 &25.3& 0.25 &0.2&  0.2182&3.8 &11.8& 0.43 &0.2& {\bf 0.2164}& 5.4 & 45.4 & 10.1 &0.2&  0.2168 &4.8 &33.2&5.04\\
IV	&0.25 &0.2& {\bf 0.2080}& 6 &27.6&0.23 &0.2& 0.2082& 4.6 &20.9& 0.43 &0.2& {\bf 0.2102}& 7.6 & 31.3 & 4.18 &0.2&0.2106 &6.4 &23.5 &5.15\\	
			\hline
			
	&0.15 &0.2& 0.2166 & 3.8 &26.1 & 0.24 &0.2&{\bf 0.2165} & 4.2 & 13.7& 0.48&0.2&{\bf 0.2149}& 5.6 &44.1&5.47&0.2& 0.2155 &5.4&30.3& 4.64\\
V	&0.25 &0.2& {\bf 0.2057}& 5.8 & 23.1 & 0.24 &0.2&{\bf0.2065 }& 4.8 & 17.2& 0.44 &0.2&{\bf 0.2077}&8.2 &34.4&5.70&0.2& 0.2091 &6.4&23.8& 6.29\\	
			
			\hline
	\end{tabular*}
\end{table}

To test the movie-1M dataset contained in the movieLens dataset,  we first randomly select $n_1$ users and their $n_2$ column ratings from $M^0$ to formulate $M^*\in\mathbb{R}^{n_1\times n_2}$, sample the observed entries, and then obtain the observation matrix $M_{\Omega}$ via \eqref{observe}. We consider different setting of $n_1=n_2$ and SR. Table \ref{tabMovie-1M} reports the average NMAE, rank, SPR and running time (in seconds) obtained by running $5$ times for each setting. As shown by Figure \ref{fig4}, for this dataset, the iLPA and the PAM yield the desirable NMAEs when the parameter $\lambda$ is such that the target rank equals $1$. Inspired by this and the fairness of comparisons, we report their results for solving \eqref{SCAD-loss} with their respective $c_{\lambda}$ such that the target rank is equal to $1$. We see that for most of test examples, the NMAEs returned by the iLPA are a little higher than those returned by the PAM, but the running time of the iLPA is much less than that of the PAM for all examples. This also matches the performance of the two solvers demonstrated in Figure \ref{fig4}.  
    
\setlength{\tabcolsep}{1mm}
\begin{table}[h]
	\setlength{\belowcaptionskip}{-0.01cm}
	\setlength\tabcolsep{1.2pt}
	\renewcommand\arraystretch{1.3}
	\centering
	\tiny
	\caption{\small Average NMAE and running time of iLPA and PAM for movie-1M dataset}\label{tabMovie-1M}
	
         \begin{tabular*}{\textwidth}{@{\extracolsep{\fill}}cc|lcccc|lcccc||lcccc|lcccc@{\extracolsep{\fill}}}
			\hline
			& & \multicolumn{5}{l}{\quad \quad  iLPA($n_{1}\!=\!3000$)}&\multicolumn{5}{l||}{\ \quad \quad PAM($n_1\!=\!3000$)}& \multicolumn{5}{l}{\quad \quad iLPA($6040\times 3706$)}&
			\multicolumn{5}{l}{\quad \quad PAM($6040\times 3706$)}\\
			\hline
			$\varpi$ & SR\ &\ $c_{\lambda}$& NMAE & rank  &SPR&time&$c_{\lambda}$&  NMAE & rank &SPR& time& $c_{\lambda}$& NMAE & rank  &SPR&time&$c_{\lambda}$& NMAE & rank &SPR& time\\
			\hline	
			
   &0.15&0.1&0.2170 & 1 & 9.52& 11.4 &0.07& 0.2136&1& 11.3 &51.3 & 0.1&0.2101 & 1.2 & 7.48 &38.9&0.07& 0.2064 & 1 & 4.38& 59.6\\
I  &0.25&0.1&0.2058 & 1 & 8.79& 10.7 &0.07& 0.2034&1 & 4.19 & 19.3& 0.1& 0.2007 & 1 & 7.76 & 33.1 &0.07& 0.1991 & 1 & 3.59&47.0\\
			\hline			
			
 &0.15 &0.35&  0.2130&1 & 8.68 &12.0 &0.2& 0.2069 &1& 10.4 &63.8 &0.35&0.2063 &1& 7.24 & 33.3 &0.2& 0.2006 & 1 & 9.57& 171.1\\
II&0.25 &0.35& 0.2024 &1 & 7.81 & 12.6 &0.2& 0.1979 &1& 9.55 & 66.2 &0.35& 0.1976 &1.2 & 5.89 & 31.6 &0.2& 0.1937 & 1 & 7.36& 183.5 \\
			
			\hline
			
    & 0.15&0.02& 0.2167 &1 & 48.8& 11.3 &0.01& 0.2073 & 1 & 51.1& 20.1 &  0.03& 0.2077 &1 & 56.6 & 37.7 & 0.02 & 0.2018 &1 & 57.8 & 134.6\\
III	& 0.25&0.02& 0.2007 &1 & 50.4& 11.9 &0.01&  0.1962& 1 & 51.7 &18.2&  0.03& 0.1945 & 1.2 & 56.6 & 40.5 & 0.02 & 0.1922 &1 & 55.0 & 144.9\\
			
				\hline
	&0.15&0.3& 0.2166& 1 & 7.24&13.0& 0.2& 0.2109 &1 & 8.87 & 67.1 &0.3&  0.2094 & 1 & 6.49 & 43.6&0.2& 0.2038 & 1 & 8.55& 219.9\\
IV 	&0.25&0.3& 0.2055& 1 & 7.00 & 13.0 & 0.2& 0.2010 &1 & 6.73 & 55.9 &0.3&  0.2004 & 1 & 6.29 & 35.9 &0.2& 0.1962 & 1 & 7.39 &167.7\\			
   \hline
			
   &0.15 &0.4& 0.2116 & 1 & 5.22& 13.9&0.2& 0.2049 & 1 & 7.21 & 77.8 & 0.4& 0.2052 & 1 & 4.92&37.8 &0.2& 0.1988 & 1 & 7.37&203.6\\
V	&0.25 &0.4& 0.2015 & 1 & 5.19& 13.0&0.2& 0.1964 & 1 & 7.53 & 82.0& 0.4 & 0.1966 & 1.2 & 3.98 & 39.4 &0.2& 0.1924 & 1 & 5.78& 184.3\\				
    \hline
 \end{tabular*}
\end{table}


\section{Conclusions}\label{sec7}

 We proposed an inexact LPA for solving the DC composite optimization problem \eqref{prob}, and established the full convergence of its iterate sequence under Assumptions \ref{ass0}-\ref{ass2} and the KL property of the potential function $\Xi$. If $\Xi$ satisfies the KL property of exponent $p=1/2$, the convergence has the R-linear rate. We provided a verifiable condition for the KL property of $\Xi$ with exponent $p\in[1/2,1)$ by leveraging such a property for the almost separable function $f$ defined in \eqref{ffun} and the condition \eqref{key-cond}, which is demonstrated to be weaker than the one obtained in \cite[Theorem 3.2]{LiPong18} for identifying the KL property of exponent $p\in[0,1)$ for a general composite function, and also discussed its relation with the regularity or quasi-regularity conditions used in \cite{HuYang16} for the case $\vartheta_2\equiv 0$ and $h\equiv 0$. For the iLPA armed with dPPASN for solving subproblems, numerical comparison with the nmBDCA \cite{Ferreria21} on some common DC program examples indicates that it more possibly seeks better solutions, while numerical comparisons with the PAM and the subGM for matrix completion with outliers under non-uniform sampling show that it yields the better relative errors and the comparable NMAEs within much less running time than the PAM, and the better relative errors and NMAEs within the comparable running time with the subGM.

 The proposed iLPA is also adequate for structured nonconvex and nonsmooth problems from image reconstruction. An inexact proximal MM algorithm along this line was recently proposed in  \cite{LiPanZeng25} for solving the nonconvex and nonsmooth composite models from linear image restoration problems for deblurring and inpainting, and a class of nonlinear image reconstruction for Fourier
phase retrieval.


\bigskip
\noindent
{\large\bf Appendix A.} In this part, we prove that the set $\mathbb{S}_+$ is semialgebraic. Let $\mathcal{L}(\mathbb{X})$ represent the set of all linear mappings from $\mathbb{X}$ to itself.	For any $\mathcal{Q}\in\mathcal{L}(\mathbb{X})$, its smallest eigenvalue function $\lambda_{\min}$ is defined by
\[
 \lambda_{\min}(\mathcal{Q})=\min_{\|v\|=1,v\in\mathbb{X}}\langle \mathcal{Q}v,v\rangle.
\]
Then $\lambda_{\min}$ is a semialgebraic function since the set $\{v\in\mathbb{X}\mid \|v\|=1\}$ is semialgebraic. Note that the level set of a semialgebraic function is semialgebraic, so $\mathbb{S}_{--}\!:=\{\mathcal{Q}\in \mathcal{L}(X)\mid \lambda_{\min}(\mathcal{Q})<0\}$ is a semialgebraic set. Thus, $\mathbb{S}_+$ is also a semialgebraic set since it is the complementary set of $\mathbb{S}_{--}$ in $\mathcal{L}(\mathbb{X})$.

\bigskip
\noindent
{\large\bf Appendix B.} This part includes the test examples used in Section \ref{sec6.2}, where $\phi(y)\!:=\max_{1\le i\le m}y_i$.

\begin{aexample}\label{examA.1} $\mathbb{X}=\mathbb{R}^2,\mathbb{Y}=\mathbb{R}^{3}\times\mathbb{R},\mathbb{Z}=\mathbb{R}^{3}, \vartheta_1(y,t)=\phi(y)+t, \vartheta_2(z)=\phi(z)$ and
 \begin{align*}
		F(x)=(f_1^1(x);f_1^2(x);f_1^3(x);f_1^1(x)+f_2^2(x)+f_2^3(x)),\\
		G(x)=(f_2^1(x)\!+f_2^2(x);f_2^2(x)\!+\!f_2^3(x);f_2^1(x)\!+\!f_2^3(x))
	\end{align*}
	where $f_1^1(x)=x_1^4+x_2^2,f_1^2(x)=(2-x_1)^2+(2-x_2)^2,f_1^3(x)=2e^{-x_1+x_2},
	f_2^1(x)=x_1^2-2x_1+x_2^2-4x_2+4$, $f_2^2(x)=2x_1^2-5x_1+x_2^2-2x_2+4$ and $f_2^3(x)\!=x_1^2\!+2x_2^2\!-4x_2\!+1$.
\end{aexample}
\begin{aexample}\label{examA.2}
	$\mathbb{X}=\mathbb{R}^4,\mathbb{Y}=\mathbb{R}^5\times \mathbb{R}^3\times \mathbb{R}^3,\mathbb{Z}=\mathbb{R}^{3}\times\mathbb{R}$, $\vartheta_1(y_1;y_2;y_3):=\|y_1\|_1+\phi(y_2)+\phi(y_3)$, $\vartheta_2(z,t):=\|z\|_1+t$, $F(x):=(F_1(x); F_2(x); F_3(x))$ with
	\begin{align*}
		F_1(x):=(x_1-1;x_3-1;10.1(x_2-1);10.1(x_4-1);4.95(x_2+x_4-2)),\\
		F_2(x):=200(0; x_1-x_2; -x_1-x_2),\,F_3(x):=180(0;x_3-x_4;-x_3-x_4),
	\end{align*}
	and $G(x):=(100x_1;90x_3;4.95(x_2-x_4);-100x_2-90x_4)$.
\end{aexample}
\begin{aexample}\label{examA.3}
	$\mathbb{X}=\mathbb{R}^2,\mathbb{Y}=\mathbb{R}\times \mathbb{R}^3\times \mathbb{R},\mathbb{Z}=\mathbb{R},\vartheta_1(y_1;y_2;y_3)=|y_1|+\phi(y_2)+y_3,\vartheta_2(t)=|t|$, $F(x):=(F_1(x); F_2(x); F_3(x))$ with $F_1(x)=x_1-1,F_2(x)=200(1;x_1-x_2;-x_1-x_2)$
	and $F_3(x)=-x_1-x_2$, and $G(x):=100x_1$.
\end{aexample}
\begin{aexample}\label{examA.4}
	$\mathbb{X}=\mathbb{R}^2,\mathbb{Y}=\mathbb{R}\times\mathbb{R}^3\times \mathbb{R}^9,\mathbb{Z}=\mathbb{R}^3,\vartheta_1(y_1;y_2;y_3)=|y_1|+\phi(y_2)+\phi(y_3),\vartheta_2(z,t)=\|z\|_1+|t|$, $F(x):=(F_1(x); F_2(x); F_3(x))$ with $F_1(x)=x_1-1,F_2(x)=200(0;x_1-x_2;-x_1-x_2)$
	and $F_3(x)=10(x_1^2+x_2^2+x_2; x_1^2+x_2^2-x_2; x_1+x_1^2+x_2^2+x_2-0.5;x_1+x_1^2+x_2^2-x_2-0.5;
	x_1\!-1;-\!x_1\!+2x_2\!-1;x_1\!-2x_2\!-1;-x_1\!-1;x_1\!+x_1^2\!+x_2^2)$, and $G(x):=(100x_1;10x_2;-100x_2+10(x_1^2+x_2^2))$.
\end{aexample}
\begin{aexample}\label{examA.5}
	 $\mathbb{X}=\mathbb{R}^3,\mathbb{Y}=\mathbb{R}^3\times\mathbb{R}^4,\mathbb{Z}=\mathbb{R}^4,\vartheta_1(y_1;y_2)=\|y_1\|_1+\phi(y_2),\vartheta_2(z,t)=\|z\|_1+|t|$, $F(x):=(F_1(x); F_2(x))$ with $F_1(x)=2x,F_2(x)=10(0;x_1\!+x_2\!+2x_3\!-3;-x_1;-x_2,-x_3)$, and $G(x):=(x_1-x_2;x_1-x_2;10x_2;9\!-\!8x_1\!-\!6x_2\!-\!4x_3\!+\!4x_1^2\!+\!2x_2^2\!+2x_3^2)$.
\end{aexample}
\begin{aexample}\label{examA.6}
$\mathbb{X}=\mathbb{Y}=\mathbb{R}^2,\mathbb{Z}=\mathbb{R},\vartheta_1(y)=\|y\|_1,\vartheta_2(t)=t, F(x):=x$ and $G(x):=-\frac{5}{2}x_1\!+\frac{3}{2}(x_1^2+x_2^2)$.	
\end{aexample} 

 \bigskip
 \noindent
 {\large\bf Appendix C.} We first describe the iteration steps of the Polyak subgradient method (see \cite{Charisopoulos21,LiZhu20}).

\begin{algorithm}[H]
 \renewcommand{\thealgorithm}{2} 	
 \caption{\label{subGrad}{\bf(subGM for solving model \eqref{SCAD-loss})}}
 \textbf{1: }Initialization: Choose an initial point $x^0\in\mathbb{X}$.\\
 \textbf{2: For}   $k=0,1,2,\cdots$  \textbf{do} \\ 
  \textbf{3: } \textbf{  }\textbf{  }
    Choose a subgradient $\zeta^k\in\partial\Phi(x^k)$.\\	
  \textbf{4:  }\label{step4} \textbf{   }\textbf{   }\textbf{  }\ Set $x^{k+1}=x^k-\frac{\Phi(x^k)-\min_{x\in\mathbb{X}}\Phi(x)}{\|\zeta^k\|^2}\zeta^k$.\\
\textbf{5: EndFor}.
\end{algorithm}

Next we introduce a PAM method for solving the problem \eqref{SCAD-loss}. With the notation at the beginning of Section \ref{sec6.3}, $\Phi(x)=\|F(x)\|_1-\rho^{-1}\sum_{i=1}^m\theta_{a}(\rho|G_i(x)|)+h(x)$ for $x=(U,V)\in\mathbb{X}=\mathbb{R}^{n_1\times r}\times\mathbb{R}^{n_2\times r}$. From the convexity and smoothness of $\theta_a$ in \eqref{theta-a}, for any $x,x'\in\mathbb{X}$, it holds that 
 \[
  \rho^{-1}\sum_{i=1}^m\theta_a(\rho|G_i(x)|)\ge \rho^{-1}\sum_{i=1}^m\theta_a(\rho|G_i(x')|)+\sum_{i=1}^m\theta_a'(\rho|G_i(x')|)(|G_i(x)|-|G_i(x')|).
 \]
 Write $w_{\rho}(x):=[\theta_a'(\rho|F_1(x)|),\ldots,\theta_a'(\rho|F_m(x)|)]^{\top}\in\mathbb{R}^m$ for $x\in\mathbb{X}$. Then, for any $x,x'\in\mathbb{X}$,
 \[
   \Phi(x)\le\|F(x)\|_1-\langle w_{\rho}(x'),|G(x)|-|G(x')|\rangle+h(x)-\rho^{-1}\sum_{i=1}^m\theta_{a}(\rho|G_i(x')|):=\Upsilon(x,x').
 \] 
 This, together with $\Upsilon(x',x')=\Phi(x)$, means that $\Upsilon(\cdot,x')$ is a majorization of $\Phi$ at $x'$. Inspired by this, we present the following proximal alternating minimization (PAM) method for solving problem \eqref{SCAD-loss}.
\begin{algorithm}[h]
\renewcommand{\thealgorithm}{B}
 \caption{\label{PAM}{\bf\,(PAM method for solving problem \eqref{SCAD-loss})}}
 \textbf{1:} {Initialization:} Choose an initial point $x^0=(U^0,V^0)\in\mathbb{X}$.\\
 \textbf{2: For} $k=0,1,2,\ldots$ \textbf{do}	\\
\textbf{3:       }\textbf{    } Let $u^k\!=e-w_{\rho}(U^k,V^k)$. Compute the strongly convex minimization problem
		\begin{align*}
		U^{k+1}=\mathop{\arg\min}_{\!U\in\mathbb{R}^{n_1\times r}}\|u^k\circ F(U,V^k)\|_1&+\lambda\|U\|_{2,1}+\frac{\gamma_{1,k}}{2}\|U-{U}^k\|_F^2\\
   &+\frac{\alpha_{1,k}}{2}\|F(U,V^k)-F(U^k,V^k)\|_F^2.
		\end{align*}
   \textbf{4:       }\textbf{    } Let $v^{k}\!=e-w_{\rho}(U^{k+1},V^k)$. Compute the  strongly convex minimization problem
		\begin{align*}
  V^{k+1}=\mathop{\arg\min}_{V\in\mathbb{R}^{n_2\times r}}\|v^k\circ F(U^{k+1},V)\|_1&+\lambda\|V\|_{2,1}+\frac{\gamma_{2,k}}{2}\|V-{V}^k\|_F^2\\
   &+\frac{\alpha_{2,k}}{2}\|F(U^{k+1},V)-F(U^{k+1},V^k)\|_F^2.
		\end{align*}
 \textbf{5: end (For)}	
\end{algorithm}

 \begin{remark}\label{remark-BCD}
 {\bf(a)} Algorithm \ref{PAM} has the same iterations as the BSUM algorithm described in \cite[Section 4]{Razaviyayn13}, and its subsequence convergence was proved in \cite[Theorem 2]{Razaviyayn13} under the regularity of $\Phi$ on a level set of $\Phi$. To the best of our knowledge, when it is applied to solve composite optimization with a nonsmooth loss, its iterate sequence lacks a full convergence certificate via the KL-based analysis. The reason is that the relative error condition (like Proposition \ref{prop3-xk}) for the iterate to be a critical point of the objective function (or a potential function) requires the Aubin property (a Lipschitz-like property) of its subdifferential mapping, which is impossible to hold even for the simple $\ell_1$-norm loss. A full convergence was achieved in \cite{Hien23} for this case under the KL property of $\Phi$ and an additional assumption on the subdifferential of $\Phi$, but this assumption is very restricted and almost does not hold for the loss $\vartheta(\mathcal{A}(UV^{\top})-b)$ in model \eqref{SCAD-loss}. In addition, it is unclear whether the convergence results in \cite{Razaviyayn13,Hien23} are adapted to the inexact computation of subproblems involved in the PAM.  

 \noindent
 {\bf(b)} For the numerical tests in Sections, we apply the dPPASN in Section \ref{sec5} to seek an inexact solution of subproblems. The inexactness of $U^k$ means that $\max\{R_{1,k},R_{2,k},R_{3,k}\}\le\epsilon_k$, where 
\begin{subnumcases}{}
 R_{1,k}:=\frac{\|{\rm prox}_{f_{k-1}}(\xi^{k}+z^{k}-\alpha_{1,k-1}(z^{k}-z^{k-1}))-z^{k}\|}{1+\|b\|}, \nonumber\\
 R_{2,k}:=\frac{\|{\rm prox}_{\lambda\|\cdot\|_{2,1}}\big(U^{k}-\mathcal{A}^*(\xi^{k})V^{k-1}\!-\!\gamma_{1,k-1}(U^{k}\!-\!{U}^{k-1})\big)-U^{k}\|_F}{1+\|b\|},\nonumber\\
 R_{3,k}:=\frac{\|\mathcal{A}(U^{k}(V^{k-1})^{\top})-b-z^{k}\|}{1+\|b\|}
 \ \ {\rm with}\ z^{k}=\mathcal{A}(U^{k}(V^{k-1})^{\top})-b.
\end{subnumcases}
 In other words, the relative KKT residual at $(U^k,z^k)$ attains a certain accuracy instead of $0$. A similar inexactness criterion is used to seek $V^k$ by solving the subproblem with respect to $V$.   
 \end{remark}

\end{document}